\documentclass[10pt,oneside,reqno]{amsart}%%%%%
\usepackage{lunarized-ng}
\usepackage{amssymb, amsfonts, amsmath, amsthm, amsbsy}
  \ifx\luatexversion\undefined
  \ifx\XeTeXversion\undefined
\usepackage[english]{babel}
\else
\usepackage{fontspec}
\usepackage{polyglossia}
\setmainlanguage{english}
\usepackage{microtype}
\fi
\else
\usepackage{fontspec}
\usepackage[english]{babel}
\usepackage{lualatex-math}
\usepackage[math-style=ISO]{unicode-math}
\fi

\usepackage{hyperref}
\usepackage{todonotes, graphicx, verbatim, enumerate}
\usepackage{pagecolor}

\newcommand{\reals}{\mathbb{R}}

\newcommand{\nin}{\not\in}

\newcommand{\eps}{\varepsilon}

\newcommand{\deh}{\textup{d}}

\newcommand{\st}{\ \text{s.t.}\ }
\newcommand{\eg}{\text{e.g.}\ }
\newcommand{\ie}{\text{i.e.}\ }

\newcommand{\resp}{\text{resp.}\ }

\newcommand{\const}{C_\#}

\newcommand{\diam}{\textup{diam}\,}

\newcommand{\inv}{^{-1}}

\numberwithin{equation}{section}

\usepackage[notocbasic]{nomencl}
\makenomenclature

\newcommand{\separator}{}
\renewcommand{\separator}{\\[-20pt]{\center··{\color{lunarized-red}––}··\\}}
\usepackage{mathrsfs}

\usepackage[inline]{enumitem}
\setlist[enumerate,1]{label=\textup{(\alph*)}}
\setlist[enumerate,2]{label=\textup{\roman*.}}
\makeatletter
\providecommand\@dotsep{5}
\makeatother
\newcommand\cA{{\mathcal A}}

\newcommand\cC{{\mathcal C}}
\newcommand\cD{{\mathcal D}}
\newcommand\cE{{\mathcal E}}
\newcommand\cF{{\mathcal F}}
\newcommand\cG{{\mathcal G}}
\newcommand\cH{{\mathcal H}}

\newcommand\cO{{\mathcal O}}

\newcommand\cM{{\mathcal M}}

\newcommand\cW{{\mathcal W}}

\newcommand\bG{{\mathbb G}}

\newcommand\bP{{\mathbb P}}

\newcommand\bR{{\mathbb R}}
\newcommand\bS{{\mathbb S}}
\newcommand\bT{{\mathbb T}}

\newcommand\bZ{{\mathbb Z}}

\newcommand\fL{{\mathfrak{L}}}

\newcommand\ve{\varepsilon}
\newcommand\vei{\ve^{-1}}
\newcommand{\invr}{^{-1}}
\newcommand\trim{\varsigma}
\newcommand\vf{\varphi}

\newcommand\nc[1]{_{\cC^{#1}}}

\newcommand{\supp}{\textup{supp}\,}

\providecommand{\R}{\bR}
\providecommand{\T}{\bT}

\newcommand{\unstableConeConst}{\chi^{u}}
\newcommand{\centreConeConst}{\chi^{c}}
\newcommand{\fakeCentreConst}{\varrho}
\newcommand{\CentreExpansion}{\upsilon}

\newcommand{\Const}{C_{\#}}
\newcommand{\plaque}{\mathbb{K}}
\newcommand{\bplaque}{\bar{\plaque}}
\newcommand{\wplaque}{\widetilde{\plaque}}
\newcommand{\plaqueSet}{\mathfrak{K}}
\newcommand{\pplaqueSet}{\bar{\mathfrak{K}}}

\newcommand{\scc}[1]{\gamma^{(#1)}} % Standard curve constants
\newcommand{\pscc}[1]{\bar{\gamma}^{(#1)}}
\newcommand{\Stdc}{\Sigma}
\newcommand{\pStdc}{\bar\Sigma}
\newcommand{\Dconst}{\mathfrak{D}}
\newcommand{\Stdd}{\mathcal D}
\newcommand{\Stdp}{\fL}
\newcommand{\stdf}{\mathcal{L}}
\newcommand{\pStdp}{\bar{\Stdp}}
\newcommand{\stdr}{K}

\newcommand{\fstdr}{\mathcal\stdr}
\newcommand{\Wc}{{\cW}^{\textup{c}}}
\newcommand{\Econst}{{\mathfrak{E}}}
\newcommand{\Rough}{{R}}
\newcommand{\fmm}{\nu}

\newcommand{\narr}{z}
\newcommand{\shor}{Z}
\newcommand{\minshor}{\underline{Z}}
\newcommand{\rough}{r}
\newcommand{\stdrough}{r_{*}}
\newcommand{\bandh}{h}
\newcommand{\patchheight}{\Delta}
\newcommand{\centreMaxExp}{\Lambda_\textrm{c}}

\newcommand{\height}{\textup{height}\,}
\newcommand{\foli}{\eta}
\newcommand{\folis}{\eta^{*}}

\newcommand{\chvar}{\Phi_\ve}
\newcommand{\regular}{regular}
\newcommand{\clock}{\clocka_{0}}
\newcommand{\clocka}{T}
\newcommand{\clockN}{\clockaN_{0}}
\newcommand{\clockaN}{N}
\newcommand{\logstep}{V}

\newcommand{\lyapFnExp}{\gamma}
\newcommand{\lyapFnDrift}{B}
\newcommand{\Leb}{\operatorname{Leb}}

\newcommand{\lowerBoundWeight}{\mathfrak c}

\newcommand{\Co}{\textrm{C}}
\newcommand{\TV}{\textrm{TV}}

\DeclareFontFamily{U}{wncy}{}
\DeclareFontShape{U}{wncy}{m}{n}{<->wncyr10}{}
\DeclareSymbolFont{mcy}{U}{wncy}{m}{n}
\DeclareMathSymbol{\El}{\mathord}{mcy}{"4C}
\DeclareMathSymbol{\el}{\mathord}{mcy}{"6C}

\DeclareMathAlphabet{\mathantt}{OT1}{antt}{li}{it}
\DeclareMathAlphabet{\mathpzc}{OT1}{pzc}{m}{it}

\renewcommand{\st}{:}

\author{Jacopo De Simoi}
\address{Jacopo De Simoi\\
  Department of Mathematics\\
  University of Toronto\\
  40 St George St. Toronto, ON, Canada M5S 2E4}
\email{{\tt jacopods@math.utoronto.ca}}
\urladdr{\href{http://www.math.utoronto.ca/jacopods}{http://www.math.utoronto.ca/jacopods}}

\author{Kasun Fernando}
\address{Kasun Fernando\\
  Department of Mathematics\\
   Brunel University London\\
    Uxbridge UB8 3PH, UK}
\email{{\tt kasun.fernando@brunel.ac.uk}}

\author{Nicholas Fleming-V{\'a}zquez}
\address{Nicholas Fleming-V{\'a}zquez\\
	Department of Mathematics\\
	University of Toronto\\
	40 St George St. Toronto, ON, Canada M5S 2E4}
\email{{\tt nicholas.fleming@utoronto.ca}}
\title[Mostly Expanding Fast-Slow Systems]{Statistical Properties of
  Mostly Expanding\\ Fast-Slow Partially Hyperbolic Systems}

\hypersetup{%
  pdftitle={Statistical properties of mostly expanding Fast-Slow
    Partially Hyperbolic systems},%
  pdfauthor={Jacopo De Simoi, Kasun Fernando, Nicholas
    Fleming-Vázquez}}
\newcommand{\ignore}[1]{}

\begin{document}
\begin{abstract} {\ifx\XeTeXversion\undefined \relax \else \normalsize
    \small %% XeLaTeX bug…
  \fi We consider a class of fast-slow $\cC^{4}$ partially hyperbolic
  systems on $\bT^2$ given by $\ve$-perturbations of maps
  $F(x,\theta)=(f(x,\theta),\theta)$ where $f(\cdot,\theta)$ are
  $\cC^{4}$ expanding maps of the circle. For sufficiently small $\ve$
  and an open set of perturbations we prove existence and uniqueness
  of a physical measure and exponential decay of correlations for
  sufficiently smooth observables with explicit almost optimal bounds
  on the decay rate.  Our result complements previous work by De
  Simoi–Liverani, which studies the case of mostly contracting centre.}
\end{abstract}
\maketitle
% ** Introduction

\section{Introduction and statement of our result}
Fast-slow (or, more generally, multi-scale) systems appear naturally
in many physical contexts and applications, and constitute an
incredibly diverse and abundant class of deterministic dynamical
systems which appear to exhibit a rich stochastic behaviour.  On the
other hand, the coupling between different timescales constitutes a
formidable difficulty in understanding such systems from the point of
view of their long-term dynamics.  As a consequence, their fine
stochastic properties are –in general– quite difficult to establish.

It is also known that if one wishes to obtain good statistical
properties of a dynamical system, it is desirable that the system
presents some degree of hyperbolicity.  In the context of fast-slow
systems, the natural assumption to make is partial hyperbolicity: the
fast component of the dynamics takes place along the stable / unstable
directions, whereas the slow component develops along the centre
directions.

Partially hyperbolic systems have a long history of results concerning
their geometric properties and stable ergodicity, starting with
~\cite{MPS94,pugh-shub97}, but their stronger statistical properties
have only been studied in more recent years, first for group
extensions of Anosov maps and flows (see
e.g.~\cite{Che98,Dima98,Dima-group,Liverani04,tsujii_zhang}), maps
whose centre direction is mostly contracting or mostly expanding (see
e.g.~\cite{Dima-mostly-contracting,buzzi2025strongpositiverecurrenceexponential,alves-pinheiro-mostly-expanding-poly-doc,alves-xin-mostly-expanding-exp-doc})
and –even more recently– specifically for fast-slow systems (see
e.g.~\cite{MR3556527,castorrini-liverani,bonetto2025synchronizationaveragingpartiallyhyperbolic}).
Fast-slow systems are, on the other hand, usually studied within the
framework known as \emph{averaging theory} (see
e.g.~\cite{DimaAveraging,Dima12,Dimaliverani,MR3842060,canestrari2023heatequationdeterministicdynamics}).
Averaging is a powerful tool that can be used to obtain limit theorems
at fixed time scales, but in order to obtain information about the
asymptotics of our dynamics (typically encoded in physical measures
and their statistical properties) one needs to obtain information at
arbitrary long time scales.

This paper walks along the route traced by Liverani--De Simoi in~\cite{MR3556527}, in which
they embark in this endeavour by studying a particularly simple
(but far from trivial) situation (see~\eqref{eq:map} below) and obtain
fine statistical properties of such systems.  More
precisely,~\cite{MR3556527} shows existence of finitely many physical
measures for an open class of partially hyperbolic slow-fast local
diffeomorphisms, and proves decay of correlations for sufficiently
smooth observables with relatively sharp bounds.  The main assumptions
on the class that are necessary in the argument of~\cite{MR3556527}
are related to some a-priori control on the centre Lyapunov exponent
of the system.  The argument in fact only works for so-called
\emph{mostly contracting} systems (see
Definition~\ref{def:mostly-contracting-expanding}).

This paper serves as a natural companion of~\cite{MR3556527}: we
obtain results about existence (and uniqueness) of physical measures,
decay of correlations with explicit bounds on the rates under the
complementary \emph{mostly expanding} assumption (again see
Definition~\ref{def:mostly-contracting-expanding}).

% ** Definitions
\subsection{A first version of our main result}
Let $\bT = \bR/\bZ$; for $\ve>0$ let us consider the map
$F_\ve\in\cC^{4}(\bT^2,\bT^2)$ defined by
\begin{align}\label{eq:map}
  F_\ve(x,\theta)=F_0(x,\theta)+\ve F_1(x,\theta)\mod 1
\end{align}
where
\begin{align*}
  F_0(x,\theta)&=(f(x,\theta),\theta) &F_1(x,\theta)&=(g(x,\theta),\omega(x,\theta))
\end{align*}
are both $\cC^{4}$ maps.  In the sequel we will denote with
$\pi_{1}:\bT^{2}\to\bT$ (\resp $\pi_2:\bT^{2}\to\bT$) the projection
onto the first (\resp second) coordinate.  We assume that
$f_\theta=f(\cdot,\theta)$ is an expanding map for each
$\theta\in\bT^1$; moreover, by possibly replacing $F_\ve$ with a
suitable iterate, we will further assume that
$\partial_x f\ge\lambda>3$.  Furthermore, since we take $\ve$ to be
fixed, although small, we will assume \ignore{without loss of
  generality }that $g=0$ by incorporating $\ve g$ into $f$.  The aim
of the research project initiated in~\cite{MR3556527} is to obtain
statistical properties for generic perturbation of $F_{0}$; we
henceforth consider $F_{0}$ to be fixed once and for all.

Observe that $F_{0}$ is a local diffeomorphism (necessarily not
invertible), hence the same holds for $F_{\ve}$ if $\ve$ is
sufficiently small.  Indeed, the system above is a fast-slow system,
since the slow variable $\theta$ needs $\cO(\vei)$ iterations to
undergo a non-negligible change.  Let us briefly explain how to
implement \emph{averaging} to study our system by recalling some ideas
from~\cite{DL, DLPV}: since $f_\theta$ is a family of expanding maps of
the circle, there exists a unique family of absolutely continuous
$f_\theta$-invariant probability measures whose densities we denote by
$\rho_\theta$.  Since $F_{\ve}$ is $\cC^{4}$, it follows (see
e.g.~\cite[Section 8]{GL06}) that $\rho_\theta$ is a $\cC^3$-smooth
family of densities of class $\cC^3$. Let us now define:
\begin{align*}
\bar\omega(\theta)=\int_{\bT^1}\omega(x,\theta)\rho_\theta(x) \deh
  x;
\end{align*}
the above discussion implies that $\bar\omega\in\cC^3(\bT)$.  The
function $\bar\omega$ can be regarded as an averaged forcing for the
slow variable $\theta$.  More precisely, let
$(x_n,\theta_n)=F^n_\ve(x_0,\theta_0)$: we can regard
$(x_{n},\theta_{n})$ as random variables with respect to some
distribution $\nu_{0}$ of initial conditions $(x_{0},\theta_{0})$.  We
can moreover define the interpolation of $\theta_{n}$ as a piecewise
linear continuous function: for any $t\in\reals_{\ge 0}$ let
\begin{align}\label{eq:interpolating-definition}
  \theta_\ve(t)= \theta_{[\ve^{-1}t]}+ (\ve^{-1}t-[\ve^{-1}t])(\theta_{[\ve^{-1}t]+1}-\theta_{[\ve^{-1}t]}).
\end{align}

Let us fix $T > 0$ and $\theta_{0}\in\bT$.  It is shown
in~\cite[Theorem 2.1]{DL} that, for any fixed $T>0$, if the initial
distribution $\nu_{0}$ equals $\mu_{0}\times\delta_{\theta_{0}}$,
where $\mu_{0}$ is  an arbitrary measure on $\bT$
that is absolutely continuous with respect to Lebesgue, then
$\theta_{\ve}$, as a random element of $ \cC^0([0,T],\T)$, converges
in probability as $\ve\to 0$ to the unique solution of the ODE
\begin{equation}\label{avgODE}
  \frac{d\theta}{dt}=\bar \omega(\theta),\,\,\, \theta(0)=\theta_0.
\end{equation}
We call the solution to~\eqref{avgODE} the \emph{averaged system}.

\begin{rmk}
  For $\omega$ generic, the $\cC^{3}$-function $\bar\omega$ has an
  even number of zeros; half of them will have positive derivative
  (and identify \emph{sources} for the averaged dynamics), while the
  other half will have negative derivative (and identify \emph{sinks}
  for the averaged dynamics).

  There are three qualitatively different scenarios:
  \begin{enumerate}
  \item $\bar\omega$ has no zeros;
  \item $\bar\omega$ has exactly one pair of zeros;
  \item $\bar\omega$ has more than one pair of zeros.
  \end{enumerate}
\end{rmk}
The case in which there are no zeros is considerably more complicated
than the case in which there are zeros.  The case in which there are
several pairs of zeros is marginally more complicated, but technically
quite cumbersome to carry out.  In this paper, we choose to study only
case (b) above, but see Section~\ref{sec:conclusions} for some further
comments on the other cases.  We proceed to define the open set:
\begin{align*}
  \Omega_{1} = \{\omega\in \cC^{4}(\bT^{2},\bT) \st \textrm{$\bar\omega$ has exactly one pair of
non-degenerate zeros}\}.
\end{align*}
We are now ready to state the first version of our Main Theorem; a
more precise version will be provided later as
Theorem~\ref{thm:doc_mostly_expanding}.  The theorem below is obtained
by combining the results in~\cite{MR3556527,castorrini-liverani} with
the results obtained in this paper.
\begin{mthm}
  There exists a $\cC^{4}$-open and dense set
  $\Omega^{*}_{1}\subset\Omega_{1}$ such that if $F_{\ve}$ is as above
  with $\omega\in\Omega^{*}_{1}$, the following holds.  For all
  $\ve>0$ sufficiently small, the map $F_\ve$ admits a unique physical
  measure $\nu_\ve$.  Moreover $\nu_{\ve}$ is absolutely continuous
  with respect to Lebesgue and enjoys exponential decay of
  correlations with rate $c_\ve>0$. That is, there exists $C_1,C_2>0$
  such that for any observables $A,B\in \cC^2(\bT^2)$,
  \begin{align}\label{eq:decay-of-correlations-weaker}
    \big| \Leb(A\cdot B\circ F_\ve^n) &- \Leb(A)\nu_\ve(B)\big| \leq C_1\|A\|_{\cC^2}\|B\|_{\cC^2} \exp(-c_\ve n).
  \end{align}
  Finally, the rate of decay of correlations satisfies the following
  bound:
  \begin{equation}\label{e_lowerBoundRate-weaker}
    c_\ve \ge C_2\ve/\log\vei.
  \end{equation}
\end{mthm}
\begin{rmk}
  The bound on the rate given by~\eqref{e_lowerBoundRate-weaker} is
  nearly optimal: we expect the optimal rate to be
  $c_{\ve} = \Const\ve$.  See Section~\ref{sec:conclusions} for further
  comments about possible refinements of this bound.
\end{rmk}
\subsection{A more precise statement}
We now proceed to examine more closely the system~\eqref{eq:map}, so
that we can discuss our result in more detail and frame it in the
context of the current literature.  As briefly mentioned above, the
system is partially hyperbolic: more precisely, as will be discussed
in Section~\ref{sec:Hyper}, there exist forward-invariant (unstable)
and backward-invariant (centre) cone fields.  Vectors in the unstable
cone are expanded by the dynamics; however, since $F_{\ve}$ is not
invertible, it does not typically possess forward-invariant
directions.  In Subsection~\ref{sec:centre_direction}, we will see
that the map $F_\ve$ admits instead  a backward-invariant centre
distribution for any $\ve$ sufficiently small.  For the map $F_{0}$,
this distribution is generated by a vector field of the form
$(s_*(x,\theta),1)$.  Vectors belonging to the centre distribution are
expanded or contracted by the dynamics according to the sign of the
function $\psi_{*}$, that is the directional derivative of $\omega$ in
the centre direction:
\begin{equation}\label{psi_star_def}
	\psi_{*}(x,\theta) = \partial_x \omega(x,\theta)s_{*}(x,\theta) + \partial_\theta \omega(x,\theta).
\end{equation}

Since $\omega\in\Omega_{1}$, the function $\bar\omega(\theta)$ has
only two zeros: we denote them by $\theta_{\pm}$, so that
$\pm\bar\omega'(\theta_{\pm}) > 0$.  Note that $\theta_{-}$ is a sink
for the averaged dynamics and $\theta_{+}$ is a source.  The averaging
principle suggests that most orbits will spend the majority of the
time close to the sink.  In fact, the dynamics for long time scales is
localized (in a very precise sense) around $\theta_{-}$ (see e.g.~\cite[Proposition 2]{DLPV}).

We can now introduce a definition
\begin{mydef}\label{def:mostly-contracting-expanding}
  Let $F_{\ve}$ as in~\eqref{eq:map} with $\omega\in\Omega_{1}$, and
  let
  $\bar{\psi}_*(\theta)=\int_{\bT^1}\psi_{*}(x,\theta)\rho_\theta(x)
  \deh x$.
  The system $F_{\ve}$ is said to be
  \begin{itemize}
  \item \emph{mostly contracting} if $\bar\psi_{*}(\theta_{-}) < 0$;
  \item \emph{mostly expanding} if $\bar\psi_{*}(\theta_{-}) > 0$.
  \end{itemize}
\end{mydef}
Mostly contracting systems are characterized by the fact that, near
the sink, centre vectors are on average contracted, whereas for mostly
expanding systems, centre vectors are on average expanded.  It is
clear that for an open and dense set of $\omega\in\Omega_{1}$, the
system $F_{\ve}$ is  either mostly contracting or mostly expanding.

We now proceed to state explicitly another open and dense condition
that will allow to give a precise statement of our main result.  Let
us first recall a few useful definitions: an observable
$\phi\in\cC^0(\bT^1)$ is said to be a \emph{coboundary $($with respect
  to a map $f:\bT^{1}\to\bT^{1})$} if there exists
$\beta\in\cC^0(\bT^1)$ so that
\begin{align*}
  \phi=\beta-\beta\circ f.
\end{align*}
Two observables $\phi,\psi\in\cC^0(\bT^1)$ are said to be
\emph{cohomologous $($with respect to $f)$} if their difference
$\phi-\psi$ is a coboundary $($with respect to $f)$.
\begin{enumerate}[label=(A\arabic*),start = 0]
  \item\label{a_noNearCobo} We assume that for each
  $\theta\in\bT^1$, each $(a,b)\in\bR^{2}\setminus\{0\}$, the function
  $x\mapsto a\omega(x,\theta)+b\psi_{*}(x,\theta)$ is not cohomologous
  to a constant with respect to $f_\theta$.
\end{enumerate}
We can now define $\Omega_{1}^{*}$ to be the set of $\omega$ such
that~\ref{a_noNearCobo} holds \ignore{(the proof of this condition
  being open and dense will be found below)} and $F_{\ve}$ is either
mostly contracting or mostly expanding.

If $F_{\ve}$ is mostly contracting, then our Main Theorem is a
consequence of the main results in~\cite{MR3556527}
and~\cite{castorrini-liverani} (for what concerns the absolute
continuity of the physical measure); hence in this paper we will need
to consider only the mostly expanding case.  We are now finally ready
to write a precise statement of our main result.
\begin{thm}\label{thm:doc_mostly_expanding}
  Assume that $\omega\in\Omega_{1}^{*}$ and $F_{\ve}$ is mostly
  expanding.  Then for $\ve>0$ sufficiently small, the map $F_\ve$ admits
  a unique physical measure $\nu_\ve$.  Moreover $\nu_{\ve}$ is
  absolutely continuous with respect to Lebesgue and enjoys
  exponential decay of correlations with rate $c_\ve>0$. That is,
  there exists $C_1,C_2>0$ such that for any functions $A\in \cC^2$,
  $B \in L^\infty(\Leb)$,
  \begin{align}\label{eq:decay-of-correlations}
    \big| \Leb(A\cdot B\circ F_\ve^n) &- \Leb(A)\nu_\ve(B)\big| \leq C_1\|A\|_{\cC^2}\|B\|_{L^\infty} \exp(-c_\ve n).
  \end{align}
  Finally, the rate of decay of correlations satisfies the following
  bounds:
  \begin{equation}\label{e_lowerBoundRate}
    c_\ve \ge C_2\ve/\log\vei
  \end{equation}
 \end{thm}
  \subsection{Remarks and comments about our assumptions}
We begin by commenting on the \emph{mostly expanding condition}.  This
condition implies that the centre Lyapunov exponent with respect to
any ergodic physical measure $\nu$ is positive (see
Remark~\ref{rem:lyap-exponent}).  This is counter-intuitive, since one
would na\"ively expect that near a sink only contraction can take
place.  We refer the reader to \cite[Section 7]{DLPV} for more details
on this paradoxical behaviour, although the key observation is that
the centre foliation for a mostly expanding system such as ours is
necessarily non-absolutely continuous.
\begin{rmk}
  Observe moreover that if $F_{0}$ is a skew-product, then $ s_* = 0$
  and $\partial_\theta\rho_\theta=0$; in
  particular~\eqref{psi_star_def} implies that
  $\bar{\psi}_*(\theta)=\bar{\omega}'(\theta)$ and therefore
  $\bar{\psi}_{*}(\theta_{-}) < 0$, which is to say that the system is
  mostly contracting.  In order for $F_{\ve}$ to be mostly expanding,
  it hence  needs to be sufficiently far away from any skew product.
\end{rmk}
Observe moreover that if the system is mostly expanding, it is always
possible to normalize the system\footnote{ In fact, under the
  rescaling
\begin{align*}
  \omega \mapsto \bar \psi_* (\theta_-)^{-1}\omega \,\, \text{and}\,\, \varepsilon \mapsto \bar \psi_{*}(\theta_{-}) \varepsilon,
\end{align*}
we see that the product $\varepsilon \omega$ remains unchanged while
$\psi_* \mapsto \bar \psi_*(\theta_-)^{-1}\psi_*$.} in such
a way that
\begin{enumerate}[label=(A\arabic*),resume]
  \item\label{a_almostTrivial}
  $\bar\psi_{*}(\theta_{-}) = 1$.
\end{enumerate}
We will take~\ref{a_almostTrivial} to be a standing assumption
throughout the paper.

We now comment on strategies that can be used to check if
condition~\ref{a_noNearCobo} holds.
\begin{rmk} If~\ref{a_noNearCobo} fails for some $\theta$, then there
  exist real numbers $a_{\theta}, b_{\theta}$ and $c_\theta$ such that
  for any $f_{\theta}$-invariant measure $\mu$, the average
  $\mu(a_{\theta}\omega(\cdot,\theta)+b_{\theta}\psi_{*}(\cdot,\theta))
  = c_\theta$.  Hence, in order to check that the condition is
  satisfied, it is sufficient to find, for each $\theta$, three
  periodic orbits of $f_\theta$ with the property that the differences
  of the averages of $(\omega(\cdot,\theta),\psi_{*}(\cdot,\theta))$
  span $\bR^{2}$.  It is not difficult to check that this condition is
  open and dense.
\end{rmk}
\ignore{\begin{rmk}By~\eqref{NstepCenter}, we have $\|s_*\| \le \frac{\|\partial_\theta f\|}{\lambda - 1}$. Using this, we define
  \[
    \widetilde{\psi}(\theta) = \int_{\bT^1} \partial_\theta \omega(x, \theta) \rho_\theta(x) \deh x - \frac{\|\partial_x \omega\| \|\partial_\theta f\|}{\lambda - 1}.
  \]
  Then $\bar{\psi}_*(\theta) \geq \widetilde{\psi}(\theta)$, so checking that $\widetilde{\psi}(\theta_-)>0$ provides an elementary sufficient condition for $\bar{\psi}_*(\theta_-)$ to be positive.
\end{rmk}
}

\subsection{Remarks and comments about our result}
The paper~\cite{castorrini-liverani} shows existence of finitely many
physical measures that are absolutely continuous with respect to
Lebesgue for any system of the form~\eqref{eq:map} and generic
$\omega$; the same paper also proves that any such measure enjoys
exponential decay of correlations; however, it provides no bound on the
rate of decay of correlations, or on the number of physical measures as
$\ve\to 0$.

In order to obtain more precise results it seems necessary to impose
further assumptions, regarding for instance the number of sinks /
sources of the averaged system and the conditions on the centre
Lyapunov exponents implied by the mostly contracting / expanding
assumptions.  This is the path followed by this paper as well
as~\cite{MR3556527}.  Both papers hinge on the limit theorems
developed in~\cite{MR3842060} for fast-slow systems to obtain concrete
bounds on the rate of decay of correlations.  As mentioned in the
introductory paragraphs, the main result in~\cite{MR3556527} deals
with the mostly contracting situation, while in this paper we deal
with the mostly expanding case.  Another difference is
that~\cite{MR3556527} allows the presence of many sinks for the
averaged dynamics, although under the rather artificial assumption that
every one of them is mostly contracting.  The purpose of this paper is
to show how to deal with the mostly expanding situation, and we chose
to implement the strategy in the simpler situation of only one sink
(see Section~\ref{sec:conclusions}).

Both this paper and~\cite{MR3556527} use a coupling argument to
obtain a concrete bound on the rate of decay of correlations,
whereas~\cite{castorrini-liverani} use a more traditional (and perhaps
more elegant) approach involving the transfer operator for $F_{\ve}$.
However, we wish to underscore a that the nature of the coupling
involved in the proof of Theorem~\ref{thm:doc_mostly_expanding} and
of~\cite[Main Theorem]{MR3556527} is substantially different.

In the mostly contracting situation, one couples measures that are
supported on unstable curves (standard pairs) along the mostly
contracting centre directions.  In this paper we instead couple
measures that are supported on two-dimensional (centre-unstable)
rectangles.  Such measures (we call them \emph{standard patches}, see
Section~\ref{sec:std-pairs-patches}) constitute the main technical
novelty introduced in this paper.  We believe that such techniques can
be extended and employed successfully in a variety of different
scenarios.
\subsection{An outline of the paper}
Section~\ref{sec:Hyper} recalls the necessary notions of hyperbolicity
and show some technical results and estimate that apply to our system.
Section~\ref{sec:std-pairs-patches} first recalls the notion and
properties of standard pairs, and then introduces the class of
measures called \emph{standard patches}, which will be used to prove
our main result.  Next, in Section~\ref{sec:averaging} we recall and
adapt the results about the averaged motion obtained in~\cite[Section
7]{MR3556527} to the current setting.
Section~\ref{sec:patch-families} proves the key invariance properties
of standard patches; the most important observation is that, due to
the effectively random nature of the dynamics along the centre
direction, we will be able to obtain a notion of invariance that only
holds \emph{on average}.  Section~\ref{sec:coupling} then presents the
coupling argument that concludes the proof of
Theorem~\ref{thm:doc_mostly_expanding}.  Finally,
Section~\ref{sec:conclusions}, contains a few comments regarding
concrete future directions, and Appendix~\ref{sec:second-deriv-bounds}
presents some necessary technical results that are used in the study
of dynamical properties of standard patches.
\subsection{Some conventions used throughout the paper}%
\label{sec:conventions}%
We conclude this introductory section by listing some notational
conventions used in the paper.  We will denote with $\Const$
an arbitrary positive constant, whose value may change from one
instance to the next, even in the same expression.  It is understood
that the actual values of the constant $\Const$ might depend on $f$
and $\omega$, but would never depend on $\ve$.

Given $U\subset\bR^{n}$ and $V\subset\bR^{m}$, a (sufficiently smooth)
function $g:U\to V$ and $p\in U$, we will denote with $d_{p}g$ the
differential of $g$ at $p$ viewed as a linear functional on $\bR^{n}$;
$\|d_{p}g\|$ will denote the norm (in the operator sense) of the
functional.  Similarly, $H_{p}g$ will denote the Hessian of $g$ at the
point $p$ viewed as a bilinear functional (\ie an operator from
$\bR^{n}$ to the space of linear functionals); $\|H_{p}g\|$ thus
denotes the operator norm.  Finally we let
$\|dg\|_{\infty} = \sup_{p\in U}\|d_{p}g\|$ (and likewise for $H$).
\subsection*{Acknowledgements}
JDS and KF have been partially funded by the NSERC Discovery grant
Fast–Slow Dynamical systems 172513; JDS also acknowledges partial
support of the University of Toronto Connaught New Researcher Award.

The authors are grateful to Carlangelo Liverani and Dmitry Dolgopyat
for the many inspiring discussions on the topic.
% ** About hyperbolicity
\section{Hyperbolicity}\label{sec:Hyper}
First, we observe that the system is partially hyperbolic. Note that,
 \begin{align}\label{DerivF}
dF_\ve = \begin{pmatrix}
\partial_x f & \partial_\theta f\\
\ve \partial_x\omega & 1+\ve \partial_\theta \omega
\end{pmatrix}.
\end{align}
Observe that for $\ve = 0$ we have $\det dF_{\ve} = \partial_{x}f$,
hence $F_{0}$ is a local diffeomorphism; we will always assume
$\ve$ to be so small that $F_{\ve}$ is also a local diffeomorphism.

% *** The unstable cone
We now define an unstable cone which is invariant under $dF_\eps$ and
a centre cone which is invariant under $dF_\eps^{-1}$.  Choosing
$\unstableConeConst,\centreConeConst >0$ appropriately we can
respectively define the unstable cone and the centre cone by
\begin{align*}
  \cC_{u,\unstableConeConst}= \{(\alpha,\beta)\in \R^2: |\beta| \leq \ve \unstableConeConst |\alpha|\}, \quad\cC_{c,\centreConeConst}= \{(\alpha,\beta)\in \R^2: |\alpha| \leq \centreConeConst |\beta|\}.
\end{align*}
Notice that the centre cone and the unstable cones are everywhere
uniformly transversal.

To make the appropriate choices of $\unstableConeConst,\centreConeConst>0$ note that for all $p=(x,\theta)\in \bT^2$,
\begin{align}\label{1StepExp}
d_pF_\ve(1,\ve u) &=(\partial_x f(p)+\ve u \partial_\theta f(p), \ve \partial_x\omega(p)+\ve u +\ve^2 u \partial_\theta \omega(p)) \nonumber \\ &= \partial_xf(p)\left(1+ \ve\frac{\partial_\theta f(p)}{\partial_xf(p)}u\right)\big(1, \ve \Xi_p^+(u)\big)
\end{align}
where
\begin{align}\label{eq:def_Xi}
\Xi^{+}(p,u) : = \frac{\partial_x \omega(p)+(1+\ve \partial_\theta \omega(p))u}{\partial_xf(p)+\ve \partial_\theta f(p) u}.
\end{align}
Taking
$M=\max \{\|\partial_x\omega\|, \|\partial_\theta \omega\|,
\|\partial_\theta f\|\}$, we have that when $|u|\leq \unstableConeConst$ where
$\unstableConeConst \leq (\ve M)^{-1} $,
\begin{align*}
 |\Xi^+(p,u)| \leq \frac{M+1+\unstableConeConst}{\lambda - 1}.
\end{align*}
Hence choosing $\frac{M+1}{\lambda -2}\le \unstableConeConst \le (\eps M)\inv$ (which can be done only
when $\ve \leq \frac{\lambda -2 }{M(M+1)}$) ensures that
$d_pF_\ve(\cC_{u,\unstableConeConst}) \subset \cC_{u,\unstableConeConst}$. Consequently, the complementary cone $\cC_{u,\unstableConeConst}$ satisfies
$d_pF_\ve^{-1}(\cC_{u,\unstableConeConst}^\complement) \subset
\cC_{u,\unstableConeConst}^\complement$. Since $\cC_{c,\centreConeConst}$ is equal to the closure of $\cC_{u,(\eps \centreConeConst)\inv}^\complement$, it follows that $\cC_{c,\centreConeConst}$
is invariant under $dF_\eps\inv$ whenever $M \le \centreConeConst \le (\frac{\eps(M+1)}{\lambda -2})\inv$.
From now on we fix $\centreConeConst = M$; we will choose $\unstableConeConst$ later.

\nomenclature[0aa]{$\centreConeConst$}{Constant appearing in the definition of the centre cone}
\nomenclature[0ab]{$\unstableConeConst$}{Constant appearing in the definition of the unstable cone}

Observe that~\eqref{1StepExp} implies that $dF_{\eps}$ expands vectors in the unstable cone: for all $p \in \bT^2$ and $v\in \cC_{u,\unstableConeConst}$, we have
\begin{equation}\label{eq:unstable_cone_expansion}
	|\pi_1 d_p F_{\eps}v|\ge \partial_x f(p)(1-\Const \eps)|\pi_1 v|> 3|\pi_1 v|
\end{equation}
for $\eps$ sufficiently small.

A vector field in $\bT^{2}$ is said to be an \emph{unstable vector
  field} (\resp \emph{centre vector field}) if it lies in the unstable
cone (\resp centre cone) at each point.  A smooth curve in $\T^{2}$ is
said to be an \emph{unstable curve} (\resp \emph{centre curve}) if its
tangent vector lies in the unstable cone (\resp \emph{centre cone}) at
each point.
\subsection{The centre direction}\label{sec:centre_direction}
%%%%%%%%%%%%%%%%%%%%%%%%%%%%%%%%%%%%%%%%%%%%%%%%%%%%%%%%%%%%%%%%%%%%%%
% *** The centre cone
By the backward invariance of the centre cone, it is possible to
define an invariant centre subspace distribution.  Since centre vector
fields belong to the centre cone, and the centre cone is oriented
along the vertical direction, any centre vector field is a multiple of
a field $p\mapsto(s(p),1)$, where $|s(p)| < \centreConeConst$; we call $s(p)$ the
associated \emph{slope field}; observe that a slope field uniquely
identifies a one-dimensional subspace distribution in the centre cone.
We now show the existence of an invariant slope field; we proceed in
two steps.
%%%%%%%%%%%%%%%%%%%%%%%%%%%%%%%%%%%%%%%%%%%%%%%%%%%%%%%%%%%%%%%%%%%%%%%%
% **** The limiting map
We begin by considering $\ve = 0$ and thus the map $F_0$; notice that
by~\eqref{DerivF} we have:
\begin{align}\label{DerivF0}
  dF_0 = \begin{pmatrix}
    \partial_x f & \partial_\theta f\\
    0 & 1
\end{pmatrix}.
\end{align}
and therefore the differential $dF_{0}$ preserves the second
coordinate of any vector.  For any $(x,\theta)\in\bT^{2}$ and any
$n\ge 0$, consider the iterate $d_pF_{0}^{n}$; and define (since
$d_pF_{0}$ is invertible) the sequence of slope fields $s_{n}$ as
follows:
\begin{align*}
  (s_{n}(p),1) = [d_pF_{0}^{n}]\inv (0,1)
\end{align*}
with $|s_n| \leq \centreConeConst$. Then
\begin{align*}
  d_{F_0(p)}F^{n-1}_0 \circ d_pF_0 (s_n(p),1) &=  d_{F_0(p)}F^{n-1}_0(s_{n-1}(F_0(p)),1),
\end{align*}
which yields
\begin{align}\label{SlopeRec}
  d_pF_0 (s_n(p),1) &= (s_{n-1}(F_0(p)),1).
\end{align}
From~\eqref{DerivF0} and~\eqref{SlopeRec}, we thus obtain:
\begin{align}\label{NstepRecur}
s_n(p) = \frac{s_{n-1}(F_0(p)) - \partial_\theta f(p)}{\partial_xf(p)};
\end{align}
the latter implies
\begin{align}\label{NstepCenter}
  s_n(p) &= -\sum_{k=0}^{n-1} \frac{\partial_\theta f(F^k_0(p))}%
           {\partial_{x}f(F^{k}_0(p))\cdots \partial_{x}f(p)},
\end{align}
where the $k^{th}$ term in the sum is bounded by
$\|\partial_\theta f\|\lambda^{-(k+1)}$. From this it is clear that the slope fields $s_n$
converge uniformly to a slope field $s_*$ exponentially fast;
$s_*(x,\theta)$ thus identifies an invariant centre distribution for
$F_0$.
\begin{lem}\label{lem:s_*_holder}
	The function $s_*$ is $\eta$-H\"older for sufficiently small $\eta$.
\end{lem}
\begin{proof}
  For $\eta\in (0,1]$ and $A:\T^2 \to \T$, let
  $|A|_\eta = \sup_{p\ne q}|A(p)-A(q)|/d(p,q)^\eta$. Note the
  interpolation inequality
  $|AB|_{\eta}\le |A|_\eta\|B\|_\infty+\|A\|_\infty|B|_\eta$ and
  $|A\circ g|\le |A|_\eta \|g\|_{\cC^1}^\eta$ for all
  $A,B\in \cC^\eta(\T^2,\T)$ and $g\in \cC^1(\T^2,\T^2)$. Thus
  by~\eqref{NstepRecur},
	\begin{align*}
		|s_n|_{\eta}&\le \lambda^{-1}\|F_0\|_{\cC^1}^\eta |s_{n-1}|_{\eta} +\left|\frac{1}{\partial_x f}\right|_\eta \|s_{n-1}\|_\infty+\left|\frac{\partial_\theta f}{\partial_x f}\right|_\eta \\
		&\le \tfrac12 |s_{n-1}|_\eta + \|s_{n-1}\|_\infty\Const + \Const
	\end{align*}
	for $\eta$ small enough, so $|s_n|_\eta$ is uniformly bounded. Since $s_n\to s_*$ in $\cC^0$, it follows that $s_* \in \cC^\eta$.
\end{proof}
%Then $1+\ve \psi_*$ turns out to be the one step expansion in the center direction. To see this we repeat the above computation for $F_\ve$.
% **** The actual map

Next, we consider $F_\ve$. By the invariance of the centre cone,
similarly to what was done earlier, we can define
\begin{equation}\label{epsnStepCenter}
\Upsilon^\ve_{n}(p)(s^\ve_{n}(p),1) = [d_pF_{\ve}^{n}]\inv (0,1),
\end{equation}
where $|s^\ve_n|<\centreConeConst$.  Notice that in this case, $\Upsilon_{n}^{\ve}$
is not identically one.  From the above, we obtain, denoting
$p_{k} = F_{\ve}^{k}(p)$ for $k\ge 0$:
\begin{equation}\label{EpsSlopeRec}
d_pF_\ve (s^\ve_n(p_{0}),1) = \frac{\Upsilon^\ve_{n-1}(p_{1}) }{\Upsilon^\ve_{n}(p_{0})}(s^\ve_{n-1}(p_{1}),1)
\end{equation}
which along with \eqref{DerivF} implies
\begin{equation}\label{NstepExp}
\frac{\Upsilon^\ve_{n-1}(p_{1}) }{\Upsilon^\ve_{n}(p_{0})} = 1 + \ve (\partial_\theta\omega(p_{0})+\partial_x\omega(p_{0})s^\ve_n(p_{0}))
\end{equation}
and
\begin{equation}\label{EpsCNstepRecur}
  s^\ve_n(p_{0}) = \frac
  {(1+\ve\partial_\theta\omega(p_{0}))s^\ve_{n-1}(p_{1})-\partial_\theta f(p_{0})}
  {\partial_x f(p_{0})-\ve \partial_x\omega(p_{0})s^\ve_{n-1}(p_{1})} =: \Xi^-(p_{0},s^\ve_{n-1}(p_{1})).
\end{equation}

Using \eqref{NstepRecur} and \eqref{EpsCNstepRecur}, one can conclude
\begin{align*}
  |s_n(p)-s^\ve_n(p)| \le \Const n\ve.
\end{align*}
Also, a direct computation (see~\cite[p.167-168]{MR3556527}) gives that for sufficiently small $\ve$,
there exist $\sigma \in (0,1)$ such that if $s$ is so that $|s| <
\centreConeConst$, then
\begin{align*}
  \left|\frac{\partial}{\partial s}\Xi^-(p,s)\right| \leq \sigma.
\end{align*}
This implies that for all $n$,
$|s^\ve_n(p)-s^\ve_{n-1}(p)| \leq \Const\sigma^n$ and hence, the slope
fields $s^\ve_n$ converge uniformly at an exponential rate to a slope
field $s^\ve_*$; which identifies an invariant centre distribution for
$F_{\ve}$. The slope field $s_{*}^{\ve}$ is, a priori, only continuous
in $(s,\theta)$; we will study the integrability properties
of the slope field $s_{*}^{\ve}$ in
Subsection~\ref{sec:local-centre-manif}.

Picking some $n \asymp \log \ve^{-1}$, and using the exponential
convergence,
\begin{align*}
|s^\ve_*(p)-s_*(p)| &\leq |s^\ve_*(p)-s^\ve_{n}(p)| + |s_n(p)-s^\ve_n(p)| +|s_n(p)-s^*(p)| \\ &\le \, \Const(\ve + \ve \log \ve^{-1} + \ve) \, \le \, \Const\ve \log \ve^{-1}.
\end{align*}
Combining this estimate,~\eqref{NstepExp} and the definition of
$\psi_{*}$, we conclude:
\begin{equation}
	\begin{alignedat}{1}
		\frac{\Upsilon^\ve_{n-1}(p_{1}) }{\Upsilon^\ve_{n}(p_{0})} &= 1+\ve\psi_*(p_0) + \ve \partial_x w(p_0)(s^\ve_n(p_0)-s_*(p_0))\\
		&= 1+\ve\psi_*(p_0)+\ve \partial_x w(p_0)[(s^\ve_n(p_0)-s^\eps_*(p_0))+(s^\eps_*(p_0)- s_*(p_0))]
	\end{alignedat}
	\label{eq:1Step-vs-Upsilon}
\end{equation}
Taking $n\to \infty$ in \eqref{EpsSlopeRec}, we observe that the one
step expansion along the centre direction, $(s^\ve_* ,1)$ is given by
\begin{align}\label{eq:1StepExpansionEstimate}
  \CentreExpansion(p_0)=\lim_{n\to \infty} \frac{\Upsilon^\ve_{n-1}(p_{1}) }{\Upsilon^\ve_{n}(p_{0})} =1+\ve\psi_*(p_0)+ \cO(\ve^2 \log \ve^{-1} ).
\end{align}
\nomenclature[0ac]{$\CentreExpansion$}{One-step expansion along the
  centre direction $(s^\ve_*, 1)$}%
Hence, up to a well-controlled error, the one-step expansion
$\CentreExpansion$ in the centre direction is $(1+\ve\psi_*)$; in
particular, we can fix $\centreMaxExp > 0$ so that, for any $n > 0$
and $p\in\bT^{2}$:
\begin{align}\label{eq:centre-manifold-expansion}
  e^{-\centreMaxExp n\ve}\le |d\pi_{2}d_pF^{n}_\ve
  (s^\ve_*(p),1)| \le e^{\centreMaxExp n\ve}.
\end{align}
We conclude this section with a remark about Lyapunov exponents
\begin{rmk}\label{rem:lyap-exponent}
  Combining the above discussion with~\cite[Section 6]{DLPV}, we have
  that the central Lyapunov exponent with respect to any ergodic
  physical measure $\nu$ is
  \begin{align}\label{CentExp}
    \lambda_{c,\nu} &=\nu(\log(1+\ve\psi_*))+\cO(\ve^2 \log \ve^{-1})\nonumber \\ &= \ve \nu(\psi_*)+\cO(\ve^2 \log \ve^{-1}) \nonumber \\ &=\ve \bar \psi_*(\theta_-)+\ve\int_{\T}\left(\bar\psi_*(\theta_-) - \bar\psi_*(\theta)\frac{1}{\sigma\sqrt{2\pi\ve}} e^{-\frac{ (\theta - \theta_-)^2}{2\ve\sigma^2}}\right){d}\theta +\cO(\ve^{3/2})\nonumber\\ &=\ve\bar\psi_{*}(\theta_{-}) + o(\ve),
  \end{align}
  where we used~\cite[Proposition 4]{DLPV} to approximate an ergodic
  physical measure (concentrated at the sink $\theta_{-}$) by a
  Gaussian centred at $\theta_{-}$ and variance $\cO(\eps)$ up to an
  error $\cO(\sqrt{\ve})$.
  Therefore, $\lambda_{c,\nu}>0$ for sufficiently small $\ve >0$.
\end{rmk}
\subsection{The unstable direction}

Next, we focus on the unstable direction. Note that by the forward invariance of the unstable cone, we can define
\begin{align}\label{epsnStepUnstable}
  d_pF^n_\ve(1,0) &=: \Gamma_n^\ve(p)(1,\ve w^\ve_n(p)).
\end{align}
where $|w^{\ve}_n|<\unstableConeConst$.
Therefore,
\begin{align*}
  \Gamma_{n+1}^\ve(p)(1,\ve w^\ve_{n+1}(p) )%
  &=d_p F^{n+1}_\ve(1,0) = [d_{p_{n}} F_\ve]\circ [d_p F_\ve^{n}](1,0) \\%
  &= \Gamma_{n}^\ve(p) [d_{p_{n}}F_\ve ] (1,\ve w^\ve_{n}(p) ).
\end{align*}
Using \eqref{1StepExp},
\begin{align}\label{UnstableExp}
  \frac{\Gamma_{n+1}^\ve(p)}{\Gamma_{n}^\ve(p) } %
  &= \partial_x f(p_{n})\left(1+\ve \frac{\partial_\theta f(p_{n})}{\partial_x f(p_{n})} w^\ve_{n}(p)\right)
\end{align}
Also,
\begin{align}\label{RecurUnstable}
w^\ve_{n+1}(p) &= \frac{\partial_x\omega(p_{n})+(1+\ve\partial_\theta
                 \omega(p_{n}))w^\ve_{n}(p)}{\partial_x f(p_{n}) + \ve
                 \partial_\theta f(p_{n}) w^\ve_{n}(p) } =\Xi^+(p_{n},w^\ve_{n}(p))
\end{align}
We will use this fact later.

\begin{lem}[Lyapunov exponents]\label{lem:lyapunov}
  Let $\ve$ be sufficiently small and assume $0\le n < \Const\vei$;
  then for any $v$:
\begin{align}\label{eq:bound-differential-fine}
    \Const[\Upsilon_{n}^{\ve}(p)]\inv\|v\|\le \|d_{p} F^{n}_{\ve}v\|\le
    \Const\Gamma_{n}^{\ve}(p)\|v\|.
 \end{align}
In particular,
\begin{align}\label{eq:bound-differential}
    \Const(1-\Const\ve)^{n}\|v\|\le \|d_{p} F^{n}_{\ve}v\|\le
    \Const(1+\Const\ve)^{n}\prod_{j = 0}^{n-1}\partial_{x}f(F^{j}_{\ve}p)\|v\|
 \end{align}
\end{lem}
\begin{proof}
Given $v \in \reals^2$, write $v=v_1(1,0)+v_2(s^\eps_n(p),1)$. Then
\begin{align*}
d_{p} F^{n}_{\ve}v &= v_1\Gamma_n^\ve(p)(1,\ve w^\ve_n(p)) +v_2 [\Upsilon^\eps_n(p)]^{-1}(1,0).
\end{align*}
Choosing $\eps$ small,\footnote{In fact, the factor $\lambda^{n/2}$ can be made as close as we want to $\lambda^{n}$ by choosing $\eps$ sufficiently small.} $\Gamma^\eps_n(p)>\lambda^{n/2}[\Upsilon^\eps_n(p)]^{-1} $ and
\begin{align*}
  \|d_{p} F^{n}_{\ve}v\| \leq (|v_1|\sqrt{1+\eps^2w^\eps_n(p)^2}+|v_2|)\Gamma^\eps_n(p)\leq \Const\Gamma_{n}^{\ve}(p)\|v\|.
\end{align*}

Next, given $v \in \reals^2$, write $v=v_1(0,1)+v_2(1, \eps  w^\ve_n(p))$. Then
\begin{align*}
[d_{p}F^{n}_{\ve}]^{-1}v &= v_1 [\Upsilon^\eps_n(p)](s^\eps_n(p),1)+ v_2[\Gamma_n^\ve(p)]^{-1}(1,0).
\end{align*}
and
\begin{align*}
\|[d_{p}F^{n}_{\ve}]^{-1}v \| \leq  (|v_1 |\sqrt{1+s^\eps_n(p)^2}+ |v_2|)\Upsilon^\eps_n(p) \leq \Const\Upsilon^\eps_n(p) \|v\|
\end{align*}
for sufficiently small $\eps$. This gives
\begin{align*}
\|d_{p}F^{n}_{\ve} v \|\geq \Const [\Upsilon^\eps_n(p)]^{-1} \|v\|
\end{align*}
So, we have \eqref{eq:bound-differential-fine}. \eqref{eq:bound-differential} follows from \eqref{eq:bound-differential-fine} due to \eqref{UnstableExp} and \eqref{NstepExp}.
\begin{comment}
  for an arbitrary vector $v$, decompose along $(1,0)$ and
  $(s_{n}^{\ve})$; then use the definition of $s_{n}$ and $w_{n}$ and
  compare the outcome.  \separator
  In fact we can use also a bound in terms of
  the actual centre foliation. Then we could drop the assumption on $n
  = O(\vei)$ and take a limit.  This limit would yield —essentially–
  the Lyapunov exponents.
\end{comment}
\end{proof}
%\begin{rmk}
%\end{rmk}
%Observe {\todo{we need to make some computations on $s_{*}$ to justify this step}} that assumption (A2) implies that $\bar\psi_{*}(\theta_{-}) > 0$; indeed without loss of generality, we can always fix a scale for $\omega$ (by rescaling $\ve$) so that assumption \ref{a_almostTrivial} reads
%\begin{enumerate}[label=(A\arabic*$^\prime$),start=2]
%\item\label{a_almostTrivialp} $\bar\psi_{*}(\theta_{-})=1$.
%\end{enumerate}
\subsection{Local centre manifolds}\label{sec:local-centre-manif}
In this subsection we collect some results about integrability of the
centre slope field $s_{*}^{\ve}$.  We say that a centre slope field
$s$ is \emph{locally uniquely $C^{r}$-integrable} if for any $p\in\bT$
there exists a $C^{r}$ (centre) curve $\cW_{s}(p)$ and $\alpha(p) > 0$
so that every piecewise $C^{1}$ curve $\gamma:(-1,1)\to\bT$ with
$\gamma(0) = p$, $\dot\gamma(t) \propto (s(p),1)$ and
$\textrm{height}(\gamma) < \alpha(p)$ is contained in $\cW_{s}(p)$.
Local $C^{r}$ integrability is guaranteed for sufficiently smooth
slope fields by classical ODE results (\eg every $s_{n}^{\ve}$ is
locally $C^{r}$ integrable); however, as it often happens in partially
hyperbolic dynamics, the invariant centre slope field
$p\mapsto(s_{*}^{\ve}(p),1)$ enjoys very poor smoothness properties.
Despite this inconvenience, it holds however true that the invariant
centre slope field is locally uniquely integrable; this follows from
classical results~\cite{MR1972227}, as proved in~\cite[Section
7]{DLPV}; we extract the results that are relevant for this paper and
summarize them in the next theorem, which is a rephrasing
of~\cite[Theorem 6 and 7]{DLPV}
\begin{thm}\label{thm:unique-integrability-centre-mfolds}
  The invariant centre distribution $s_{*}^{\ve}$ is locally uniquely
  $C^{4}$-integrable; the integral leaves are compact, and
  homeomorphic to $\bT^{1}$.
\end{thm}
\begin{rmk}
  The resulting foliation in integral leaves (centre manifolds) has in
  general very poor smoothness properties; in particular, in our
  setting it will not be absolutely continuous.
\end{rmk}
A centre curve is called a \emph{local centre manifold} if it is a
subcurve of an integral leaf of $s_{*}^{\ve}$ with the property that
its projection on the second coordinate is an interval of length
less\footnote{ The constraint on the length is not essential, but it
  is indeed convenient.} than $1/2$.  We will denote local centre
manifolds by the symbol $\Wc$, and we will denote by
$\height(\Wc) = |\pi_{2}\Wc|$ the length of the interval $\pi_{2}\Wc$.
\nomenclature[1_1]{$\Wc$}{Local centre manifold}%
The result stated above implies that for any point $p\in\bT^{2}$ there
exists a “unique” local centre manifold (of positive length) passing
through $p$; uniqueness here is intended in the sense that the
intersection of any two local centre manifolds passing through $p$ is
itself a local centre manifold passing through $p$.

Let us now control the $n$-step expansion along local centre manifolds
for $n = \cO(\eps\inv)$. Recall from~\eqref{eq:1StepExpansionEstimate}
that the one-step expansion along the centre direction $(s^\ve_* ,1)$,
which we denote by $\CentreExpansion$, is approximately equal to
$1+\eps \psi_*$. Since the function $\psi_*$ is typically not smooth
(although it is H\"older continuous by Lemma~\ref{lem:s_*_holder}), we find it
convenient to define a regularized function $\psi$ that approximates
$\psi_*$.

Recall that
$\psi_*(p)=\partial_\theta\omega(p)+\partial_x\omega(p)s_*(p)$ and let
$\psi_n(p) = \partial_\theta\omega(p)+\partial_x\omega(p)s_n(p)$.
\nomenclature[0ad]{$\fakeCentreConst$}{Constant that appears in
  estimate on the growth of centre manifolds
  (Lemma~\ref{lem:expansion-centre-mfolds})} Pick
$0<\fakeCentreConst < \frac{1}{4}$ small (to be determined below) and
$n_0$ such that $\|\psi_{n_0}-\psi_*\| < \fakeCentreConst$. We define
$\psi = \psi_{n_0}$ and
\begin{align}\label{eq:zeta_n_def}
	\zeta_n &= \eps \sum_{k=0}^{n-1} \psi\circ F^k_\eps.
\end{align}
\begin{rmk}\label{rmk:psi_regularity}
  Since $F_\eps \in \cC^4$, the formula for $s_n$
  in~\eqref{NstepCenter} implies that $\psi\in \cC^3$. Let
  $\bar{\psi}(\theta)=\int_{\bT^1}\psi(x,\theta)\rho_\theta(x) \deh
  x$. Then by~\ref{a_almostTrivial}, we have that
  $\bar{\psi}(\theta_-)\ge \bar{\psi}_*(\theta_-)-\fakeCentreConst\ge
  \frac{3}{4}$.  We moreover need to ask $\fakeCentreConst$ to be so
  small that~\ref{a_noNearCobo} holds when substituting $\psi_{*}$
  with $\psi_{n_{0}}$.
\end{rmk}
For later use, we also define the interpolation $\zeta_{\ve}(t)$ as we
did for $\theta_{\ve}$; once again, we will regard $\zeta_{\ve}$ as a
random element with values in in $\mathcal C^{0}(\bR_{+},\bR)$.
% TODO(edit): change N to n
\begin{lem}\label{lem:expansion-centre-mfolds}
  Recall the definition of $\centreMaxExp$ given
  above~\eqref{eq:centre-manifold-expansion} and the
  definition~\eqref{eq:1StepExpansionEstimate} of $\CentreExpansion$,
  then
	\begin{enumerate}
		\item  for any $p\in \bT^2$:
		\begin{align*}
			e^{-\centreMaxExp n\ve}\le
          \prod_{k=0}^{n-1}\CentreExpansion\circ F^k_\eps(p)\le e^{\centreMaxExp n\ve}.
		\end{align*}
        \end{enumerate}
        Let $T>0$ and let $n\le T\eps\inv$. Then, for any local centre
        manifold $\Wc$,
        %%% This is very poor Lateχ, but it looks horrible otherwise
		\begin{align*}
          &\textup{(b)}&\inf_{\Wc}\prod_{k=0}^{n-1}\CentreExpansion\circ F_\eps^k &\ge \exp\bigg(\sup_{\Wc}\zeta_n-C_T(\height(\Wc)+\eps\log\eps^{-1})-T\fakeCentreConst\bigg);\\
          &\textup{(c)}&
			\sup_{\Wc}\prod_{k=0}^{n-1}\CentreExpansion\circ F_\eps^k &\le \exp\bigg(\inf_{\Wc}\zeta_n+C_T(\height(\Wc)+\eps\log\eps^{-1})+T\fakeCentreConst\bigg).
		\end{align*}
\end{lem}
\begin{rmk}
  We will not use part (c) of this lemma in the present article;
  however, we include it for completeness. Note that parts (b) and (c)
  provide sharper estimates than part (a) only when $n$ is of order
  $\eps\inv$.
\end{rmk}
\begin{proof}
	Part (a) follows immediately from~\eqref{eq:centre-manifold-expansion}.

	Let $p\in \bT^2$ and write $p_k =
    F_\eps^k(p)$. By~\eqref{eq:1StepExpansionEstimate},
    $\|\CentreExpansion-1-\eps \psi \|\le \Const \eps^2 \log
    \eps\inv+\fakeCentreConst\ve$. Thus for $\eps$ sufficiently
    small,
	\begin{align}
		\sum_{k=0}^{n-1}\log \CentreExpansion(p_k) &\ge \sum_{k=0}^{n-1}(\CentreExpansion(p_k) - 1 - \Const \eps^2)
		\ge \sum_{k=0}^{n-1}(\eps\psi(p_k) - \fakeCentreConst\eps  - \Const \eps^2 \log \eps\inv)\notag\\
		&\ge \zeta_n(p) - T\fakeCentreConst - T\Const \eps\log \eps\inv\label{eq:nStepPtwiseBd}
	\end{align}
	and similarly
	\begin{equation}\label{eq:nStepPtwiseUpperBd}
		\sum_{k=0}^{n-1}\log \CentreExpansion(p_k) \le \zeta_n(p) + T\fakeCentreConst + T\Const \eps\log \eps\inv.
	\end{equation}
	Notice that:
	\begin{align*}
		\sup_{\Wc}\zeta_{n}-\inf_{\Wc}\zeta_{n}\le \|d\zeta_{n}\|\,\height(\Wc);
	\end{align*}
	we now proceed to obtain an upper bound on $\|d\zeta_{n}\|$: since
    $dF_\eps(s_*^\eps,1) = \CentreExpansion\cdot (s^{\eps}_* \circ
    F_\eps, 1)$, we have
	\begin{align*}
      d\zeta_n(s_*^\eps,1) = \eps\sum_{k=0}^{n-1}d\psi\circ F_\eps^k\, dF_\eps^k (s_*^\eps,1) = \eps\sum_{k=0}^{n-1} \prod_{j=0}^{k-1}\CentreExpansion\circ F_\eps^j \, d\psi\circ F_\eps^k(s_*^\eps\circ F_\eps^k,1).
	\end{align*}
	Hence by part (a) of this lemma we obtain that
	\begin{align}\label{eq:dzeta_n_bd}
		|d\zeta_n(s_*^\eps,1)|\le \eps\sum_{k=0}^{n-1} e^{\centreMaxExp k \eps}\|d\psi\|\sqrt{|\centreConeConst|^2 + 1} \le \Const \frac{\eps e^{\centreMaxExp T}}{1 - e^{\centreMaxExp \eps}}\le C_{T}.
	\end{align}
	Part (b) of this lemma then follows by combining~\eqref{eq:nStepPtwiseBd} and~\eqref{eq:dzeta_n_bd}. Similarly, part (c) follows by combining~\eqref{eq:nStepPtwiseUpperBd} and~\eqref{eq:dzeta_n_bd}.
\end{proof}
\begin{lem}\label{lem:fake_centre_exp}
	Let $T>0$. Then there exists $C_T>0$ such that for any local centre manifold $\Wc$ and any $0\le n\le T\eps\inv$, we have
	\begin{equation*}%\label{CentreExpLwrBnd}
		\inf_{\Wc}\,[\Upsilon^\ve_{n}]^{-1} \ge \exp\bigg(\sup_{\Wc}\zeta_n-C_T(\height(\Wc)+\eps\log\eps^{-1})-T\fakeCentreConst\bigg).
	\end{equation*}
\end{lem}
\begin{proof}
	By \eqref{eq:1Step-vs-Upsilon}, for all $q\in \bT^2$ and $k\ge 1$, we have
	\begin{align*}
		\frac{\Upsilon^\ve_{k-1}(F_\eps(q)) }{\Upsilon^\ve_{k}(q)} &= \CentreExpansion(q)+\ve \partial_x w(q)[s^\ve_k(q)-s^\eps_*(q)].
	\end{align*}
	Now $\CentreExpansion(q)\ge 1-\Const\eps$ and $\|s^\eps_k -s^\eps_* \|\le \Const\sigma^k$ so for all $\eps$ sufficiently small,
	\begin{align*}
		\log\frac{\Upsilon^\ve_{k-1}(F_\eps(q)) }{\Upsilon^\ve_{k}(q)}\ge \log \CentreExpansion(q) - \Const \eps \sigma^k
	\end{align*}
	Hence by taking $q = F_\eps^{n-k}(p)$ and summing over $1\le k\le n$ we obtain that
	\begin{align*}
		\log\Upsilon^\ve_{n}(p)^{-1} \ge \sum_{k=1}^n (\log \CentreExpansion\circ F_\eps^{n-k}(p) - \Const\eps \sigma^k)\ge \sum_{i=0}^{n-1} \log \CentreExpansion\circ F_\eps^{i}(p)- \Const \eps.
	\end{align*}
	The proof of the lemma follows by combining this bound with~\eqref{eq:nStepPtwiseBd} and~\eqref{eq:dzeta_n_bd}.
\end{proof}

\section{Standard pairs and standard patches}\label{sec:std-pairs-patches}
In order to establish statistical properties of $F_{\ve}$, we
introduce a particular class of probability measures, that we call
\emph{standard patches}.  Such measures are inspired by the class of
\emph{standard pairs} introduced by Dolgopyat~\cite{Dima-mostly-contracting} in the early 2000s.  The
idea behind these measures is that on the one hand they should be
sufficiently localized so that they can be used for conditioning,
while on the other hand they should be sufficiently rich so that they
can be used to derive limit theorems.  In the system under our
consideration, standard pairs are measures supported on sufficiently
short unstable curves and described by a sufficiently regular density
with respect to Lebesgue measure on the curve.  Standard patches, on
the other hand, capture the “mostly expanding” feature of our
dynamics, and can be –for now vaguely– depicted as an $\ve$-thickening
of standard pairs.

In Subsections~\ref{sec:std_pairs} and~\ref{sec:std_pairs_dynamics} we
recall from~\cite{MR3556527} the appropriate definitions of standard
pairs and their dynamical properties.
Subsection~\ref{sec:standard-patches} is devoted to introducing the
definition of standard patches and proving some of their dynamical
properties.
\subsection{Standard curves and pairs:
  definitions}\label{sec:std_pairs}
% *** Definition of standard curves
\nomenclature[1aa]{$\scc{i}$}{Constants appearing in the definition of
  standard curves}%
\nomenclature[1ab]{$\pscc{i}$}{Constants appearing in the definition
  of prestandard curves}%
\nomenclature[1ac]{$\Stdc$}{Set of standard curves}%
\nomenclature[1ad]{$\pStdc$}{Set of prestandard curves}%
\nomenclature[1ae]{$\narr$}{Quantity describing the \emph{narrowness} of
  a standard curve}

Fix $\delta > 0$ and $\trim > 0$ to be determined later and let
$\narr \ge 2$.  A closed interval $I\subset\bT$ is called a
\emph{$\narr$-interval} if $|I|\in[\delta/\narr,\delta]$; a
$\narr$-interval is said to be \emph{trimmed} if
$|I|\in [\delta e^{\trim}/\narr, \delta e^{-\trim}]$.  A curve in
$\bT^{2}$ is said to be a (trimmed) $\narr$-curve if it projects
bijectively by $\pi_{1}$ onto a (trimmed) $\narr$-interval\footnote{\
  We introduce here the notion of a trimmed curve to allow some
  fuzziness in the definition of a $\narr$-curve; see also
  Footnote~\ref{fn:justify-trimmed}}.

Let $\scc1,\scc2$ and $\scc3$ be positive real
numbers; then we define
\begin{align*}
  \Stdc_{\narr}(\scc1,\scc2,\scc3)=\Big\{G \in\cC^{3}(I,\T)&:%\\
  I\textrm{ is a $z$-interval},
  \|G^{(1)}\|_{\infty} \leq \ve \scc1, \\&%
  \|G^{(2)}\|_{\infty} \leq \ve \scc2,%
  \|G^{(3)}\|_{\infty} \leq \ve \scc3\Big\}.
\end{align*}%%%%
Let us now fix two sets of constants $(\scc1,\scc2,\scc3)$ and
$(\pscc1,\pscc2,\pscc2)$ with $0 < \scc{i} < \pscc{i}$ to be
determined later and define the shorthand notation:
\begin{align*}
  \Stdc_{\narr} &= \Stdc_{\narr}(\scc1,\scc2,\scc3) &%\\
  \pStdc_{\narr} &= \Stdc_{\narr}(\pscc1,\pscc2,\pscc3).%\\
\end{align*}
Given $G\in\Stdc_{\narr}$, we introduce the function
$\bG:x\mapsto(x,G(x))$.  The image of $\bG$ (\ie the graph
$\{(x,G(x))\}_{x\in I}$ of $G \in \Stdc_{\narr}$) is called a
$\narr$-\emph{standard curve};\footnote{ Note that a $\narr$-standard
  curve is in particular a $\narr$-curve} the function $G$ is called
its \emph{associated function}.  The parameter $\narr$ controls how
\emph{narrow} a curve in the class $\Stdc_{\narr}$ is allowed to be:
the larger $\narr$, the narrower a curve; of course
$\Stdc_{\narr}\subset\Stdc_{\narr'}$ if $\narr < \narr'$.  A
$z$-standard curve is said to be trimmed if it is a trimmed $z$-curve.
%
% **** Variation 2: pre-standard curves
%
We define similarly \emph{prestandard curves} (and \emph{trimmed
  prestandard curves}) by replacing $\Stdc_{\narr}$ with
$\pStdc_{\narr}$ in all the above
definitions.\\

% *** Definition of stacked :COMMENT:
\ignore{Two standard curves are said to be \emph{stacked} if their
  projections to the $x$-axis coincide and $\Delta$-stacked if they
  are stacked and $\Delta$ close in $\cC^1$
  topology.%{\todo{Not sure if we will need
 %   this definition; we will see}}
}%

% *** Definition of standard density
\nomenclature[1ba]{$\Dconst$}{Constant appearing in the norm control
  of standard densities on curves}
\nomenclature[1bb]{$\rough$}{Parameter associated to the
  \emph{roughness} of a standard density on curves}
\nomenclature[1bc]{$\Stdd$}{Set of standard densities on curves} Fix
$\Dconst > 0$ to be specified later; for any $\rough>0$, we define the
set of $\rough$-\emph{standard probability densities} on a standard
curve $G$ as
\begin{align*}
  \Stdd_{\rough}(G) &= \{\rho \in \cC^2(I,\R_{ > 0}):\
                      \|\rho\|_{L^1} =1, \|\rho^{(1)}/\rho\|_{\infty}\leq
                      \rough,\|\rho^{(2)}/\rho\|_{\infty}\leq \Dconst\,\rough \}.
\end{align*}
  The parameter $\rough$ controls how \emph{rough} a density in the
class $\Stdd_{\rough}$ is allowed to be; the larger $\rough$, the
rougher the density.
% *** Definition of standard pair
\nomenclature[1ca]{$\ell$}{standard pair}
\nomenclature[1cc]{$\Stdp$}{Collection of all standard pairs}

Define a $(\narr,\rough)$-\emph{standard pair} $\ell$ as a pair
$(\bG,\rho)$ given by $\bG$, the graph of $G \in \Stdc_{\narr}$, and a
density $\rho \in \Stdd_{\rough }(G)$; we denote with
$\Stdp_{\narr,\rough}$ be the collection of all
$(\narr,\rough)$-standard pairs.  A standard pair is said to be
\emph{trimmed} if the associated standard curve is trimmed. Note that
each $\ell = (\bG,\rho) \in \Stdp_{\narr,r}$ induces a Borel
probability measure on $\T^2$ as follows: for any continuous
real-valued function $g$ on $\T^{2}$ we let
\begin{align*}
  \mu_{\ell}(g) = \int_I g(x,G(x))\rho(x)\, dx,
\end{align*}
where $I = \pi_{1} \bG$.  We also introduce a slight abuse of notation
by calling $\supp \ell = \supp \mu_{\ell}$. The set
$\Stdp_{\narr,\rough}$ of $(\narr,\rough)$-standard pairs can be
identified as a space of smooth functions: it is thus naturally a
measurable space with the Borel $\sigma$-algebra. If $\bG:I\to\bT^2$
and $\rho:I\to\bR_{ > 0}$ are defined as above, let $\hat\bG$ and
$\hat\rho$ be defined by precomposing $\bG$ and $\rho$ respectively
with the unique affine orientation-preserving transformation that maps
$[0,1]$ onto $I$.

% *** Definition of standard family
\nomenclature[1cb]{$\stdf$}{standard family}%

The symbol $\stdf$ will denote a \emph{family of}
$(\narr,\rough)$-\emph{standard pairs}, that is, a \emph{random}
$(\narr,\rough)$-standard pair.\footnote{ These objects are also
  called standard families in the literature.  Since in the sequel we
  will use families of standard pairs and family of standard patches,
  we prefer to be more explicit in the wording. } More precisely,
$\stdf$ denotes a Lebesgue probability space
$(\mathcal{A},\mathcal{F}, \nu)$ together with a
$\mathcal{F}$-measurable map
$\ell: \mathcal A \to \Stdp_{\narr,\rough}$.  A standard pair-valued
function is thus $\cF$-measurable if both maps
$(\alpha,s)\mapsto \hat\bG_{\alpha}(s)$ and
$(\alpha,s)\mapsto \hat\rho_{\alpha}(s)$ are jointly measurable. In
particular, for any Borel set $E\subset\bT^2$, the function
$\alpha\mapsto\mu_{\ell_{\alpha}}(E)$ is $\cF$-measurable.
% **** Example
Each family of standard pairs
$\stdf = ((\mathcal{A},\mathcal{F}, \nu),\ell)$ induces a Borel
probability measure on $\bT^{2}$ defined by:
\begin{align*}
  \mu_{\stdf}(g) = \int_{\mathcal A} \mu_{\ell(\alpha)}(g)d\nu[\alpha].
\end{align*}
For example, given a sequence of standard pairs, $\ell_i$, and weights
$0< c_{i} \leq 1$ so that $\sum_{i} c_{i} =1$, the associated family of
standard pairs $\stdf$ induces the measure:
\begin{align*}
  \mu_{\stdf}(g) = \sum_{i} c_{i}\int_{I_i}g(x,G_i(x))\rho_i(x)\, dx.
\end{align*}
A Borel probability measure $\mu$ on $\T^2$ is said to \emph{admit a
  disintegration as a family of $(\narr,\rough)$-standard pairs} if
there exist a family of $(\narr,\rough)$-standard pairs $\stdf$ such that
$\mu_{\stdf} = \mu$.%
\ignore{ A family of standard
  pairs $\stdf$ is said to be \emph{finite} if the associated space
  $\mathcal A$ is a finite set.}  We can likewise define families of
\emph{trimmed} standard pairs.  Finally, we define \emph{prestandard
  pairs} (and \emph{families of prestandard pairs}) by replacing
$\Stdc_{\narr}$ with $\pStdc_{\narr}$ in the above definitions.  We
denote with $\pStdp_{\narr, \rough}$ the set of
$(\narr,\rough)$-prestandard pairs.

% ** The dynamics of standard pairs
\subsection{Standard curves and pairs:
dynamics}\label{sec:std_pairs_dynamics}
We now proceed to describe the behaviour of $(\narr,\rough)$-standard
pairs under the dynamics of $F_\ve$.  The following proposition
amounts to a consolidation of~\cite[Proposition 5.2, Remarks 5.6, 5.7
and 5.8]{MR3556527} and some minor improvements.
% *** Invariance of standard pairs
\begin{prp}[Dynamics of standard pairs]\label{proposition:invariance}
  Choosing $\delta > 0$ sufficiently small, there exist constants
  $\pscc1$, $\pscc2$, $\pscc3$, $\scc1$, $\scc2$, $\scc3$ and $\trim$
  so that the following holds for any sufficiently small $\ve >
  0$. For any $\narr\ge2$ and any $\narr$-prestandard curve
  $G\in\pStdc_{\narr}$:
  \begin{enumerate}
  \item the image $F_{\ve}\bG$ can be partitioned (mod 0) into
    finitely many trimmed $\narr'$- curves with
    \begin{align*}
        \narr' &= \max\left\{\frac45\narr, 2\right\}.
    \end{align*}
  \item any element of a partition (mod 0) of $F_{\ve}\bG$ into
    $\narr'$-curves is a $\narr'$-standard curve.
  \end{enumerate}
  Moreover, choosing $\Dconst$ sufficiently large there exists $\stdrough > 0$ so that for any $\narr\ge 2$,
  $\rough > 0$, any $(\narr,\rough)$-prestandard pair
  $\ell = (\bG,\rho) \in \pStdp_{\narr,\rough}$:
  \begin{enumerate}
     \setcounter{enumi}{2}
  \item any partition (mod 0) of $F_{\ve}\bG$ into $\narr'$-standard
    curves $\{\bG_{j}\}$ induces a disintegration of
    $F_{\ve*}\mu_\ell$ as family of $(\narr',\rough')$-standard pairs
    $\{(\bG_{j}, \rho_{j})\}$, where we can take
  \begin{align*}
    \rough' &= \frac13 \rough + \stdrough.
  \end{align*}
\end{enumerate}
\end{prp}
% **** Definition of regular pair
Before giving the proof of the proposition, let us introduce the
notion of \emph{\regular{}} standard pairs.  Such pairs were called
\emph{proper} in~\cite{MR3556527}, but we prefer to avoid the
confusion with the notion of properness given in~\cite{Chernov}, that
will be adapted later for our purposes.
\begin{mydef}\label{def:regular-pairs}
  A pair $\ell$ is said to be \emph{\regular{}} if it is a
  $(2,3\stdrough/2)$-standard pair.
\end{mydef}
% **** Remark about the statement
\begin{rmk}[Invariance of \regular{} standard
  pairs]\label{rem:regular-pair-attracting}
  It is immediate to check that, given the choice of constants in the
  above definition of \emph{\regular{}} pair,
  Proposition~\ref{proposition:invariance} implies that (families of)
  \regular{} standard pairs are \emph{invariant}, in the sense that
  the push-forward of a (family of) \regular{} standard pair can be
  disintegrated as a family of \regular{} standard pairs.

  The proposition also shows that the dynamics eventually brings any
  $(\narr,\rough)$-prestandard pair to be \regular{}.  More precisely:
  if $\ell\in \pStdp_{\narr,\rough}$, then for
  $n\sim\max\{\log \narr,\log \rough\}$ we have that
  $F^{n}_{\ve *}{\mu_{\ell}}$ admits a disintegration as a family of
  trimmed \regular{} standard pairs.
\end{rmk}
% **** Remark about the proof
\begin{rmk}
  The proof of Proposition~\ref{proposition:invariance} found below is
  a re-writing of the proof of~\cite[Proposition 5.2]{MR3556527} with
  emphasis on different aspects.  We decided to reproduce it here for
  completeness but also to mark the difference with the \emph{random}
  version of the invariance proposition that will be later presented
  as Proposition~\ref{lem:invariance_patches}.
\end{rmk}
At this point we can finally define (recall the definition of $M$
given below~\eqref{eq:def_Xi})
  \begin{align*}
    \unstableConeConst = \max\left\{\frac{M+1}{\lambda-2},\pscc1 \right\};
  \end{align*}
  in particular, notice that according to the above choice, any
  prestandard curve is unstable.  {Let us also observe that any
  unstable curve that projects by $\pi_{1}$ onto a proper subinterval
  of $\bT$ indeed projects \emph{bijectively} onto such
  subinterval.}
  % **** Proof of invariance lemma
\begin{proof}[Proof of Proposition~\ref{proposition:invariance}]
  Recall $\ell=(\bG,\rho)\in\pStdp_{\narr,\rough}$ is a
  $(\narr,\rough)$-prestandard pair.
% **** This is for the cutting
  Let us introduce the shorthand notation $f_{\bG}= f\circ\bG$ and
  $\omega_{\bG}= \omega\circ\bG$. Since any prestandard curve is an
  unstable curve,~\eqref{eq:unstable_cone_expansion} implies
  $f'_{\bG}=\pi_x dF_{\eps}\circ \bG \bG'>3$.  Let $I = \pi_{1}\bG$;
  provided that $\delta$ has been chosen small enough, $f_\bG$ maps
  $I$ injectively onto some interval $J\subset \bT$ of length
  $|J| \le 1/2$.  The image $F_{\ve}\bG$ is thus a graph of some
  function over $J$.

  In order to prove item (a), it thus suffices to show how to
  partition (mod 0) the interval $J$ into trimmed $z'$--intervals.  We
  can do this in several ways; for instance we can proceed as follows:
  define
\begin{align*}
  n = \left\lceil \frac{|J|}{\delta e^{-\trim}}\right\rceil;
\end{align*}
if $n = 1$ there is no need for partitioning, otherwise we will cut
$J$ into $n$ sub-intervals of equal length $|J|/n$.  We now show that
such sub-intervals are trimmed $z'$-intervals, \ie:
\begin{align*}
  |J|/n\in [\delta e^{\trim}/\narr',\delta e^{-\trim}].
\end{align*}
If $n = 1$ we have $|J| > \tilde\lambda|I| > 3\delta/\narr$, and we
are done, provided that $\trim$ is sufficiently small; otherwise, if
$n\ge3$ we have
\begin{align*}
  \frac{|J|}n > \frac{n-1}n\delta e^{-\trim}\ge
  \frac23\delta e^{-\trim} > \frac12\delta e^{\trim},
\end{align*}
provided that $\trim$ is chosen sufficiently small.  Finally, if
$n = 2$, we have either $|J|/2 > \frac12\delta e^{\trim}$ (hence we
are done), or otherwise $|J| \le \delta e^{\trim}$, which implies
$|I| < \frac13 \delta e^{\trim}$, which only can happen if
$\narr > \frac{11}4$.  But
$|J| > \frac12\delta e^{-\trim} > \frac5{11}\delta e^{\trim}$ if
$\trim$ is sufficiently small.  In all cases we have
\begin{align*}
  |J|/n\ge \min\{5/(4\narr),1/2\} \delta e^{\trim} \ge \delta e^{\trim}/\narr'.
\end{align*}
% **** Arbitrary partitioning
We now proceed to the proof of item (b); denote by $\{J_j\}$ a
partition of $J$ into $z'$-intervals.
% **** This is for the density
Let us introduce the convenient notation $\vf=f_{\bG}\invr:J\to I$
and let us denote with $\vf_{j}$ the restriction
$\vf_{j} = \vf|_{J_{j}}$.  Elementary calculus yields the following
expressions for the derivatives of $\vf$:
\begin{align}\label{e_estimatesVf}
  \vf'   &= \frac1{f_\bG'}\circ\vf &%\\
  \vf''  &= -\frac{f_\bG''}{f_\bG'^3}\circ\vf &%\\
  \vf''' &= \frac{3f_\bG''^2-f_\bG'''f_\bG'}{f_\bG'^5}\circ\vf.
\end{align}
Let us now define $\bar G(x):=G(x)+\ve\omega_\bG(x)$ and let
$G_{j}= \bar G\circ\vf_{j}$: by design, $\bG_{j}$ is a $z'$-curve for
any $j$.  We now proceed to show that
$G_{j}\in\Stdc_{\narr'}(\scc1,\scc2,\scc3)$ for appropriate choices of
$\scc1, \scc2, \scc3,\pscc1, \pscc2$ and $\pscc3$, thus concluding the
proof of item (b).  Differentiating the above definitions and
using~\eqref{e_estimatesVf} we obtain
  \begin{subequations}\label{e_C1}
    \begin{align}
      G_j' &=%
      \frac{\bar G'}{f'_\bG}\circ\vf_j\label{e_c1'}\\%
      G_j'' &=%
      \frac{\bar G''}{f'^2_\bG}\circ\vf_j%
      -G_j'\cdot\frac{f_\bG''}{f_\bG'^2}\circ\vf_j\label{e_c1''}\\%
      G_j''' &=%
      \frac{\bar G'''}{f'^3_\bG}\circ\vf_j%
      -3 G''_j\cdot\frac{f_\bG''}{f_\bG'^2}\circ\vf_j-
      G'_j\cdot\frac{f_\bG'''}{f_\bG'^3}\circ\vf_j\label{e_c1'''}%
    \end{align}
  \end{subequations}
  First, notice that for any sufficiently smooth function $A$ on
  $\bT^2$, the definition of prestandard curve allows to conclude
  that:
  \begin{subequations}\label{e_AG}
    \begin{align}
      \|(A\circ\bG)'\|  &\leq \|A\|\nc1 (1+\ve \pscc1)\label{e_AG'}\\
      \|(A\circ\bG)''\| &\leq \|A\|\nc2\left[(1+\ve \pscc1)^2+\ve\pscc2\right]\label{e_AG''}\\
      \|(A\circ\bG)'''\| &\leq \|A\|\nc3\left[(1+\ve\pscc1+\ve\pscc2)^3+\ve\pscc3\right]\label{e_AG'''}.
    \end{align}
  \end{subequations}
  % The estimates below are a bit more precise, but we do not need
  % those anywhere…
  \ignore{ \begin{subequations}\label{e_AG}
    \begin{align}
      \|(A\circ\bG)'\|  &\leq \|d A\| (1+\ve \pscc1)\label{e_AG'}\\
      \|(A\circ\bG)''\| &\leq \ve\|d A\|\pscc2+\|\deh A\|\nc1(1+\ve \pscc1)^2\label{e_AG''}\\
      \|(A\circ\bG)'''\| &\leq \ve\|d A\| \pscc3+\|\deh A\|\nc2(1+\ve(\pscc1+\pscc2))^3\label{e_AG'''}.
    \end{align}
  \end{subequations}}%
Using~\eqref{e_c1'}, the definition of $\bar G$ and~\eqref{e_AG'} we
obtain, for small enough $\ve$:
  \begin{align*}
    \|G'_j\| &\leq \left\|\frac{ G' + \ve\omega_\bG'}{f_\bG'}\right\|
    \leq\frac13 \ve\left[\pscc1+\|\omega\|\nc1(1+\ve\pscc1)\right] \\%(1+\ve\|\deh\omega\|)\ve c_1 + \frac23\ve\|\deh\omega\|\\%
    &\leq\frac23\ve \pscc1+ \ve\|\omega\|\nc1/3.
  \end{align*}
  Notice that choosing $\pscc1 = 2\|\omega\|\nc1$ and
  $\scc1 =5\pscc1/6$ guarantees that $\|G_{j}'\|\le\ve\scc1$.  Similar arguments can be carried on for the higher
  derivatives; namely:\footnote{ Observe
  that the choices for $\pscc1$ and $\scc1$ only depend on
  $F_{\eps}$, and thus can be absorbed in a constant $\const$.}
  \begin{align*}
    \|G''_j\|&\leq \frac19 \left[\|\bar G''\|+\const\ve\right] \\%
    &\leq \frac19\ve \left[(1+\ve\|\omega\|\nc2)\pscc2 + \const\right]
      < \frac29\ve\pscc2 +\const\ve,%
  \end{align*}
  where in the last inequality we chose $\ve$ small enough.  As
  before, the above inequality implies the existence of $\pscc2$ and
  $\scc2$ that only depend on $F_{\ve}$ so that
  $\|G_{j}''\|\le\ve\scc2$.
 Finally, for the third derivative:
 \begin{align*}
   \|G'''_j\|&\leq \frac2{27}\ve\pscc3 + \const\ve,
  \end{align*}
  from which we gather the existence of $\pscc3$ and $\scc3$
  satisfying the requirements, concluding the proof of item (b).

  In order to prove item (c), let $(\bG_{j})$ denote a partition of
  the image curve into $z'$-standard curves; then we can write:
\begin{align*}
  F_{\ve*}\mu_\ell(g)=\mu_{\ell}(g\circ F_{\ve})&=\int_{I} g(f_\bG(x),\bar
                                                  G(x))\rho(x) \deh x\\&=\int_{J} g(x, \bar
  G(\vf(x)))
  \cdot\rho(\vf(x))\vf'(x)\deh x =\\
                                                &= \sum_j\fmm_{j}\int_{J_{j}} g(x, G_j(x))
                                                  \cdot\rho_j(x)\deh x=%\\&=
  \sum_j \fmm_{j}\mu_{(\bG_{j},\rho_{j})}(g),% Possible mismatched
  \end{align*}
  where $\rho_{j} = \fmm_{j}^{-1} \cdot \rho\circ\vf_{j}\cdot\vf_{j}'$
  and $\fmm_{j} = \int_{J_{j}} \rho(\vf_{j}(x))\vf'_{j}(x)\deh x$.
  Observe that, by construction, we have $\sum_j \fmm_{j}=1$.
  Differentiating the definition of $\rho_{j}$ and
  using~\eqref{e_estimatesVf} we obtain:
  \begin{subequations}\label{e_C2}
    \begin{align}
      \frac{\rho_j'}{\rho_j}&=\frac{\rho'}{\rho\cdot f'_\bG}\circ\vf_j-\frac{f_\bG''}{f_\bG'^2}\circ\vf_j\label{e_c2'}\\
      \frac{\rho_j''}{\rho_j}&=\frac{\rho''}{\rho\cdot
        f'^2_\bG}\circ\vf_j-3\frac{\rho'_j}{\rho_j}\cdot
      \frac{f_\bG''}{f_\bG'^2}\circ\vf_j -\frac{f_\bG'''}{f_\bG'^3}\circ\vf_j.\label{e_c2''}
    \end{align}
  \end{subequations}
% **** Density
  In particular we have:
  \begin{align*}
    \left\|\frac{\rho'_j}{\rho_j}\right\|&\leq \frac13 \rough + \Const&
                                                                        \left\|\frac{\rho''_j}{\rho_j}\right\|&\leq \frac19 (\Dconst+\Const) \rough + \Const.
  \end{align*}
  We can then fix $\Dconst$ so large that the second equation reads
  \begin{align*}
    \left\|\frac{\rho''_j}{\rho_j}\right\|&\le \frac29 \Dconst \rough + \Const
                                            \le \Dconst \left(\frac29 \rough+\const\right).
  \end{align*}
  Choosing $\rough' = 1/3\rough+\Const$ we conclude that
  $\rho_{j}\in \Stdd_{r'}(G_{j})$, which yields item (c) and concludes
  the proof of the proposition.
\end{proof}
\subsection{Standard patches}
\label{sec:standard-patches}
% *** Here I am starting the overhauling
% *** Define raw standard patches
% **** Introductory paragraph
We now begin to introduce the definition of \emph{standard patches} by
defining their support; the support of standard pairs is given by the
class of standard curves, whereas the support of standard patches is
given by a class of sets which we call \emph{standard rectangles} and
that we will define below.  Such rectangles are tubular neighbourhoods
of standard curves along the centre direction; since the natural scale
along the centre direction is $\cO(\ve)$, we expect the natural
smoothness scale of invariant densities along the centre direction
also to be of $\cO(\ve)$.  This is the reason to define standard
rectangles as $\cO(\ve)$-thickening of standard curves along the
centre direction.

% The comment above can be improved
% **** Define standard rectangles
\nomenclature[2aa]{$\patchheight$}{Maximal height of a standard
  rectangle}%
\nomenclature[1__]{$\delta$}{Maximal length of a standard curve}%
\nomenclature[2ab]{$\shor$}{Quantity describing the shortness of a
  standard rectangle}%
\nomenclature[2abb]{$\stdr$}{(Pre)standard rectangle} Let
$\stdr\subset\bT^{2}$ be a compact region diffeomorphic to $[0,1]^{2}$
that is bounded by the union of two prestandard curves (top and
bottom) and two local centre manifolds (left and right).  For any
$p\in\stdr$, we denote with $\Wc_{\stdr}(p)$ the maximal local centre
manifold passing through $p$ and contained in $\stdr$.
Theorem~\ref{thm:unique-integrability-centre-mfolds} implies that,
unless $p$ belongs to the left or right boundary curves,
$\Wc_{\stdr}(p)$ cannot cross them.  Moreover, $\Wc_{\stdr}(p)$ is a
centre curve; since the the centre cone is transverse to the unstable
cone, $\Wc_{\stdr}(p)$ must therefore intersect both the top and
bottom boundary curves at its endpoints.  In particular, by the
definition of prestandard curve, for any $p,q\in \stdr$ we have
\begin{align}\label{eq:difference-in-height-for-Wc}
  |\height \Wc_{\stdr}(p)-\height\Wc_{\stdr}(q)| < 2\pscc1\delta\ve
\end{align}
and also
	\begin{equation}\label{eq:height_of_rect_upper_bd}
	|\pi_2 \stdr|<(\patchheight+2\pscc1\delta)\ve.
  \end{equation}
\begin{mydef}\label{def:std-rectangles}
  Let us fix $\patchheight > 0$, $\minshor \ge 2$ to be determined
  later\footnote{ $\minshor$ will be determined at the end of this
    section, and $\patchheight$ in Section~\ref{sec:coupling}.} and satisfying the following relation
  \begin{align}\label{eq:std-patches-are-thick-enough}
    \minshor^{-1}\patchheight \ge 10\delta\pscc1.
  \end{align}
  For $\narr \in [2,100]$ and $\shor \ge \minshor$, such a region
  $\stdr$ is called a $(\narr,\shor)$-\emph{prestandard rectangle}
  (\resp $(\narr,\shor)$-\emph{standard rectangle}) if:
  \begin{enumerate}
  \item the top and bottom curves are $\narr$-prestandard curves
    (\resp $\narr$-standard curves).
  \item for any $p\in \stdr$, we have:
    $\height\Wc_{\stdr}(p) \in [\patchheight\ve/\shor,
    \patchheight\ve]$.
  \end{enumerate}
\end{mydef}
The value of $\patchheight$ will be determined at the end of
Section~\ref{sec:coupling}.
  % **** Observe some trivial facts about rectangles
We begin by explicitly stating some simple properties of the geometry
of (pre)standard rectangles; the proof is immediate but we write it
down for completeness.
\begin{lem}\label{lem:straight_rect_in_std_rect}
  For sufficiently small $\ve$, the following holds: let $\stdr$ be a
  $(\narr,\shor)$-prestandard rectangle; then there exist intervals
  $I,J\subset\bT$ such that
  $|I|\ge \narr\inv \delta - 2\patchheight \centreConeConst\eps$,
  $|J|\ge (\patchheight/\shor - 2\pscc1 \delta)\eps$,
  $I\times J\subset \stdr$,
  $\sup_{p\in \stdr}d(\pi_1 p,I)\le \patchheight \centreConeConst\eps$
  and $\sup_{p\in \stdr}d(\pi_2 p,J)\le \pscc1 \eps \delta $.
  Moreover:
  \begin{align}\label{eq:bound-on-Lebesgue-patch}
    \frac{\delta\patchheight\ve}{2\narr Z} \le \Leb \stdr \le 2\,\delta\patchheight\ve.
  \end{align}
  An analogous statement holds for standard rectangles (replacing
  $\pscc1$ with $\scc1$).
\end{lem}
\begin{proof}
  First, {note that since $\ve$ is small enough and $z\le 100$ we have
    $\narr\inv \delta - 2\patchheight \centreConeConst\eps > 0$;
    likewise,~\eqref{eq:std-patches-are-thick-enough} implies that
    $\patchheight/\shor - 2\pscc1\delta \ge \frac45
    \patchheight/\shor>0$}.  Write
  $R=[a,b]\times [c,d]= \pi_1\stdr \times \pi_2 \stdr$. Let
  $\Wc_0,\Wc_1$ denote the centre manifolds that bound $\stdr$ on the
  left and right, respectively. Then
  $|\pi_2 \Wc_i|\le \patchheight \eps$ so
  $|\pi_1 \Wc_i|\le \centreConeConst \patchheight \eps$. Since
  $a\in \pi_1 \Wc_0$, $b\in \Wc_1$ we have
  $\pi_1\Wc_0\subset [a,a+\centreConeConst \patchheight \eps]$ and
  $\pi_1\Wc_1 \subset [b-\centreConeConst \patchheight \eps,b]$. Let
  $\bG_0$, $\bG_1$ denote the prestandard curves that bound the bottom
  and top of $\stdr$, respectively. Then
  $|\pi_2\bG_i|\le \pscc1 \eps\delta$ so similary
  $\pi_2 \bG_0 \subset [c,c+\pscc1 \eps\delta]$ and
  $\pi_2 \bG_1 \subset [d-\pscc1 \eps\delta,d]$. Thus
  $U=(a+\centreConeConst \patchheight \eps,b-\centreConeConst
  \patchheight \eps)\times (c+\pscc1 \eps\delta,d-\pscc1 \eps\delta)$
  does not intersect $\partial \stdr$. Since $U$ is connected and
  clearly intersects $\stdr$, it follows that $U\subset \stdr$.  We
  conclude by taking $I$ and $J$ such that $I\times J = \bar U$.  The
  bound on the Lebesgue measure follows immediately by the above
  discussion.
\end{proof}
% **** Define (pre)standard foliated rectangles
\nomenclature[2ad]{$\fstdr$}{(Pre) standard foliated rectangle} We now
refine the above definition by introducing \emph{standard
  foliations} of standard rectangles.  Let $\foli:\stdr\to[0,1]$
be a $C^{3}$-smooth function without critical points; then the
connected components of the level sets of $\foli$ yield a foliation of
$\stdr$.
\begin{mydef}\label{def:std-foliations}
  A function $\foli$ as above is called a \emph{standard foliation of
    the $(\narr,\shor)$-standard rectangle} $\stdr$ if the following
  conditions hold:
  \begin{enumerate}
  \item for any $\folis\in[0,1]$, the level sets
    $\{p\in \stdr\st \foli(p) = \folis\}$ are $z$-standard curves.
  \item a point $p\in\stdr$ belongs to the bottom (\resp top) curve if
    and only if $\foli(p) = 0$ (\resp $\foli(p) = 1$).
  \end{enumerate}
  Similarly, we define $(\narr,\shor)$-prestandard foliations of a
  prestandard rectangle as above, replacing “standard curves” with
  “prestandard curves”.  (Pre)standard rectangles together with a
  (pre)standard foliation are called \emph{(pre)standard foliated
    rectangles} and will be denoted with $\fstdr = (\stdr,\foli)$.
\end{mydef}
We will find convenient to denote with $\bG_{\fstdr,\folis}$
the (pre)stan\-dard curve corresponding to the level set
$\foli = \folis$.
% ***** Observation about $\foli$ parametrizing any centre manifold
Given any prestandard foliated rectangle $\fstdr = (\stdr,\foli)$ and
a maximal centre manifold $\Wc_{\stdr}$, observe that, since standard
curves are transversal to centre manifolds, $\foli$ serves as a
parametrization of $\Wc_{\stdr}$.

% **** Prove “invariance” of standard foliated rectangles
We will now proceed to describe (in
Lemmata~\ref{lem:T-adapted-rects},~\ref{lem:invariance_std_rects}
and~\ref{lem:create-foliation}) the evolution of (pre)standard
foliated rectangles with the dynamics;
% ***** Define m-adapted rectangles
before doing so it is however necessary to introduce a definition.
Standard rectangles will become taller or shorter depending on whether
the centre direction is expanding or contracting; since we need to
keep the height of rectangle below a certain threshold, rectangles
should be cut once they reach a certain height.  In fact we find more
convenient to preemptively cut them if they have the chance to grow
too tall in the near future.  This leads to the definition of
$m$-adapted rectangle given below.  Recall the definition of
$\centreMaxExp$ given above~\eqref{eq:centre-manifold-expansion}.
\begin{mydef}
  Let $m\ge 0$; a $(\narr,\shor)$-(pre)standard foliated rectangle
  $\fstdr = (\stdr,\foli)$ is said to be \emph{$m$-adapted} if, for any
  maximal local centre manifold $\Wc_{\stdr}$, we have:
  \begin{align}\label{eq:short-enough}
	\height\Wc_{\stdr} \le e^{-\centreMaxExp m\eps}\patchheight \ve.
  \end{align}
\end{mydef}
Observe that a $(\narr,\shor)$-(pre)standard rectangle with
$Z < e^{\centreMaxExp m\ve}$ will necessarily fail to be $m$-adapted.
\begin{lem}\label{lem:T-adapted-rects}
  Let $T > 0$ and $\fstdr = (\stdr,\foli)$ be an arbitrary
  $(\narr,\shor)$-(pre)standard foliated rectangle.  If
  $\shor > 4e^{\centreMaxExp T}$, there exists a finite collection
  $\{\fstdr_{j} = (\stdr_{j},\foli_{j})\}_{j\in\mathcal J}$ of
  $(T\vei)$-adapted $(\narr,\shor)$-(pre)standard foliated rectangles
  such that $\bigcup_{j}\stdr_{j} = \stdr$.  Moreover there exists
  affine maps $\psi_{j}:[0,1]\to[0,1]$ so that
  $\psi_{j}\circ\foli_{j} = \foli|_{\stdr_{j}}$.
\end{lem}
\begin{proof}
  We write down the proof in the case of prestandard foliated
  rectangles; the proof for standard foliated rectangles follows by
  an identical argument.  First of all, observe that
  if $\stdr$ is already $(T\vei)$-adapted, we are done; otherwise,
  there exists a centre manifold $\Wc_{\stdr}$ such that
  $\height \Wc_{\stdr} > e^{-\centreMaxExp T}\patchheight \ve$.  We
  partition $ \Wc_{\stdr}$into $N$ subcurves $\Wc_{j}$ of equal
  height, where
  $N = \lceil 2\,\height\Wc_{\stdr}/(e^{-\centreMaxExp T}\patchheight
  \ve)\rceil$, so that
	\begin{align}\label{eq:bound-partitioned-centre}
		\frac13 e^{-\centreMaxExp T}\patchheight\ve\le\height\Wc_{j}\le
		\frac12 e^{-\centreMaxExp T}\patchheight\ve.
	\end{align}
	Let $(\folis_{j})_{j = 0}^{N}$ be the parameters corresponding to
    the endpoints of the subcurves $\Wc_{j}$; for $j = 0,\cdots,N-1$
    define
    $\stdr_{j} = \{p\in\stdr:\foli(p)\in[\folis_{j},\folis_{j+1}]\}$
    and let $\foli_{j}:\stdr_{j}\to[0,1]$ be the following affine
    rescaling of the restriction $\foli|_{\stdr_{j}}$:
	\begin{align*}
		\foli_{j}(p) = \frac{\foli(p)-\folis_{j}}{\folis_{j+1}-\folis_{j}}.
	\end{align*}
    Of course $\bigcup_{j}\stdr_{j} = \stdr$, and choosing
    $\psi_{j}(s) = ({\folis_{j+1}-\folis_{j}})s+\folis_{j}$ yields
    the desired relation between $\foli_{j}$ and $\foli$.  We now show that any
    such $(\stdr_{j},\foli_{j})$ is a $(T\vei)$-adapted
    $(\narr,\shor)$-prestandard foliated rectangle.  First of all, the
    boundary of $\stdr_{j}$ is the union of the two $z$-prestandard
    curves $\bG_{\fstdr,\folis_{j}}$ and $\bG_{\fstdr,\folis_{j+1}}$
    and two centre manifolds (corresponding to the restriction to
    $\stdr_{j}$ of the centre manifolds bounding $\stdr$).  Let
    $\Wc_{\stdr_{j}}$ be an arbitrary maximal local centre manifold in
    $\stdr_{j}$, then by~\eqref{eq:difference-in-height-for-Wc} we
    have $|\height\Wc_{\stdr_{j}}-\height\Wc_{j}| < 2\ve\delta\pscc1$,
    and using~\eqref{eq:bound-partitioned-centre},%
    ~\eqref{eq:std-patches-are-thick-enough} and our assumption on
    $\shor$ we conclude:
	\begin{align}
      \frac14 e^{-\centreMaxExp T}\patchheight\ve\le\height\Wc_{\stdr_j}\le
      \frac34 e^{-\centreMaxExp T}\patchheight\ve.
	\end{align}
    We conclude that each $\stdr_{j}$ is a
    $(\narr,4e^{\centreMaxExp T})$ prestandard rectangle (and thus a
    $(\narr,\shor)$-prestandard rectangle since
    $\shor > 4e^{\centreMaxExp T}$).  Moreover~\eqref{eq:short-enough}
    holds with $m = T\vei$; it is then immediate to check that
    $(\stdr_{j},\foli_{j})$ is a $(\narr,\shor)$-prestandard foliated
    rectangle.
\end{proof}
% ***** Invariance Lemma
\begin{lem}\label{lem:invariance_std_rects}
  Let $m\ge 1$ and $\fstdr = (\stdr,\foli)$ be an $m$-adapted
  $(\narr,\shor)$-prestandard foliated rectangle.  Then, there exists
  a finite collection of $(m-1)$-adapted $(\narr',\shor')$-standard
  foliated rectangles $(\bar\stdr_{j},\bar\foli_{j})$, where
  \begin{align*}
    \narr' &= \max\left\{\frac45 \narr, 2\right\},&%\\
    \shor' &= e^{\centreMaxExp\ve}\shor,
  \end{align*}
  satisfying the following properties:
  \begin{enumerate}
  \item for any $j$ there exists a diffeomorphism
    $\varphi_j:\bar\stdr_{j}\to\stdr$ such that
    $F_{\ve}\circ\varphi_j$ is the identity and
    $\bar\foli_{j} = \foli\circ\varphi_j$.
  \item the rectangles $\bar\stdr_{j}$ form a partition (mod 0) of
    $F_{\ve} \stdr$ and $\varphi_{j}\stdr_{j}$ form a partition (mod
    0) of $\stdr$.
  \end{enumerate}
\end{lem}
\begin{proof}
  To ease notation, for any $\folis\in[0,1]$ we denote with
  $\bG_{\folis}= \bG_{\fstdr,\folis}$; in particular we have
  $\bG_{0} = \bG_{\fstdr,0}$ (and likewise
  $\bG_{1} = \bG_{\fstdr,1}$).
  Proposition~\ref{proposition:invariance}(a) implies that the image
  $F_{\ve}\bG_{0}$ can be partitioned into finitely many trimmed
  $z'$-standard curves, where
  \begin{align*}
    \narr' = \max\left\{\frac45\narr,2\right\};
  \end{align*}
  let us denote such curves by $(\bar\bG_{0,j})_j$, and by $\bG_{0,j}$
  the subcurves of $\bG_{0}$ such that
  $F_{\ve}\bG_{0,j} = \bar\bG_{0,j}$.  Let us now consider the finite
  collection of maximal local centre manifolds contained in $\stdr$
  passing through the endpoints of $\bG_{0,j}$.  By unique
  integrability (Theorem~\ref{thm:unique-integrability-centre-mfolds})
  such manifolds are disjoint and, as noticed earlier, they will
  terminate on the top bounding prestandard curve $\bG_{1}$,
  partitioning it into subcurves that we call $\bG_{1,j}$; such
  subcurves, together with the centre manifolds attached to their
  endpoints, partition (mod 0) $\stdr$ into subsets $\stdr_{j}$.
  Consider $\bar\stdr_{j} = F_{\ve}\stdr_{j}$; since each $\stdr_j$ is
  diffeomorphic to $[0,1]^{2}$ (the left and right boundary centre
  manifolds being disjoint), and $F_{\ve}$ restricted to a
  neighbourhood of $\stdr$ is a diffeomorphism,\footnote{ Observe that
    given our choice of $\delta$ in the proof of
    Proposition~\ref{proposition:invariance}, $F_{\ve}K$ is
    contractible, so since $F_{\ve}$ is a local diffeomorphism, it is
    automatically a diffeomorphism restricted to a neighbourhood of
    $K$} we conclude that each $\bar\stdr_{j}$ is also diffeomorphic
  to $[0,1]^{2}$.  We can thus define
  $\varphi_{j} = F_{\ve}|_{\stdr_{j}}^{-1}$. %
  Notice that, by construction, if $\Wc_{\bar\stdr_{j}}$ is a maximal
  local centre manifold in $\bar\stdr_{j}$, then
  $\varphi_{j}\Wc_{\bar\stdr_{j}}$ is a maximal local centre manifold
  in $\stdr$.  Examining $\bar\stdr_{j}$, we observe that it is
  bounded below by $\bar\bG_{0,j}$, above by the curve
  $\bar\bG_{1,j} = F_{\ve}\bG_{1,j}$ and on the sides by centre
  manifolds (that are the images of the maximal local centre manifolds
  in $\stdr$) Let us now define
  $\bar\foli_{j} = \foli\circ\varphi_{j}$; it is immediate to check
  that the functions $\bar\foli_{j}$ satisfy item (b) in the
  definition of a standard foliation.  In order to check item (a), we
  need to show that each level set of $\foli_{j}$ is indeed a
  $z'$-standard curve.

  Fix $\folis\in[0,1]$ and consider the standard curve
  $\bG_{\folis}\subset \stdr$.  The image $F_{\ve}\bG_{\folis}$ is
  partitioned by the collection $(\bar\stdr_{j})_{j}$ into subcurves
  $\bar\bG_{\folis,j} = F_{\ve}\bG_{\folis}\cap\bar\stdr_{j}$ that
  --by construction-- are the level set corresponding to
  $\bar\foli_{j} = \folis$.  By
  Proposition~\ref{proposition:invariance}(b), such curves are
  $z'$-standard once we show that they are $z'$-curves.

  Denote by $a_{\folis,j}, b_{\folis,j}$ the endpoints of the
  interval $I_{\folis,j} = \pi_{1}\bar\bG_{\folis,j}$; observe that the point
  $\bar\bG_{\folis,j}(a_{\folis,j})$ belongs to the local centre
  manifold connecting $\bar\bG_{0}(a_{0,j})$ and
  $\bar\bG_{1}(a_{1,j})$: let us denote it with $\bar\Wc_{j}$.  By
  Lemma~\ref{lem:expansion-centre-mfolds}(a) we gather that
  $\height\bar\Wc_{j} < 2 \patchheight\ve$, and since centre manifolds
  belong to the centre cone:
  \begin{align*}
    |a_{0,j} - a_{\folis,j}| <2 \centreConeConst \patchheight\ve.
  \end{align*}
Applying the same argument to the other endpoints we conclude that
\begin{align*}
  |b_{0,j}-a_{0,j}| - 4\centreConeConst\patchheight\ve < |b_{\folis,j}-a_{\folis,j}| < |b_{0,j}-a_{0,j}| + 4\centreConeConst\patchheight\ve.
\end{align*}
Since $\bar\bG_{0,j}$ is a \emph{trimmed} $\narr'$-standard curve and
observing that the definition of standard rectangle implies $z \le 100$
(and thus $z'\le 80$) we conclude that, assuming that $\ve$ is so small
that
\begin{align*}
  4\centreConeConst\patchheight\ve < \frac{(e^{\trim}-1)\delta}{80},
\end{align*}
then $\bar\bG_{\folis,j}$ is a
$\narr'$-curve\footnote{\label{fn:justify-trimmed} This part of the
  argument justifies the need for the definition of trimmed curves.}.
This fact concludes the proof that $\bar\foli_{j}$ is a standard
foliation for each $j$.

% **** Prove the bound on the height
We now proceed to bound the height of local centre manifolds.  Let
$\Wc_{\bar\stdr_{j}}$ be a maximal local centre manifold in
$\bar\stdr_{j}$; in particular $\varphi_{j}\Wc_{\bar\stdr_{j}}$ is a
maximal local centre manifold $\Wc\subset\stdr$.  Since $\stdr$ is an
$m$-adapted $(\narr,\shor)$-prestandard rectangle, we have
\begin{align*}
  Z^{-1}\patchheight\ve\le \height\Wc\le e^{-\centreMaxExp m\ve}\patchheight\ve
\end{align*}
Using the definition of $\shor'$ and
Lemma~\ref{lem:expansion-centre-mfolds}(a) we can thus conclude that:
\begin{align*}
  \shor'^{-1}\patchheight\ve = %
  \shor^{-1}e^{-\centreMaxExp\ve}\patchheight\ve\le%
  \height \Wc_{\bar\stdr_{j}}\le%
  e^{\centreMaxExp\ve}  e^{-\centreMaxExp m\ve}\patchheight\ve = %
  e^{-\centreMaxExp (m-1)\ve}\patchheight\ve.
\end{align*}
This implies that $\bar\stdr_{j}$ is a
$(\narr',\shor')$-standard foliated rectangle, and concludes the proof of
the lemma.
\end{proof}
% **** Prove standard rectangles → prestandard foliated rectangles
Given an arbitrary standard rectangle $\stdr$, it is not necessarily
true that $\stdr$ can be foliated into a foliated standard rectangle.
However –as proved in the next lemma– it is always possible to
construct a prestandard foliation for $\stdr$.  Recall the conventions
for the differential and Hessian operators outlined in
Section~\ref{sec:conventions}.
\begin{lem}\label{lem:create-foliation}
  Let $\stdr$ be a $(\narr,\shor)$-standard rectangle.  Then $\stdr$
  can be partitioned into foliated $(3\narr,\shor)$-prestandard
  rectangles $\fstdr_{j} = (\stdr_{j},\foli_{j})$ that satisfy the
  following estimates:
  \begin{align*}
  	\|d\log\partial_{\theta}\foli\|_{\infty}&\le \Const\shor&%\\
  	\|H\log\partial_{\theta}\foli\|_{\infty}&\le \Const(\shor+\shor^2).
  \end{align*}
  \end{lem}
\begin{proof}
  Let us denote with $\mathbb G_i(x)=(x,G_i(x)), x \in [a_i,b_i]$,
  $i=0,1$, the prestandard curves bounding $\stdr$ from below and
  above, respectively; let $[a,b] = \pi_{1}\stdr$.  Observe that by
  definition of the centre cone and by item (b) in the definition of
  prestandard rectangle, we can conclude that
  $|a-a_i| < \centreConeConst\patchheight\ve$ and
  $|b-b_{i}| < \centreConeConst \patchheight\ve$ for any $i = 0,1$.

  Let us assume first that $|b-a| < \delta$; then, we extend
  $\mathbb G_i,\, i=0,1$ to a $C^{3}$-smooth function on $[a,b]$ using
  the Taylor polynomial of order 3; for instance, for $x\in(b_i,b)$,
  we set:
  \begin{align*}
    \tilde G_i(x) = \sum_{k = 0}^{3}\frac1{k!}G^{(k)}_{i}(b_i)(x-b_i)^{k}.
  \end{align*}
  Since $|b-b_{i}| < \centreConeConst \patchheight\ve$, we can ensure --by
  choosing $\ve$ sufficiently small-- that the extensions
  $\tilde{{\mathbb G}}_i=(x, \tilde G_i(x)),\, i=0,1$ are prestandard.
  Now, given $p = (x,\theta)\in\stdr$, we define
  \begin{align*}
    \foli(p) &= \frac{\theta-\tilde G_{0}(x)}%
    {\tilde G_{1}(x)-\tilde G_{0}(x)}.
  \end{align*}
  Observe that $\foli$ is $C^{3}$ and has no critical points; for any
  $\folis \in [0,1]$, the level sets $\foli(p) = \folis$ are the
  graphs of the functions
  \begin{equation}\label{NewLeaf}
    \tilde G_{\folis}(x) = \folis\tilde G_1(x)+(1-\folis)\tilde G_0(x).
  \end{equation}
  Since $\tilde{{\mathbb G}}_0, \tilde{{\mathbb G}}_1$ are
  prestandard,
  $\|\tilde G'_i \|\leq \eps \pscc1, \|\tilde G''_i\|\leq \eps \pscc2$
  and $\|\tilde G'''_i\|\leq \eps \pscc3$ for
  $i=0,1$. Equation~\eqref{NewLeaf} implies that $\tilde \bG_{\folis}$
  is prestandard for any $\folis\in[0,1]$.  The length of each leaf is
  at least
  $\narr^{-1}\delta-2\centreConeConst\patchheight\ve >
  (3\narr)^{-1}\delta$ assuming $\ve$ to be small enough.  By the
  above observations we conclude that $\foli$ is a foliation of
  $\stdr$ into $3z$-prestandard curves.

  We now turn to bounding the norms
  $\|d \log \partial_\theta \foli\|_{\infty}$ and
  $\|H \log \partial_\theta \foli\|_{\infty}$.  Note that
  $\log \partial_\theta \foli = -\log(\tilde G_1 - \tilde G_0)$, that
  is independent on $\theta$.  It therefore suffices to consider
  $\partial_x \log \partial_\theta \foli$ and
  $\partial_{xx} \log \partial_\theta \foli$.  Elementary calculus
  yields
  \begin{align*}
  	\partial_x \log \partial_\theta \foli = -\frac{\tilde G'_1 - \tilde G'_0}{\tilde G_1 - \tilde G_0}, \qquad \partial_{xx}\log \partial_\theta \foli = \left(\frac{\tilde G'_1 - \tilde G'_0}{\tilde G_1 - \tilde G_0}\right)^2 - \frac{\tilde G''_1 - \tilde G''_0}{\tilde G_1 - \tilde G_0} .
  \end{align*}
  To proceed, we find a lower bound on $|\tilde G_1 - \tilde G_0|$. Let $p=(x,\theta)\in K$ and for $i=0,1$ let $p_i = \tilde\bG_{i}(x_i)$ denote the point where $\Wc_{\stdr}(p)$ intersects the graph of $\tilde\bG_i$. Since $\tilde\bG_{i}$ is prestandard and $\Wc_{\stdr}(p)$ is a centre curve, we have
  \begin{align*}
  	|\tilde G_i(x_i) - \tilde G_i(x)|\le \eps\pscc1 |x_i-x|\le \eps\pscc1\centreConeConst \ \height\Wc_{\stdr}(p).
  \end{align*}
  Since $\height\Wc_{\stdr}(p)=|\tilde G_1(x_1)-\tilde G_0(x_0)|$, it follows that
  \[ |\tilde G_1(x) - \tilde G_0(x)|\ge (1-2\eps\pscc1\centreConeConst)\height\Wc_{\stdr}(p)\ge (1-2\eps\pscc1\centreConeConst)\shor\inv\patchheight\ve. \]
  Now $\|\tilde G'_1 - \tilde G'_0\|\le 2\eps\pscc1 $ and $\|\tilde G''_1 - \tilde G''_0\|\le 2\eps\pscc2 $. Thus for $\eps$ sufficiently small we obtain that $\|\partial_x \log \partial_\theta\foli\|\le \Const \shor$ and $\|\partial_{xx} \log \partial_\theta\foli\|\le \Const(\shor^2 + \shor)$, as required. This concludes the proof of the lemma in the
  case where $|b-a| < \delta$.

  If, on the other hand, $|b-a|\ge\delta$ we split $\stdr$ in two
  narrower subpatches $\stdr'$ and $\stdr''$.  The splitting can be
  done in many ways: for instance, $\stdr$ can be cut along the centre
  manifold passing through the mid-point of $\bG_{0}$ into the left
  subrectangle $\stdr'$ and the right subrectangle $\stdr''$.  We us
  obtain two $(5\narr/2,\shor)$-prestandard rectangles whose
  projection is narrower than $\delta$.  We can thus apply the above
  argument to each of the sub-rectangles; notice that in this case the
  lower bound on the length of each leaf is
  $2/(5\narr)\delta-2\centreConeConst\patchheight\ve$, that is still larger
  than $(3\narr)^{-1}\delta$ for sufficiently small $\ve$.

\end{proof}
% ***** bit about the densities - to  be relocated
% **** Introduce densities on std rectangles
We now introduce the notion of a standard density on a (pre)standard
rectangle $\stdr$: in order to do that, we find convenient to
introduce a change of variables that scales standard rectangles to be
of size $O(1)$.  Define the auxiliary transformation:
$\chvar:\T\times\bR/(\vei\bZ) \to \T^2$ by $\chvar(x,y) = (x,\eps
y)$. %
\nomenclature[z1]{$\chvar$}{Change of variable stretching the slow
  variable by $\vei$}%
Let $\stdr\subset \bT^{2}$ be a (pre)standard rectangle; fix
$\Econst > 0$ to be specified later and let $\Rough > 0$; %
\nomenclature[2ac]{$\Econst$}{Constant involved in the definition of a
  standard density on a (pre)standard rectangle} \nomenclature[2ad]{$\Rough$}{Quantity
  describing the roughness of a standard density on a (pre)standard rectangle} we define the the set of
$\Rough$-\emph{standard probability densities on $\stdr$} as follows:
\begin{align*}
  \Stdd_{\Rough}(\stdr) = \{\rho\in\cC^{2}(\stdr,\reals_{ > 0}) \st
  \|\rho\|_{L^{1}} = 1, &\|d(\log\rho\circ\chvar)\|_{\infty} < \Rough,\\&
  \|H(\log\rho\circ \chvar)\|_{\infty} < \Econst  \Rough\}.\
\end{align*}
% **** Observation about roughness
First, we state a simple fact about standard densities which will be
useful in the sequel
\begin{lem}\label{lem:split_density_leb}
  Let $\stdr$ be a $(\narr,\shor)$-prestandard rectangle and
  $\rho\in\Stdd_{\Rough}(\stdr)$ be a $\Rough$-standard density.
  Assuming $\patchheight$ to be large enough, for any $p \in \stdr$,
	\begin{equation}\label{eq:patch_density_bds}
      \frac{\exp(-2\patchheight \Rough)}{\Leb(\stdr)}\le \rho(p)\le \frac{\exp(2\patchheight \Rough)}{\Leb(\stdr)}.
	\end{equation}
	Moreover, we can write
    $\rho = \tau \frac{1}{\Leb(\stdr)}+(1-\tau)\tilde \rho$ where
    $\tau = \frac12 \exp(-2\patchheight \Rough)$ and
    $\tilde \rho \in \Stdd_{\Rough'}(\stdr)$ with
    $\Rough' = 2\Rough + 4\Econst\inv\Rough^2$.
  \end{lem}
\begin{proof}
  The bound on $\|d(\log \rho\circ\chvar)\|_{\infty}$ implies that for
  any $p, p'\in\stdr$ we have:
  \begin{align*}
    \left|\log\frac{\rho(p)}{\rho(p')}\right|\le
    R\cdot\textrm\diam(\chvar\inv \stdr).
  \end{align*}
  By the Intermediate Value Theorem, there exists $p'\in\stdr$ so that
  $\rho({p'}) = 1/\Leb\stdr$, from which we conclude, taking
  $\patchheight$ to be large enough, that~\eqref{eq:patch_density_bds}
  holds.  The bound on $\Rough'$ follows
  from~\eqref{eq:patch_density_bds} and elementary calculus.
\end{proof}

% **** Define foliated standard patches
Let $\fstdr = (\stdr,\foli)$ be a foliated $(\narr,\shor)$-prestandard
rectangle and $\rho\in\Stdd_{\Rough}(\stdr)$ be a $\Rough$-standard
probability density on $\stdr$.  We can disintegrate the probability
measure induced by\ $\rho$ along the foliation $\foli$ as follows: for
any continuous function $g:\bT^{2}\to\bR$, we have:
\begin{align*}
\int_{\stdr}g(x,\theta)\rho(x,\theta)dxd\theta
  &=\int_{0}^{1}\nu_{\folis}\left[\int_{I_{\folis}}\rho_{\folis}(x)g(\bG_{\fstdr,\folis}(x))dx \right]d\folis,
\end{align*}
where $I_{\folis} = \pi_{1}(\bG_{\fstdr,\folis})$ and
\begin{align}\label{eq:induced-density}
  \rho_{\folis}(x) &=
                   \frac{\rho(\bG_{\fstdr,\folis}(x))\frac{\partial}{\partial\folis}G_{\fstdr,
                   \folis}(x)}{\nu_{\folis}}, &%\\
  \nu_{\folis} &= \int_{I_{\folis}} \rho(\bG_{\fstdr,\folis}(x))\frac{\partial}{\partial\folis}G_{\fstdr,\folis}(x) dx.
\end{align}
It is natural to explore the relation between the roughness $\Rough$
of the density $\rho$ and the roughness $\rough$ of the disintegration
of the associated measure along the foliation $\foli$.  We pursue this
task in the lemma below.
\begin{lem}\label{lem:disintegration-regularity}
  Let $\fstdr = (\stdr,\foli)$ be a foliated
  $(\narr,\shor)$-prestandard rectangle and
  $\rho\in\Stdd_{\Rough}(\stdr)$ be a $\Rough$-standard density; then
  for any $\folis\in[0,1]$, $\rho_{\folis}$ is an
  $\rough'$-standard density on $G_{\fstdr,\folis}$, where
  \begin{align*}
    \rough' &= \Const \left(\left[\Rough + \|d\log\partial_{\theta}\foli\|_{\infty}\right]^2%
              + \Rough +\|d\log\partial_{\theta}\foli\|_{\infty}+%
              \|H\log\partial_{\theta}\foli\|_{\infty}\right).
  \end{align*}
\end{lem}

% For the horizontal foliation we have
% $\frac\partial{\partial \folis}G_{\folis} = 1/\Delta\ve$
\begin{proof}
			Since $\fstdr$ is fixed, let us use the shorthand notation
$\bG_{\folis} =\bG_{\fstdr,\folis}$ (and similarly for $G$).  Since
$\bG_{\folis}$ is a level set for the foliation $\foli$, we have
$\frac{\partial}{\partial \folis}G_{\folis} = (\partial_\theta \foli \circ
\bG_{\folis})\inv$. Thus by the above formula for $\rho_{\folis}$
\begin{equation}\label{eq:log_rho_eta}
	\log \rho_{\folis} = \log \rho\circ\bG_{\folis} - \log \partial_\theta \foli \circ \bG_{\folis}	- \log \nu_{\folis}.
\end{equation}
Define
$k_{\folis}(x) = \chvar\inv(\bG_{\folis}(x)) = (x,\eps\inv
G_{\folis}(x))$. Since $G_{\folis}$ is a prestandard curve,
$\|G'_{\folis} \|\leq \eps \pscc1$ and
$\|G''_{\folis}\|\leq \eps \pscc2$. It follows that
$\|k_{\folis}\|_{\cC^2}\le \Const$, hence:
\begin{align*}
  \|(\log \rho \circ \bG_{\folis})' \|_{\infty} &\le%\\
  \|d(\log \rho \circ \chvar)\|_{\infty} \|k'_{\folis}\|_{\infty}\le \Const \Rough.
\end{align*}
Moreover, by Lemma~\ref{lem:hessian_chain_rule},
\begin{align*}
  \|(\log \rho \circ \bG_{\folis})''\|_{\infty} &%\\
  \le \|H(\log \rho \circ \chvar)\|_{\infty}\|k'_{\folis}\|_{\infty}^2 +%\\
  \|d(\log \rho \circ \chvar)\|_{\infty} \|k''_{\folis} \|_{\infty}\\
  &\le \Const\Econst\Rough  + \Const\Rough .
\end{align*}
Similarly, since $\|\bG'_{\folis}\|\le \Const$ and
$\|\bG''_{\folis}\|\le \Const$, we obtain, using again
Lemma~\ref{lem:hessian_chain_rule}, that:
\begin{align*}
	\|(\log \partial_\theta \foli\circ \bG_{\folis})'\|_{\infty}&\le \Const \|d\log \partial_\theta\foli\|_{\infty},\\
	\|(\log \partial_\theta \foli\circ \bG_{\folis})''\|_{\infty}&\le \Const (\|H\log \partial_\theta\foli\|_{\infty}+\|d\log \partial_\theta\foli\|_{\infty})
\end{align*}
By combining the above inequalities with~\eqref{eq:log_rho_eta}, it
follows that
\begin{align*}
  \|(\log \rho_{\folis})'\|_{\infty}&\le \Const (\Rough+\|d\log \partial_\theta\foli\|_{\infty}),\\
  \|(\log \rho_{\folis})''\|_{\infty}&\le \Const(\Rough+\|H\log \partial_\theta\foli\|_{\infty}+\|d\log \partial_\theta\foli\|_{\infty}).
\end{align*}
Finally, the proof of the lemma follows by noting that $\rho'_{\folis}/\rho_{\folis}=(\log\rho_{\folis})'$ and $\rho''_{\folis}/\rho_{\folis} = (\log\rho_{\folis})'' + [(\log \rho_{\folis})']^2$.
\end{proof}
The above lemma justifies the following definition:
\begin{mydef}
  Let $\fstdr = (\stdr, \foli)$ be a prestandard foliated rectangle
  and $\rho$ be a standard density on $\stdr$.  The object
  $\plaque = (\fstdr,\rho)$ is called a
  $((\narr,\shor),(\rough,\Rough))$-\emph{prestandard patch} if
  $\fstdr$ is a $(\narr,\shor)$-prestandard foliated rectangle, $\rho$
  is a $\Rough$-standard density on $\stdr$ and the disintegration
  along the foliation $\foli$ of the measure induced by $\rho$ is
  $r$-standard on every leaf.  \emph{Standard patches} are defined as
  above, replacing “prestandard” with “standard”.
\end{mydef}
\nomenclature[2b]{$\plaque$}{standard patch}
\ignore{ We denote with $\plaqueSet_{(\narr,\shor)}^{(\rough,\Rough)}$
  (\resp $\pplaqueSet_{(\narr,\shor),(\rough,\Rough)}$) the set of
  $((\narr,\shor),(\rough,\Rough))$-standard (\resp -prestandard)
  patches}

A standard patch $\plaque$ induces a Borel probability measure on
$\bT^{2}$ given by:
\begin{align*}
  \mu_{\plaque}(g) = \int_{\stdr}g(x,\theta)\rho(x,\theta)dxd\theta;
\end{align*}
henceforth we abuse notation by writing $\supp \plaque$ instead of
writing $\supp \mu_{\plaque}$.
\ignore{We say that a probability measure $\mu$ \emph{admits a
  $((\narr,\shor),(\rough,\Rough))$-standard patch representation} if
there exists $((\narr,\shor),(\rough,\Rough))$-standard patch
$\plaque$ such that $\mu = \mu_{\plaque}$.}
\begin{rmk}\label{rmk:standard-patches-are-families}
  Let $\plaque$ be an arbitrary
  $((\narr,\shor),(\rough,\Rough))$-standard patch , then
  $\mu_{\plaque}$ naturally admits a representation as a family of
  $(\narr,\rough)$ standard pairs (see
  Lemma~\ref{lem:disintegration-regularity})
\end{rmk}
% **** Observation about cutting
A patch $\plaque = (\fstdr,\rho)$ is said to be $m$-adapted if so is
$\fstdr$.
\begin{lem}\label{lem:cutting-patches}
  Let $\plaque = (\fstdr,\rho)$ be a
  $((\narr,\shor),(\rough,\Rough))$-standard patch with
  $\shor > 4e^{\centreMaxExp m\ve}$ and let
  $\{\mathcal\stdr_{j} = (\stdr_{j},\foli_{j})\}$ be the collection of
  $m$-adapted foliated standard rectangles obtained by applying
  Lemma~\ref{lem:T-adapted-rects} to $\mathcal \stdr$.  Let
  $c_{j} = \mu_{\plaque}(\stdr_{j})$,
  $\rho_{j} = \frac{\rho|_{\stdr_{j}}}{c_{j}}$ and
  $\plaque_{j} = (\fstdr_{j}, \rho_{j})$.  Then each $\plaque_{j}$ is
  a $((\narr,\shor),(\rough,\Rough))$ $m$-adapted patch and
\begin{align*}
    \mu_{\plaque} = \sum_{j}c_{j}\mu_{\plaque_{j}}.
  \end{align*}
  \end{lem}
\begin{proof}
  The fact that $\mu_{\plaque} = \sum_{j}c_{j}\mu_{\plaque_{j}}$ is
  immediate by construction.  The only thing to check is the
  regularity of the disintegration of $\plaque_{j}$ along the
  foliation $\foli_{j}$.  Since $\foli_{j}$ and $\foli$ are related by
  an affine transformation (\ie
  $\psi_{j}\circ \foli_{j} = \foli|_{\stdr_{j}}$),
  inspecting~\eqref{eq:induced-density} we conclude that the
  disintegrated density on any leaf $\bG_{\fstdr_{j},\folis}$ of
  $\fstdr_{j}$ equals the disintegrated density on the corresponding
  leaf $\bG_{\fstdr,\psi_{j}\folis}$ of $\fstdr$.  Since the latter
  belongs to $\Stdd_{r}(G_{{\fstdr,\psi_{j}\folis}})$ by assumption,
  we conclude that the former belongs to
  $\Stdd_{r}(G_{\fstdr_{j},\folis})$, hence each $\plaque_{j}$ is
  $((\narr,\shor),(\rough,\Rough))$-standard.
\end{proof}
% TODO(think): define the measure and eq classes and decomp in std patches
% **** Prove invariance of the class of standard densities
% do we actually need it? well yes; we need to check it one way or the
% other, and it is probably a good idea to separate the arguments
The proposition that follows states invariance properties of the class
of standard patches.  We now fix a timescale $\clock > 0$ to be
specified later and, from now on, we denote
$\clockN = \lfloor \clock\vei\rfloor$.
\nomenclature[29a]{$\clock$}{number of timesteps of order $\vei$} The
timescale $\clock$ will be chosen to be long enough with respect to
some features of the averaged dynamics concerning both the slow
variable $\theta$ (see Section~\ref{sec:averaging}) and the centre
foliation (see Section~\ref{sec:patch-families}) and its value will be
determined right after Lemma~\ref{lem:patch_family_a_priori_bd}.  For
definiteness, we will regard $\clock$ as a natural \emph{clock} for
our system, and we will perform all our manipulations at a number of
iterates that are multiples of $\clockN$.

%{\todo[inline]{The natural way to state this lemma is about
%    disintegration of the pushforward into zZrR - std patches}}
\begin{prp}[Dynamics of standard
  patches]\label{lem:invariance_patches}
  Assume that $\minshor$ in the definition of prestandard rectangles
  and $\Econst$ in the definition of prestandard patches are
  sufficiently large (depending on the fixed $\clock$), then the
  following holds for sufficiently small $\ve$. There exists
  $\Rough_{*} > 0$ such that if $\plaque$ is a
  $((\narr,\shor),({\rough,\Rough}))$ prestandard patch, for any
  $n\le \clockN$:
	\begin{enumerate}
    \item there exists a finite collection $\{\bplaque_{n,j}\}$ of
      $((\narr_{n},\shor_{n}),(\rough_{n},\Rough_{n}))$-standard
      patches such that $F_{\eps*}^n \mu_{\plaque}$ is a convex
      combination $\sum_{j}c_{n,j}\mu_{\bplaque_{n,j}}$ and:
      \begin{align}\label{eq:std_patch_inv_weak_bds}
        \narr_{n} &= \max\left\{\left(\frac45\right)^{n}\narr, 2\right\},&%\\
        \shor_{n} &= e^{\centreMaxExp n\ve}\shor\\
        \rough_{n} &= \left(\frac13\right)^{n}\rough+\frac32\left(1-3^{-n}\right) \rough_{*},&%\\
        \Rough_{n} &= C_{\clock}\Rough+\Rough_{*}.\nonumber
      \end{align}
      In addition, for each $j$ there exists a diffeomorphism
      $\varphi_{n,j}:\supp{\bplaque_{n,j}}\to\stdr$ such that
      $F_{\ve}^{n}\circ\varphi_{n,j}$ is the identity,
      $\{\varphi_{n,j}(\supp{\bplaque_{n,j}})\}_{j}$ forms a
      partition (mod 0) of $\stdr$; finally,
      $\mu_{\plaque}|_{\varphi_{n,j}(\supp{\bplaque_{n,j}})} =
      c_{n,j}\,\varphi_{n,j*}\mu_{\bplaque_{n,j}}$ (in particular this
      implies that
      $c_{n,j} = \mu_{\plaque}(\varphi_{n,j}(\supp{\bplaque_{n,j}}))$).
    \item Moreover, each $\bplaque_{n,j}$ is indeed a
      $((\narr_{n},\shor_{n,j}),(\rough_{n},\Rough_{n,j}))$-standard
      patch, where $\shor_{n,j} = \max\{M_{n,j} \shor,\minshor\}$,
      $\Rough_{n,j} = \max\{M_{n,j}, M_{n,j}^2 + \frac{1}{4}\} \Rough
      + \Rough_{*}$ and
      \begin{align*}
        M_{n,j} = \exp\bigg(-\sup_{\supp{\bplaque_{n,j}}}\zeta_{n}\circ\varphi_{n,j}+C_{\clock}\ve\log\vei
        + \clock\fakeCentreConst + D\bigg),
      \end{align*}
      where $\fakeCentreConst$ was defined above~\eqref{eq:zeta_n_def}
      and $D>0$ is a constant (that can be chosen uniformly in
      $\clock$).
	\end{enumerate}
\end{prp}
\begin{proof}
  Let $\plaque = (\fstdr,\rho)$: we will assume that $\plaque$ is
  $\clockN$-adapted; otherwise Lemma~\ref{lem:cutting-patches}
  guarantees that we can cut $\plaque$ in $\clockN$-adapted patches
  provided that $\minshor > 4e^{\centreMaxExp\clock}$.

  We first prove part (a), apart from the bound on $\Rough_n$, which we establish separately later. It is convenient to
  prove the following stronger statement:
  \begin{lem}\label{sublem:dynamics-standard-patches}
    If $\plaque$ is $m$-adapted, then for any $n \le m$, there exists
    a finite collection
    $\{\bplaque_{n,j} =
    ((\bar\stdr_{n,j},\bar\foli_{n,j}),\bar\rho_{n,j})\}_{j\in\mathcal
      J}$ of $(m-n)$-adapted
    $((\narr_{n},\shor_{n}),(\rough_{n},\Rough^{*}_{n}))$-standard
    patches such that $F_{\ve*}^{n}\mu_{\plaque}$ is a convex
    combination $\sum_{j}c_{n,j}\mu_{\bplaque_{n,j}}$, and $\narr_n$,
    $\shor_n$ and $\rough_n$ are as defined
    in~\eqref{eq:std_patch_inv_weak_bds} and $\Rough_n^{*}<\infty$.
    In addition, for each $j$, there exists a diffeomorphism
    $\varphi_{n,j}:\bar\stdr_{n,j}\to\stdr$ such that
    $F_{\ve}^{n}\circ\varphi_{n,j}$ is the identity and
    $\{\varphi_{n,j}\bar\stdr_{n,j}\}_{j}$ forms a partition (mod 0)
    of $\stdr$. Moreover, $\bar\foli_{n,j}=\foli\circ\varphi_{n,j}$
    and
    $\bar\rho_{n,j} = c_{n,j}^{-1}\cdot\rho\circ\varphi_{n,j}\cdot\det
    d\varphi_{n,j}$, where
    $c_{n,j} =\mu_{\plaque}\,\varphi_{n,j}\bar\stdr_{n,j}$.
  \end{lem}
  \begin{proof}
    We consider the case where $n=1$; the general case can be
    obtained by iterating the argument.  Let
    $(\bar\stdr_j,\bar\foli_j)$ be foliated standard rectangles and
    $\varphi_j:\bar\stdr_j\to \stdr$ be maps that satisfy the
    properties specified in
    Lemma~\ref{lem:invariance_std_rects}. Since
    $\{\varphi_j \bar\stdr_j \}$ partitions $\stdr$ (mod 0), for any
    continuous function $g:\bT^2\to \R$ we have, changing variable:
    \begin{align*}
      F_{\eps*}\mu_{\plaque}(g) &= \sum_j c_j\int_{\bar\stdr_j} g(x,\theta)\bar\rho_j(x,\theta)dx d\theta,
    \end{align*}
    where $c_j =\mu_{\plaque}(\varphi_{j}\bar\stdr_{j})$ and
    $\bar\rho_j =c^{-1}_{j}\cdot \rho \circ \varphi_{j}\cdot\det d\varphi_{j}.$ Hence setting
    $\bplaque_j = ((\bar\stdr_j, \bar\foli_{j}),\bar\rho_j)$ for each
    $j$ yields that
    $F_{\eps*}\mu_{\plaque}=\sum_j c_j \mu_{\bplaque_j}$.

    The foliated rectangles $(\bar\stdr_{j},\bar\foli_{j})$ are, by
    Lemma~\ref{lem:invariance_std_rects}, automatically
    $(m-1)$-adapted $(\narr_{1}, \shor_{1})$-standard rectangles.  It
    remains to bound $\rough_1$. Fix $\folis\in[0,1]$ and consider the
    standard pair
    $ \ell_{\folis} = (\bG_{\fstdr,\folis}, \rho_{\folis}) $.
    Lemma~\ref{lem:invariance_std_rects}(a) further implies that
    the collection of curves $(\bG_{\bar{\fstdr}_{j},\folis})$ forms
    (mod 0) a partition of $F_{\ve}\bG_{\fstdr,\folis}$ into
    $\narr_{1}$-standard curves. Hence by
    Proposition~\ref{proposition:invariance}(c), for each $j$ we can
    write
    $(F_{\eps*} \mu_{\ell_{\folis}})|_{\bar\stdr_j} =
    \alpha_{j,\folis}\mu_{\bar{\ell}_{j,\folis}}$, where
    $\bar{\ell}_{j,\folis}=(\bG_{\bar{\fstdr}_{j},\folis}\,
    ,\breve\rho_{j,\folis})$ is a $(\narr_1,\rough')$-standard pair
    with $\rough' = \frac13 \rough + \stdrough$ and
    $\alpha_{j,\folis} =
    \mu_{\ell_{\folis}}(\varphi_{j}{\bar\stdr_{j}})$.

    We now need to show
    that $\breve\rho_{j,\folis}$ coincides with the density
    $\bar\rho_{j,\folis}$ on
    $I_{j,\folis}=\pi_1 \bG_{\bar{\fstdr}_{j},\folis}$ obtained by
    disintegrating $\bar{\rho}_j$ along the foliation $\bar{\foli}_j$
    for $\bar{\foli}_{j} = \folis$.  This is a simple consequence of
    the change-of-variable formula, and can be checked with the
    following argument: for any continuous function
    $g: \bar{\stdr}_j\to \R$,
    \begin{align*}
      F_{\eps*}\mu_{\plaque}(g) &= %\\
      \int_0^1 \nu_{\folis} F_{\eps*}\mu_{\ell_{\folis}}(g)d\folis%\\
      = \int_0^1 \nu_{\folis} \alpha_{j,\folis} \left[\int_{I_{j,\folis}}\breve{\rho}_{j,\folis}(x)g(\bG_{\fstdr_j,\folis}(x))dx \right]d\folis\\
                                &= c_j \mu_{\bar{\plaque}_j}(g)= c_j\int_{0}^{1}\bar\nu_{j,\folis}\left[\int_{I_{j,\folis}}\bar\rho_{j,\folis}(x)g(\bG_{\fstdr_j,\folis}(x))dx \right]d\folis,
    \end{align*}
    where
    $\bar\nu_{j,\folis}=\int_{I_{j,\folis}}
    \bar\rho_j(\bG_{\fstdr_j,\folis}(x))\frac{\partial}{\partial\folis}G_{\fstdr_j,\folis}(x)
    dx$. Since $\bar\rho_{j,\folis}$ and $\breve{\rho}_{j,\folis}$ are
    both continuous probability densities that depend continuously on
    $\folis$, and the weights $\nu_{\folis}$, $\alpha_{j,\folis}$ and
    $\bar\nu_{j,\folis}$ depend continuously on $\folis$, it follows
    that $\rho_{j,\folis} = \breve{\rho}_{j,\folis}$ for all
    $\folis\in [0,1]$, as claimed.  This concludes the proof.
  \end{proof}
  In the remainder of this proof, we let $\bplaque_{n,j}$ be as
  defined in the above lemma.  Next we prove the bound on
  $\shor_{n,j}$ stated in (b). Let $\Wc_{\bar\stdr_{n,j}}$ be a
  maximal local centre manifold in $\bar\stdr_{n,j}$.  Since
  $\bar\foli_{n,j}=\foli \circ \varphi_{n,j}$, we conclude that
  $\varphi_{n,j}$ maps the top and bottom curves of
  $\bar{\stdr}_{n,j}$ to (a subcurve of) the top and bottom curves of
  $\stdr$, respectively. It follows that
  $\Wc = \varphi_{n,j}\Wc_{\bar\stdr_{n,j}}$ is a maximal local centre
  manifold in $\stdr$. Hence by
  Lemma~\ref{lem:expansion-centre-mfolds}(b),
	\begin{align}
		\height\Wc_{\bar{\stdr}_{n,j}}&\ge\left( \inf_{\Wc}\prod_{k=0}^{n-1}\CentreExpansion\circ F_\eps^k\right)\, \height\Wc\label{eq:new_height_old_height_comp}\\
		&\ge \exp\bigg(\sup_{\Wc}\zeta_n-C_{\clock}(\height\Wc+\eps\log\eps^{-1})-\clock\fakeCentreConst\bigg)\,\height\Wc.\nonumber
	\end{align}
	Since $\stdr$ is a prestandard rectangle,
	\begin{equation}\label{eq:upper_height_bd_absorbed}
		\height\Wc+\eps \log\vei\le \patchheight \eps +\eps\log \vei\le 2 \eps\log\vei.
	\end{equation}
	for $\eps$ sufficiently small.

	Now let $q \in \bar\stdr_{n,j}$ and let
    $\bG = \bG_{\bar{\fstdr}_{n,j},\bar\foli_{n,j}(q)} \subset
    \bar\stdr_{n,j}$ be the standard curve that contains $q$.  By a
    standard distortion estimate,
    $|\zeta_n(\varphi_{n,j}(q))-\zeta_n(p)|\le C_{\clock} \eps$ for any
    $p\in \varphi_{n,j}(\bG)$. Since $\Wc$ intersects
    $\varphi_{n,j}(\bG)$, it follows that
	\begin{equation}\label{eq:distortion_preimage_rectangles}
		\sup_{\Wc}\zeta_n \ge \sup_{q \in \bar{\stdr}_{n,j}}\zeta_n (\varphi_{n,j}(q))- C_{\clock} \eps.
	\end{equation}
	The required bound on $\shor_{n,j}$ then follows by
    combining~\eqref{eq:new_height_old_height_comp},
    \eqref{eq:upper_height_bd_absorbed}
    and~\eqref{eq:distortion_preimage_rectangles}.

	Finally, we prove the bounds on $\Rough_n$ and $\Rough_{n,j}$
    stated in parts (a) and (b).  By Lemma~\ref{sublem:dynamics-standard-patches}, for fixed $n$
    and $j$ we have that
    $\bar\rho_{n,j} = c^{-1}_{n,j}\cdot\rho\circ \varphi_{n,j}\det
    d\varphi_{n,j}$, where
    $c_{n,j} =\mu_{\plaque}\,\varphi_{n,j}\bar\stdr_{n,j}$. Let
    \begin{align*}
      A_{n,j} &= \log \rho\circ \varphi_{n,j}\circ\chvar, &%\\
      B_{n,j} &= \log \det d\varphi_{n,j}\circ\chvar
    \end{align*}
    so that
    $\log \bar\rho_{n,j}\circ \chvar = A_{n,j}+B_{n,j} - \log c_{n,j}$. By
    Lemma~\ref{lem:C2_bds}\ref{item:log_jacob_C2_bd},
    $\|dB_{n,j}\|_{\infty} \le C_{\clock}$ and
    $\|HB_{n,j}\|_{\infty} \le C_{\clock}$. Also, by
    Lemma~\ref{lem:dFeps_inv_conj_bd},
	\begin{align*}
      \|dA_{n,j}\|_{\infty} \le \| d(\log \rho \circ \chvar)\|_{\infty}\|d(\chvar\inv\circ \varphi_{n,j}\circ \chvar)\|_{\infty}\le D\Rough\, \xi_{n,j},
	\end{align*}
	where
    $\xi_{n,j} =
    \sup_{\varphi_{n,j}\bar{\stdr}_{n,j}}\Upsilon_{n}^\eps$.
    Similarly, by Lemma~\ref{lem:hessian_chain_rule} and
    Lemma~\ref{lem:C2_bds}\ref{item:chvar_hessian_bd},
	\begin{align*}
      \|HA_{n,j}\|_{\infty} &\le \|H(\log \rho \circ \chvar)\|_{\infty}(D\xi_{n,j})^2 + 2\| d(\log \rho \circ \chvar)\|_{\infty}\|H(\chvar\inv\circ \varphi_{n,j}\circ \chvar)\|_{\infty}\\
                   &\le \Econst\Rough D^2 \xi_{n,j}^2 + C_{\clock}\Rough.
	\end{align*}
    by choosing $\Econst\ge 4C_{\clock}$, we obtain that
	\begin{align*}
      \|d(\log \rho_{n,j} \circ \chvar)\|_{\infty}&\le D \xi_{n,j}\Rough+ C_{\clock},&%\\
      \|H(\log \rho_{n,j} \circ \chvar)\|_{\infty} &\le (D^2 \xi_{n,j}^2+\tfrac14)\Econst \Rough + C_{\clock}
	\end{align*}
	so we can take
    $\Rough_{n,j} = \max\{D\xi_{n,j},D^2\xi_{n,j}^2+\frac14\}\Rough +
    C_{\clock}$. The bound for $\Rough_{n,j}$ stated in (a) (independent on
    $j$) follows by noting that $\|\Upsilon_n^\eps\| \le C_{\clock}$ for
    all $n\le \clockN$. The bound for $\Rough_{n,j}$ given in part
    (b) follows by using Lemma~\ref{lem:fake_centre_exp}
    and~\eqref{eq:distortion_preimage_rectangles} to bound
    $\xi_{n,j}$.
\end{proof}
\begin{lem}\label{lem:lower_bd_weights}
  Let $\plaque$ be a $((\narr,\shor),({\rough,\Rough}))$-prestandard
  patch; for any $n\le \clockN$ let $\bplaque_{n,j}$ be a collection
  of standard patches as obtained in
  Proposition~\ref{lem:invariance_patches}, so that
  $F_{\ve *}^{n}\mu_{\plaque} = \sum_{j}c_{n,j}\mu_{\bplaque_{n,j}}$;
  then $c_{n,j} > C_{\Rough,\clock}\lowerBoundWeight^{n}/\shor$, where
  $C_{\Rough,\clock}$ depends on $\Rough$ and $\clock$ and
  $\lowerBoundWeight \in (0,1)$ is a constant that depends only on $f$
  and $\omega$.\
\end{lem}
\begin{proof}
  To simplify the notation, let
  $\bar{\stdr}_{n,j} = \supp{\bplaque_{n,j}}$ as in the proof of
  Proposition~\ref{lem:invariance_patches}; then by the change-of-variables
  formula, since $c_{n,j} = \mu_{\plaque}(\varphi_{n,j}\bar{\stdr}_{n,j})$:
	\begin{equation}\label{eq:weight_formula}
      c_{n,j} =\int_{\bar\stdr_{n,j}}\rho \circ \varphi_{n,j} \det
      d\varphi_{n,j}d\Leb \ge \Leb(\bar\stdr_{n,j}) \inf_{\stdr} \rho \|dF_\eps\|_\infty^{-2n}.
	\end{equation}
	On the one hand, by~\eqref{eq:bound-on-Lebesgue-patch}, we have
    $\Leb\stdr\le 2\delta\patchheight\eps$,
    hence~\eqref{eq:patch_density_bds} implies that
	\begin{equation}\label{eq:rho_lower_bd_vei}
      \inf_{\stdr} \rho \ge e^{-2\patchheight R}/ \Leb(\stdr)\ge C_{\Rough}\vei.
	\end{equation}
	On the other hand, again by~\eqref{eq:bound-on-Lebesgue-patch},
    and the bound on $Z_{n}$, we have
    \begin{align*}
      \Leb \bar\stdr_{n,j} &\ge e^{-\centreMaxExp n\ve
      }\frac{\delta\patchheight\ve}{200\shor}
    \end{align*}
    (recall that $\narr\le 100$); by combining this bound
    with~\eqref{eq:weight_formula} and~\eqref{eq:rho_lower_bd_vei} and
    recalling that $n < T\vei$, it follows that
    $c_{n,j}\ge C_{\Rough,T} \|dF_\eps\|_\infty^{-2n}/\shor $, as needed.
\end{proof}
We now fix $\minshor$ and $\Econst$ so that
Proposition~\ref{lem:invariance_patches} holds.
\section{Averaged motion and large deviations}\label{sec:averaging}
In this section we describe the relation between the averaged and the
actual dynamics for sufficiently long time-scales.  To make this
relation quantitative in the mostly expanding scenario that we study
in this paper, we proceed to prove a Lemma analogous to~\cite[Lemma
7.2]{MR3556527} and state some other useful results corresponding to
some Lemmata given in~\cite[Section 7]{MR3556527}.  The proofs are a
relatively straightforward adaptation of the arguments used
in~\cite[Section 7]{MR3556527}.

Recall the notation $\theta_{n}$, $\zeta_{n}$ defined
in~\eqref{eq:zeta_n_def}.  Let $\bandh>0$ be small, and define
$H_\bandh=\{\theta\ :\ |\theta - \theta_-|<\bandh\}$.  Recall also
that we fixed $\clock$ near the end of the previous section; we will
assume $\clock$ to be so large that all statement in this section hold
true.
\begin{lem}\label{lem:large-dev-near-sink}
  Let $\clock$ be sufficiently large, $h$ and $\eps$ sufficiently
  small; then for any \regular{} standard pair $\ell$ supported on
  $\bT\times H_{\bandh}$:
  \begin{align*}
    \mu_{\ell}(\theta_{\clockN} \in H_{3\bandh/4},\,\, \zeta_{\clockN} \geq 9\clock/16) \geq 1-\exp (-\const \eps^{-1}).
  \end{align*}
\end{lem}
\begin{proof}
  Let $\bar\theta(\cdot,\theta_{*})$ denote the solution
  to~\eqref{avgODE} with initial condition
  $\bar\theta(0,\theta_{*}) = \theta_{*}$ and
  $\bar\zeta(t,\theta_{*})$ be given by the integral
  \begin{align*}
    \bar\zeta(t,\theta_{*}) &= \int_{0}^{t}\bar\psi(\bar\theta(t',\theta_{*}))dt'.
  \end{align*}
  Recall (see~\eqref{avgODE}) that $\theta_{-}$ is a sink for the
  averaged dynamics; hence, for $\bandh$ is sufficiently small,
  $H_\bandh$ is forward-invariant for the averaged dynamics.  Recall
  moreover (see Remark~\ref{rmk:psi_regularity}) that
  $\bar\psi(\theta_-) \ge 3/4$; by continuity of $\bar \psi$, we can
  thus choose $\bandh$ sufficiently small so that
  $\bar \psi(\theta) > 5/8$ for all $\theta \in H_{\bandh}$.
  Moreover, we can choose $\clock>0$ sufficiently large so that
  $\bar\theta(\clock,\cdot)(H_\bandh)\subset H_{\bandh/2}$.

  Next, let $\ell$ be a regular standard pair supported on
  $\bT\times H_\bandh$; recall the definitions of the interpolations
  $\theta_{\ve}$ and $\zeta_{\ve}$
  (see~\eqref{eq:interpolating-definition} and
  below~\eqref{eq:zeta_n_def}) and define the set
  \begin{align*}
    \mathcal{S}=\Big\{p_{*} =(x_{*}, \theta_{*}) \in\, \text{supp}\, {\ell} \, :&\sup_{t \in [0,\clock]} |\theta_\eps(t,p_{*})-\bar\theta(t,\theta_{*})|<\bandh/4,\\&\sup_{t \in [0,\clock]} |\zeta_\eps(t,p_{*})-\bar\zeta(t,\theta_{*})|<\clock/16 \Big\}.
  \end{align*}
  Recall also that assumption~\ref{a_noNearCobo} is satisfied with $\psi$ in place of $\psi_*$, so $(\omega,\psi)$ satisfies~\cite[condition (A1')]{MR3556527}. Thus by the Large Deviation Principle (\cite[Theorem 6.1]{MR3556527}) we
  conclude that, if $\ve$ is sufficiently small,
  $\mu_\ell(\mathcal{S})\geq 1-\exp(-\const \eps^{-1})$; the lemma
  follows if we show that $\theta_{\clockN}(p) \in H_{3\bandh/4}$ and
  $\zeta_{\clockN}(p)\geq 9\clock/16$ for any $p\in\mathcal S$ and
  sufficiently small $\eps$.

  By definition of $\mathcal{S}$, and our choice of $\clock$ we have
  $\theta_{\clockN}(\mathcal{S})\subset H_{3\bandh/4}$.  Moreover,
  by our assumption on $\bandh$, we have $\bar\zeta(t,\theta_{*}) >
  5t/8$ for any $\theta_{*}\in H_{\bandh}$.  Hence, by definition of
  $\mathcal{S}$ we conclude that for any $p\in\mathcal S$:
  \begin{align*}
    \zeta_{\clockN}(p) > \frac58\clock -\frac{\clock}{16} &\geq
    \frac{9\clock}{16}. \qedhere
  \end{align*}
\end{proof}
Next, we proceed to describe the $\vei\log\vei$-timescale dynamics of
standard pairs: more precisely, for $\logstep > 0$ we will study the
dynamics after $\lfloor \logstep \log\vei\rfloor\clockN$ iterates,
\nomenclature[3a]{$\logstep$}{Number of time steps of size $\log\vei$}
The first lemma below is a restatement of~\cite[Lemma 7.4]{MR3556527}
adapted from the mostly contracting case and describes the dynamics of
standard pairs supported near the sink; the second lemma is a
restatement of~\cite[Lemma 7.5]{MR3556527} and shows that arbitrary
standard pairs tend to the sink.

\begin{lem}\label{lem:piling-up-near-the-sink}
  Let $\clock > 0$ be sufficiently large; then there exists
  $C,\logstep_{*} > 0$ such that if $\ve$ is sufficiently small, for
  any \regular{} standard pair supported on $\bT\times H_{\bandh}$ and any
  $\logstep\ge \logstep_{*}$:
  \begin{align*}
    \mu_{\ell}\big(\theta_{\lfloor V\log\vei\rfloor\clockN }\in
    B(\theta_{-},C\sqrt{\ve})\big)\ge \frac23.
  \end{align*}
\end{lem}
\begin{proof}
  The statement follows immediately by~\cite[Lemma 7.4]{MR3556527},
  which is stated in a slightly different language.
  In~\cite{MR3556527}, the following notion was introduced: we say
  that a standard pair is \emph{located at $U\subset\bT$} if the
  average of the random variable $\theta(\cdot)$ with respect to
  $\mu_{\ell}$ belongs to $U$.  Clearly, if a standard pair is
  supported on $\bT\times U$, then it is located at $U$; on the other
  hand, if it is located at $U$, then it is supported on
  $\bT\times \hat U$ where $\hat U$ is an $O(\ve)$-neighbourhood of
  $U$.

  Next, observe that the symbol $H_k$ appearing in~\cite[Lemma
  7.4]{MR3556527} is defined above~\cite[(6.14)]{MR3556527} as an
  appropriate $\cO(1)$-neighbourhood of $\theta_{k,-}$, where
  $\theta_{k,-}$ is the $\theta$-coordinate of the $k$-th sink.  In
  the present paper we only have one sink, so we can take
  $\theta_{k,-} = \theta_{-}$.  The choice of the size of the
  neighbourhood is made in~\cite[Lemma 6.14]{MR3556527}, which is the
  analogue of Lemma~\ref{lem:large-dev-near-sink} in this paper; as in
  Lemma~\ref{lem:large-dev-near-sink}, the choice is made to ensure
  that $\bar\Psi$ is sufficiently close to the value of the same
  function at the sink.  This choice, on the other hand, does not play
  any role in the proof of~\cite[Lemma 7.4]{MR3556527}, which only
  deals with the dynamics of the variable $\theta$.  We conclude that
  the statement of~\cite[Lemma 7.4]{MR3556527} holds for our system,
  when we replace $H_{k}$ with
  $H_\bandh=\{\theta\ :\ |\theta - \theta_-|<\bandh\}$ defined above.

  Combining these observations then yields
  Lemma~\ref{lem:piling-up-near-the-sink}.
\end{proof}
\begin{lem}\label{lem:capture-by-the-sink}
  Let $\clock > 0$ be sufficiently large: under
  assumption~\ref{a_noNearCobo}, there exists $\beta > 0$ and
  $\logstep_{0} > 0$ such that if $\ve$ is sufficiently small, for any
  \regular{} standard pair $\ell$ and any
  $n \ge {\lfloor\logstep_{0}\log\vei\rfloor\clockN}$ we have:
  \begin{align*}
    \mu_{\ell}\big(\theta_n \nin H_{3h/4}\big) < \ve^{\beta}.
  \end{align*}
\end{lem}
\begin{proof}
  By applying the same observations presented in the proof of
  Lemma~\ref{lem:piling-up-near-the-sink}, this lemma follows by
  applying~\cite[Lemma 7.5]{MR3556527}. Notice that~\cite[Lemma
  7.5]{MR3556527} is stated in terms of $\hat{\mathbb{H}}$, defined
  immediately below~\cite[(6.14)]{MR3556527}; since in this paper we
  only deal with one sink, we can replace $\hat{\mathbb{H}}$ with
  $H_{3\bandh/4}$.  Lemma~\ref{lem:capture-by-the-sink} follows
  by~\cite[Lemma 7.5]{MR3556527} as follows.  Set
  $T_{0} = T_{\textrm{S}}$, $V_{0} = \mathcal{R}_{\textrm{A}}$; then
  for any
  $n = j+{\lfloor\logstep_{0}\log\vei\rfloor\clockN} \ge
  {\lfloor\logstep_{0}\log\vei\rfloor\clockN}$, applying~\cite[Lemma
  7.5]{MR3556527} to each pair in a standard family representation of
  $F^{j}_{\ve*}\mu_{\ell}$ yields the required estimate.
\end{proof}
%%%%%%
Combining Lemmata~\ref{lem:piling-up-near-the-sink}
and~\ref{lem:capture-by-the-sink} we obtain the following corollary,
that will be used in the sequel.
\begin{cor}\label{cor:convenient}
  There exists $C, \logstep > 0$ so that for any
  $n\ge \lfloor V\log\vei\rfloor\clockN$ and any regular standard pair
  $\ell$ we have:
\begin{align*}
  \mu_{\ell}\left(\theta_{n}\in B\left(\theta_{-},C\sqrt{\ve}\right)\right)\ge \frac12.
\end{align*}
\end{cor}

Finally we need the following result, which is an immediate
consequence of the Local Central Limit Theorem (\cite[Theorem
6.8]{MR3556527}) and the fact that $\theta_{-}$ is a sink:
\begin{lem}\label{lem:lclt_corol}
  Let $\clock$ be sufficiently large and $\eps$ sufficiently small.
  Under assumption~\ref{a_noNearCobo}, for any $C>0$ there exists
  $p > 0$ such that for any \regular{} standard pair supported on
  $\bT\times B(\theta_-,C\sqrt{\eps})$ and any interval
  $I\subset B(\theta_-,C\sqrt{\eps})$ such that $|I|\ge \eps$ we have
  $ \mu_{\ell}(\theta_{\clockN}\in I)\ge p\eps^{1/2}$.
\end{lem}
We hereby fix $\bandh$ to be small enough so that
Lemma~\ref{lem:large-dev-near-sink} holds.
\section{Patch families}\label{sec:patch-families}
We now define patch families, which relate to standard patches in the
same way that families of standard pairs relate to standard
pairs. Recall that Remark~\ref{rem:regular-pair-attracting} implies
that families of \regular{} standard pairs are invariant under the
dynamics.

We call a standard patch $\plaque=(\fstdr,\rho)$ \textit{regular} if
every standard pair obtained by disintegrating $\rho$ along the
foliation is\regular{} (that is, $\plaque$ is a
$((2,\shor),(3\rough_*/2),\Rough)$-standard patch for some
$\shor\ge \minshor$ and $\Rough>0$, where $\rough_{*}$ is given by
Proposition~~\ref{proposition:invariance}).  We denote with
$\plaqueSet$ the collection of all \regular{} standard patches.
\begin{mydef}
  A \textit{patch family} is a discrete\footnote{ The assumption about
    the discreteness of the probability space $\cG$ has the sole
    purpose of simplifying the exposition, as it will not be necessary
    in this paper to consider families that are not discrete.  Of
    course it is an inessential assumption. } probability space
  $\cG = (\mathcal{A},\lambda_{\mathcal{G}})$ together with a map
  $\plaque: \mathcal{A}\to \plaqueSet$. We will abuse terminology by
  referring to elements of $\plaque(\mathcal{A})$ as ``standard
  patches in $\cG$''.
\end{mydef}
Each patch family $\cG$ induces a Borel probability measure on $\T^2$
defined by
\begin{equation*}
	\mu_{\cG}(g) = \int_{\mathcal{A}} \mu_{\plaque(\alpha)}(g)d\lambda_{\cG}[\alpha].
\end{equation*}
for all continuous functions $g:\T^2 \to \R$.  Once again, it is
helpful to keep in mind that patch families can be regarded as random
elements in the space of regular standard patches.

We say that a Borel probability measure $\mu$ on $\T^2$ \textit{admits
  a disintegration as a patch family} if there exists a patch family
$\cG$ such that $\mu = \mu_{\cG}$. We then let $[\mu]$ denote the
equivalence class of all patch families $\cG$ such that
$\mu_{\cG} = \mu$.  Given a patch family $\cG$, we also introduce the
(mildly abusing) notation $[\cG] = [\mu_{\cG}]$ and
$[F^j_{\eps *} \cG] = [F^j_{\eps*} \mu_{\cG}]$.  Given a patch family
$\cG=((\mathcal{A},\lambda_{\cG}),\plaque)$ and a set
$\mathcal{A}'\subset \mathcal{A}$ with
$\lambda_{\cG}(\mathcal{A}')>0$, we define \textit{the subfamily
  conditioned on $\mathcal{A}'$} to be
$\cG|\mathcal{A}'=((\mathcal{A}',\lambda_{\cG|\mathcal{A}'}),\plaque|_{\mathcal{A}'})$,
where $\lambda_{\cG|\mathcal{A}'}(\cdot)=\lambda(\cdot|\mathcal{A}')$
and $\plaque|_{\mathcal{A}'}$ is the restriction of $\plaque$ to
$\mathcal{A}'$.

Given an (at most countable) collection of patch families
$\cG_i = (\mathcal{A}_i,\lambda_{\cG_i})$ with associated maps
$\plaque_i: \mathcal{A}_i\to \plaqueSet$, and weights $c_{i} \in [0,1]$
such that $\sum_i c_{i}=1$, the \textit{convex combination}
$\sum_i c_{i} \cG_i$ is the patch family
$\cG=(\mathcal{A}, \lambda_{\cG})$ defined by ``choosing the patch
family $\cG_i$ randomly with probability $c_{i}$". More precisely,
$\mathcal{A}= \{(\alpha,i): \alpha\in \mathcal{A}_i\}$ and
$\lambda_{\cG}$ and $\plaque$ are defined by
$\lambda_{\cG}(\{(\alpha,i)\})=c_{i} \lambda_{\cG_i}(\{\alpha\})$ and
$\plaque((\alpha,i))=\plaque_i(\alpha)$ for all
$(\alpha,i)\in \mathcal{A}$.

Let $\plaque = (\fstdr,\rho)$ be a \regular{} standard patch. We
define $\mathcal{\shor}(\plaque)$ to be the minimum
$\shor\ge \minshor$ such that $\fstdr$ is a foliated
$(2,\shor)$-standard rectangle; likewise, we define
$\mathcal{\Rough}(\plaque)$ to be the minimum $\Rough>0$ such that
$\rho$ is an $\Rough$-standard density.  Finally, we set
$\mathcal{M}(\plaque) =
\max\{\mathcal{\shor}(\plaque),\mathcal{\Rough}(\plaque)\}$.
\nomenclature[4aa]{$\mathcal{\shor}(\plaque)$}{Minimal $\shor$ such
  that $\plaque$ is a $((\narr,\shor),(\rough,\Rough))$-standard patch}
\nomenclature[4ab]{$\mathcal{\Rough}(\plaque)$}{Minimal $\Rough$ such
  that $\plaque$ is a $((\narr,\shor),(\rough,\Rough))$-standard patch}%
Given a patch family $\cG = ((\mathcal{A}, \lambda_\cG),\plaque)$ we
can naturally regard $\mathcal{\shor}$, $\mathcal{\Rough}$ and $\cM$
as random variables on the probability space
$(\mathcal{A}, \lambda_\cG)$ (by composing with the function
$\plaque$).

\begin{rmk}
  Let $\plaque$ be a regular standard patch; if we apply
  Lemma~\ref{lem:cutting-patches} to $\plaque$ with $m = \clockN$, we
  obtain a patch family $\cG$ of $\clockN$-adapted patches.  Observe
  that since $\minshor > e^{\centreMaxExp\clock}$ (assumption made in
  Proposition~\ref{lem:invariance_patches}), for any patch
  $\tilde{\plaque}$ in $\cG$, we have
  $\mathcal{\shor}(\tilde{\plaque}) = \mathcal{\shor}(\plaque)$ (and
  likewise for $\mathcal{\Rough}$ and $\cM$). In particular, if $\cG$
  is a patch family, then applying Lemma~\ref{lem:cutting-patches} to
  each patch in $\cG$ yields a family $\tilde\cG$ that is
  $\clockN$-adapted and so that the distributions of the random
  variables $\mathcal{\shor}$, $\mathcal{\Rough}$ and $\cM$ do not
  change.
\end{rmk}

Observe that the pushforward of a patch family by $F_{\ve}$ of is
itself a patch family (by Proposition~\ref{lem:invariance_patches}).
The ideal situation would be that patch families that satisfy a
certain \emph{uniform} bound on $\cM$ were invariant for the dynamics
(similarly to what was observed in
Remark~\ref{rem:regular-pair-attracting} for standard pairs).
However, due to the effectively random nature of the slow dynamics, we
need to pare down our expectations, and aim for a much weaker (but
still useful!) $L^{1}$-bound.  The following lemma is a first step in
establishing the invariance of a class of patch families that satisfy
a suitable $L^{1}$-bound (the curious reader can have a look at
Definition~\ref{def:proper_patch_family})
\begin{lem}\label{lem:patch_family_a_priori_bd}
  Assume $\clock$ is sufficiently large; then there exist constants
  $\alpha_1,\alpha_2>0$ such that the following holds for all $\eps$
  sufficiently small. Let $\plaque_0$ be a \regular{} standard patch,
  for any $n\ge 0$:
	\begin{enumerate}
    \item There exists a patch family $\cH\in [F_{\eps*}^n \plaque_0]$
      such that for any standard patch $\plaque$ in $\cH$, we have
      $\cM(\plaque)\le e^{\alpha_1 n\ve}(1+\cM(\plaque_0))$.
    \item If moreover $n \le \eps^{-3/2}$ and $\plaque_0$ is supported
      on $\bT\times H_\bandh$ (recall the definition of $H_{\bandh}$
      given above Lemma~\ref{lem:large-dev-near-sink}), then $\cH$ can
      be chosen such that
      \begin{align*}
        \lambda_{\cH}(\cM(\plaque)) \le \Const (e^{-\alpha_2
        n\ve}\cM(\plaque_0) + 1).
      \end{align*}
	\end{enumerate}
\end{lem}
\begin{proof}
  Let $i = \lfloor n/\clockN\rfloor$ so that we can write
  $n = i \clockN + (n-i\clockN)$.  Then by repeatedly applying
  Proposition~\ref{lem:invariance_patches} we obtain a patch family
  $\cH \in [F_{\eps*}^n {\plaque_0}]$ such that for any standard patch
  $\plaque$ in $\cH$,
  $\mathcal{\shor}(\plaque) \le e^{\centreMaxExp
    n\eps}\mathcal{\shor}(\plaque_0)$ and
  \begin{align*}
    \mathcal{\Rough}(\plaque) \le C_{\clock}^{i+1} + \sum_{l=0}^{i} C_{\clock}^l\Rough_*;
  \end{align*}
  this concludes the proof of (a) with a suitable choice of
  $\alpha_{1}$.  For item (b), suppose that $\plaque_0$ is supported
  on $\T\times H_\bandh$. We claim that for $\clock$ sufficiently
  large and $\eps$ sufficiently small, there exist patch families
  $\mathcal{H}^g_i$ and $\mathcal{H}^b_i$ with the following
  properties:
	\begin{itemize}
    \item[(1)]
      $w_i \mathcal{H}^g_i + (1-w_i)\mathcal{H}^b_i \in [F_{\eps*}^{i
        \clockN} {\plaque_0}]$, where
      $w_i \ge \left(1-e^{-\const \eps\inv}\right)^i$;
    \item[(2)] any standard patch $\plaque$ in $\mathcal{H}^g_i$ is
      supported on $\T \times H_\bandh$ and satisfies
      \begin{align*}
        \mathcal{M}(\plaque)\le 2^{-i}
        \mathcal{M}(\plaque_0)+2(\Rough_*+\minshor).
      \end{align*}
    \item[(3)] any standard patch $\plaque$ in $\mathcal{H}^{b}_{i}$
      satisfies
      \begin{align*}
        \cM(\plaque)\le C_{\clock}^{i}\cM(\plaque_0)+C_{\clock}^{i}
      \end{align*}
	\end{itemize}
    We will prove the claim by induction: it clearly holds for $i=0$,
    so let us suppose that it holds for $i = m$. Let $\plaque$ be a
    patch in $\cH_m^g$: applying
    Proposition~\ref{lem:invariance_patches}(b) for $n = \clockN$ to
    $\plaque$ we obtain a patch family
    $\tilde{\mathcal G_{\plaque}} =
    ((\bar{\cA}_{\plaque},\bar{\lambda}_{\plaque}),\bar{\plaque})$
    where $\bar{\cA}_{\plaque}$ denotes the index set of the
    collection $\{\bplaque_{\clockN,j}\}$,
    $\bar\lambda_{\plaque}\{j\} = c_{\clockN,j}$ and
    $\bar{\plaque}(j) = \bar{\plaque}_{\clockN,j}$.  Let us
    also denote for convenience
    $\bar{\stdr}_{\clockN,j}=\supp{\bplaque_{\clockN,j}}$ and recall
    that
    $c_{N,j} =
    \mu_{\plaque}(\varphi_{\clockN,j}(\bar{\stdr}_{\clockN,j}))$;
    define
	\begin{align*}
      \bar{\cA}'_{\plaque} &=
                                     \left\{j\in\bar{\cA}_{\plaque}:\inf_{\bar{\stdr}_{\clockN,j}}\zeta_{\clockN}\circ\varphi_{\clockN,j}\ge
                                     9\clock/16,\ \bar{\stdr}_{\clockN,j}\cap (\bT
                                     \times H_{3\bandh/4}) \neq \emptyset \right\}.
	\end{align*}
    Then by definition and Lemma~\ref{lem:large-dev-near-sink} we
    obtain\footnote{ Note that by
      Remark~\ref{rmk:standard-patches-are-families}, $\mu_{\plaque}$
      admits a disintegration as a family of \regular{} standard pairs
      supported on $\bT\times H_\bandh$.}:
    \begin{align*}
      \bar\lambda_{\plaque}(\bar{\cA}'_{\plaque})\ge \mu_{\plaque}(\theta_{\clockN} \in H_{3\bandh/4},\,\, \zeta_{\clockN} \geq 9\clock/16) \geq 1-\exp
      (-\const \eps^{-1}).
    \end{align*}
    Let $j\in \bar{\cA}'_{\plaque}$; then the quantity $M_{\clockN,j}$
    defined in Proposition~\ref{lem:invariance_patches}(b), for
    $\clock$ sufficiently large and $\eps$ sufficiently small
    satisfies the following bound (recall $\fakeCentreConst < 1/4$):
    \begin{align*}
      M_{\clockN,j}\le\exp\left(-5\clock/16+C_{\clock}\ve\log\vei +
      D\right)\le \frac{1}{2}.
    \end{align*}

    We conclude that
    \begin{align*}
      \mathcal{\shor}(\bplaque_{\clockN,j})&\le \max\{\tfrac12
                                              \mathcal{\shor}(\plaque),\minshor\},&%\\
      \mathcal{\Rough}(\bplaque_{\clockN,j})&\le \tfrac12
                                              \mathcal{\Rough}(\plaque) + \Rough_* .
    \end{align*}
    By the inductive hypothesis, it follows that
    \begin{align*}
      \cM(\bplaque_{\clockN,j}) \le \tfrac12 \cM(\plaque) +\minshor +
      \Rough_* &\le 2^{-m-1}\cM(\plaque_0) + 2 (\minshor + \Rough_*).
    \end{align*}
	Moreover, for $\eps$ sufficiently small
    $\bar{\stdr}_{\clockN,j}\subset\bT\times H_{\bandh}$.

	Let
    $\tilde{\cG}^g_{\plaque} =
    \tilde\cG_{\plaque}|\bar{\cA}'_{\plaque}$ and\footnote{ Here we
      implicitly assume that
      $\bar{\cA}_{\plaque}\not = \bar{\cA'}_{\plaque}$; otherwise we
      can ignore $\tilde{\cG}^g_{\plaque}$ in what follows.  The
      reader can easily fill in the details.}
    $\tilde{\cG}^b_{\plaque} =
    \tilde\cG_{\plaque}|(\bar\cA_{\plaque}\setminus\bar{\cA}'_{\plaque})$;
    then by construction we have that
    $\bar\lambda_{\plaque}(\bar{\cA}')\tilde{\cG}_{\plaque}^g +
    (1-\bar\lambda_{\plaque}(\bar{\cA}'))\tilde{\cG}_{\plaque}^b \in
    [F_{\eps*}^{\clockN} {\plaque}]$ and each patch family in
    $\tilde{\cG}_{\plaque}^g$ satisfies the properties stated in (2)
    for $i = m+1$.

	By taking convex combinations of the patch families
    $\tilde{\cG}^g_{\plaque}$, $\tilde{\cG}^b_{\plaque}$ obtained for
    each $\plaque \in \cH^g_{m}$, we obtain patch families
    $\cH^g_{m+1}$ and $\cE_{m+1}$ such that
    \[c \cH^g_{m+1} + (1-c)\cE_{m+1} \in [F_{\eps*}^{\clockN}
      {\cH_m^g}],\] where
    $c\ge \min_{\plaque \in \cH^g_{m}}\bar{\lambda}_{\plaque}(\bar{\cA}'_{\plaque})\ge 1 - e^{-\Const
      \eps\inv}$.
	% and $\cH_{m+1}^g$ satisfies the properties stated in (b).
	On the other hand, by Proposition~\ref{lem:invariance_patches} we can
    choose a patch family
    $\cE'_{m+1} \in [F_{\eps*}^{\clockN} {\cH_{m}^b}]$. It follows
    that
    \[ cw_m\cH_{m+1}^g+(1-c)w_m\cE_{m+1}+(1-w_m)\cE'_{m+1}\in
      [F_{\eps*}^{(m+1)\clockN}{\plaque_0}], \] which completes the
    proof of the first two items in our claim.  In order to prove item
    (3) notice that since the patch family $\cH^b_i$ was constructed
    by repeatedly applying Proposition~\ref{lem:invariance_patches},
    by the same reasoning used in the proof of (a) we have
    $\cM(\plaque)\le C_{\clock}^{i}\cM(\plaque_0)+C_{\clock}^{i}$
    for any standard patch $\plaque$ in $\cH_i^b$.

	We can now complete the proof of (b). We apply the above claim
    with the previously defined $i = \lfloor n/\clockN\rfloor$.  Then:
	\begin{align*}
      \lambda_{\widehat{\cG}_{i}}(\cM(\plaque)) &\le 2^{-i} \mathcal{M}(\plaque_0)+2(\Rough_*+\minshor) + (1 - w_i)(C_{\clock}^i\cM(\plaque_0)+C_{\clock}^i)\\
                                                  &\le 2^{-i} \mathcal{M}(\plaque_0)+2(\Rough_*+\minshor) + i e^{-\Const \eps\inv}(C_{\clock}^i\cM(\plaque_0)+C_{\clock}^i).
	\end{align*}
	Now since $i\le \eps^{-1/2}$, for $\eps$ sufficiently small we
    have $ie^{-\Const \eps\inv}C^i_{\clock}\le 2^{-i}$, so
	\begin{equation}\label{eq:lasota_yorke_multiple_N}
      \lambda_{\widehat{\cG}_{i}}(\cM(\plaque))\le 2\cdot 2^{-i} \cM(\plaque_0)+2 (\Rough_* + \minshor) + 2^{-i}.
	\end{equation}
	Finally, by part (a) we can choose a patch family
    $\cH \in [F_{\eps*}^{n - i\clockN}\widehat{\cG}]_{i}$ such
    that
    \[\lambda_{\cH}(\cM(\plaque))\le
      C_{\clock}\lambda_{\widehat{\cG}_{i}}(\cM(\plaque))+C_{\clock}.\]
    The proof of part (b) follows by combining this
    with~\eqref{eq:lasota_yorke_multiple_N}.
\end{proof}
We can now once and for all determine $\clock$ to be large enough so
that
Lemmata~\ref{lem:large-dev-near-sink},~\ref{lem:piling-up-near-the-sink},~\ref{lem:capture-by-the-sink}
and the above Lemma~\ref{lem:patch_family_a_priori_bd} hold.
\begin{prp}\label{prop:geometric_drift}
  There exist constants $\lyapFnExp\in(0,1)$, $\lyapFnDrift>0$ and
  $\logstep_{1}>0$ such that the following holds for all $\eps$
  sufficiently small. Let $\plaque_0$ be a \regular{} standard
  patch; then for all $n\ge \logstep_{1} \vei \log\vei$ there exists a
  patch family $\cH\in [F_{\eps*}^n {\plaque_{0}}]$ such that
  \begin{align*}
    \lambda_{\cH}(\cM(\plaque)^{\lyapFnExp}) \le \eps\,
    \cM(\plaque_{0})^{\lyapFnExp}+\lyapFnDrift.
  \end{align*}
\end{prp}
\begin{proof}
  Fix $\nu>0$ to be chosen later and let
  $q = \lfloor\logstep_{0}\log\vei\rfloor\clockN+\lfloor\nu
  \vei\log\vei\rfloor$, where $\logstep_{0}$ is as in
  Lemma~\ref{lem:capture-by-the-sink}. Let $q\le n\le 2q$ and choose
  $\lfloor\logstep_{0}\log\vei\rfloor\clockN\le i\le
  2\lfloor\logstep_{0}\log\vei\rfloor\clockN$ and
  $\lfloor\nu \vei\log\vei\rfloor\le j \le 2\lfloor\nu
  \vei\log\vei\rfloor$ such that $n = i+j$.

  By Lemma~\ref{lem:patch_family_a_priori_bd}(a), there exists a patch
  family $\cH\in [F_{\eps*}^i {\plaque_{0}}]$ such that for any
  standard patch $\plaque$ in $\cH$ we have
	\begin{equation}\label{eq:first_step_lyapFn_bd}
		\cM(\plaque)^{\lyapFnExp} \le \eps^{-2\lyapFnExp\alpha_1 \logstep_{0}\clock}(1+\cM(\plaque_0)^{\lyapFnExp}).
	\end{equation}
	By Lemma~\ref{lem:capture-by-the-sink} and
    Remark~\ref{rmk:standard-patches-are-families},
    $\mu_{\plaque_0}(\theta_j\nin H_{3\bandh/4})<\eps^\beta$. For
    $\eps$ sufficiently small, any standard patch that intersects
    $\bT\times H_{3\bandh/4}$ will be supported on
    $\bT\times H_{\bandh}$. It follows that
    $p = \lambda_{\cH}(\plaque \in S)>1-\eps^\beta$, where
    $S = \{\plaque\in \plaqueSet:\supp\plaque \subset \bT\times
    H_{\bandh}\}$.

	Since any standard patch $\plaque$ in $\cH|\{\plaque \in S\}$
    satisfies~\eqref{eq:first_step_lyapFn_bd}, by
    Lemma~\ref{lem:patch_family_a_priori_bd}(b) and Jensen's
    inequality, there exists a patch family
    $\cG_1\in [F_{\eps*}^j({\cH|\{\plaque \in S\}})$ such that
	\begin{align*}
      \lambda_{\cG_1}(\cM(\plaque)^{\lyapFnExp})&\le \lambda_{\cG_1}(\cM(\plaque))^{\lyapFnExp}\le \Const (\eps^{\lyapFnExp \alpha_2 \nu}\lambda_{\cH|\{\plaque \in S\}}(\cM(\plaque))^{\lyapFnExp}+1)\\
                                                &\le \Const(\eps^{\lyapFnExp(\alpha_2 \nu - 2\alpha_1 \logstep_{0}\clock)}\cM(\plaque_0)^{\lyapFnExp} + \eps^{\lyapFnExp(\alpha_2 \nu - 2\alpha_{1}\logstep_{0}\clock)} + 1).
	\end{align*}
	Similarly, by Lemma~\ref{lem:patch_family_a_priori_bd}(a), there
    exists a patch family
    $\cG_2 \in [F_{\eps*}^j ({\cH|\{\plaque \notin S\}})]$ such that
	\begin{align*}
      \lambda_{\cG_2}(\cM(\plaque)^{\lyapFnExp})&\le (\eps^{-2\lyapFnExp \alpha_1 (\nu +\logstep_{0}\clock)}\cM(\plaque_0)^{\lyapFnExp}+2\eps^{-2\lyapFnExp \alpha_1 (\nu +\logstep_{0}\clock)}).
	\end{align*}
	%	Let $\nu$ be sufficiently large such that $\alpha_2 \nu - 2\alpha_1 \logstep_{0}\clock>0$ and let $\lyapFnExp$ be sufficiently small such that $\beta - 2\lyapFnExp \alpha_2 (\nu +\logstep_{0}\clock)>0$ and set
	Let $\cG = p\,\cG_1 + (1-p)\cG_2$. Since $p\, \cH|\{\plaque \in S\}+(1-p)\cH|\{\plaque \notin S\}\in [\mu_{\cH}]$, we have $\cG \in [F_{\eps*}^j {\cH}]=[F_{\eps*}^n {\plaque_0}]$. Let $\nu$ be sufficiently large and $\lyapFnExp$ be sufficiently small such that
	\[\tau = \min\{\lyapFnExp(\alpha_2 \nu - 2\alpha_1 \logstep_{0}\clock),{\beta -2\lyapFnExp \alpha_2 (\nu +\logstep_{0}\clock)}\}>0.\]
	Then we obtain that
	\begin{align*}
		\lambda_{\cG}(\cM(\plaque)^{\lyapFnExp})&= p\lambda_{\cG_1}(\cM(\plaque)^{\lyapFnExp})+(1-p)\lambda_{\cG_2}(\cM(\plaque)^{\lyapFnExp})\\
		&\le (\Const + 1)\eps^\tau \cM(\plaque_0)^{\lyapFnExp} + \Const \le \eps^{\tau/2}\cM(\plaque_0)^{\lyapFnExp} + \Const
	\end{align*}
	for $\eps$ sufficiently small. Since $n \in [q,2q]$ was arbitrary,
    by iterating this argument it follows that for all $n\ge q$, there
    exists a patch family $\tilde\cG\in [F_{\eps*}^n \mu_{\plaque_0}]$
    such that
	\begin{align*}
		\lambda_{\tilde\cG}(\cM(\plaque)^{\lyapFnExp})\le \eps^{\lfloor n/q\rfloor\tau/2}\cM(\plaque_0)^{\lyapFnExp} + \sum_{l=0}^\infty \eps^{l\tau/2}\Const \le \eps^{\lfloor n/q\rfloor\tau/2}\cM(\plaque_0)^{\lyapFnExp}+\Const.
	\end{align*}
	Since $q\le (\nu + \logstep_{0}\clock)\vei \log\vei$, the result
    follows by choosing $\logstep_{1}$ large enough so that
    $\logstep_{1}\tau\ge2q$.
\end{proof}
Proposition~\ref{prop:geometric_drift} leads us to consider the
following definition: \nomenclature[4ac]{$\El(\plaque)$}{Lyapunov
  function describing the regularity of standard patches}%
\begin{mydef}\label{def:proper_patch_family}
  We call a patch family $\cG$ \textit{proper} if
  $\lambda_{\cG}(\El(\plaque))\le 2\lyapFnDrift$, where
  $\El(\plaque)=\cM(\plaque)^{\lyapFnExp}$ and $\lyapFnExp$ and
  $\lyapFnDrift$ are as in Proposition~\ref{prop:geometric_drift}.
\end{mydef}
Observe that the set of proper patch families is closed by convex
combination.  Proposition~\ref{prop:geometric_drift} shows that the
class of proper patch families has some (weak) invariance properties
(invariance holds only after a sufficiently long time)
\begin{cor}\label{cor:proper_invariant}
  Let $\cG$ be a proper patch family. Then for all
  $n\ge \logstep_{1} \vei\log\vei$, $F_{\eps*}^n \mu_{\cG}$ admits a
  disintegration as a proper patch family.
\end{cor}
\begin{rmk}
  Let $\cG$ be an arbitrary proper patch family; notice that its
  push-forward $F_{\ve*}\mu_{\cG}$ does not necessarily admit a
  decomposition as a proper patch family; hence the class of proper
  patch families is not –strictly speaking– invariant.  The weaker
  invariance stated above is, however, enough for our future purposes.
  It is also the main reason behind our bound on $c_{\ve}$: since we
  will only be able to run the coupling argument on proper patch
  families, the argument will only work with a timestep of order
  $\vei\log\vei$.
\end{rmk}
We conclude this section by observing that the class of proper patch
families is large enough to contain smooth measures on $\bT^{2}$.
\begin{lem}\label{lem:smooth_density_proper}
  Let $\kappa>0$ and $\nu$ be an absolutely continuous probability
  measure on $\bT^2$ with smooth density and such that
  $\|d\log (d\nu/d\Leb)\|_{\cC^1}\le \kappa$.  Then for $\ve$
  sufficiently small, $\nu$ admits a disintegration as a proper patch
  family.
\end{lem}
\begin{proof}
  Let us first construct a partition (mod 0) of $\bT^{2}$ into
  foliated standard rectangles.  First, partition (the horizontal)
  $\bT$ into intervals $I_{i}$ of equal length between $3\delta/5$ and
  $4\delta/5$ (which is always possible if $\delta < 5/12$); let
  $\bar x_{i}$ denote the endpoints of such intervals.  Then, partition
  (the vertical) $\bT$ into intervals $\{J_{j}\}$ of equal length
  between $\patchheight\ve/2$ and $\patchheight\ve$; let $\bar\theta_{j}$
  denote the endpoints of such intervals.  For each $i,j$, denote
  $\cW_{i,j}$ the centre manifold connecting the point
  $(\bar x_{i},\bar\theta_{j})$ to $\bT\times\{\bar\theta_{j+1}\}$.

  Then, provided $\ve$ is sufficiently small, for any $j$, the
  collection of curves $\{\cW_{i,j}\}_{i}$ partitions the strip
  $\bT\times J_{j}$ into standard rectangles (since horizontal
  segments are –in particular– standard curves).  Moreover it is
  immediate to check that the horizontal foliation is a standard
  foliation.  We thus obtain a partition of $\bT^{2}$ into standard
  foliated rectangles $\fstdr_{i,j} = (\stdr_{i,j},\foli_{i,j})$.

  By our assumption on $\nu$, we know that $d\nu = \rho\cdot d\Leb$
  where $\rho$ is a density function so that
  $\|d\log\rho\|_{\cC^{1}} < \kappa$.  Let
  $\rho_{i,j} = c_{i,j}^{-1}\rho|_{\stdr_{i,j}}$, where
  $c_{i,j} = \int_{\stdr_{i,j}}\rho d\Leb$.

  We claim that $\plaque_{i,j} = (\fstdr_{i,j}, \rho_{i,j})$ is a
  regular standard patch and that the convex combination of
  $\plaque_{i,j}$ with weights $c_{i,j}$ is a proper patch family.  In
  fact, by definition
  \begin{align*}
    \mathcal\shor(\plaque_{i,j}) &= \minshor &\mathcal\Rough(\plaque_{i,j}) &\le \Const{\Econst\inv( \kappa+\kappa^{2})\ve}.
  \end{align*}
  Moreover, $\foli_{i,j}$ is a foliation by horizontal segments, hence
  $\partial_\theta\foli_{i,j} = \textrm{const}$ and
  Lemma~\ref{lem:disintegration-regularity} implies that
  $\plaque_{i,j}$ is regular (for sufficiently small $\ve$).  Finally
  $\El(\plaque_{i,j}) = \minshor^{\gamma}$ for $\ve$ small enough,
  hence the patch family obtained by the above construction is proper.
\end{proof}
\begin{rmk}
  In fact it is immediate to check that the above results holds also
  if $\|d\log \frac{d\nu}{d\Leb} \circ \chvar\|_{\cC^1}\le \kappa$ for
  any sufficiently small $\kappa$.
\end{rmk}

\section{Coupling}\label{sec:coupling}
We present in this section the coupling argument that is key to prove
exponential convergence and decay of correlations: we state the
crucial result as the following theorem.
\begin{thm}[Coupling]\label{thm:bootstrap}
  There exist constants $c\in(0,1)$ and $\logstep_\Co>0$ such that the
  following holds.  Let $\mathcal{G}_1$ and $\mathcal{G}_2$ be proper
  patch families; then for any $n\ge \logstep_\Co\,\eps\inv \log \eps\inv$
  there exist proper patch families $\widetilde{\mathcal{G}}_1$ and
  $\widetilde{\mathcal{G}}_2$ such that
	\[ F_{\eps*}^{n}( \mu_{\mathcal{G}_1} - \mu_{\mathcal{G}_2}) = c(\mu_{\widetilde{\mathcal{G}}_1} - \mu_{\widetilde{\mathcal{G}}_2}). \]
\end{thm}
The above result will be proved in
Subsections~\ref{sec:reduct-dynam-near}
and~\ref{sec:coupling_near_sink}; we now show how to use it for
proving our Main Theorem.
\begin{proof}[Proof of Theorem~\ref{thm:doc_mostly_expanding}]
  Let $\cG_1$, $\cG_2$ be proper patch families and let $n\ge 1$.
  Repeated applications of Theorem~\ref{thm:bootstrap} guarantee the
  existence of patch families $\cH_1,\cH_2$ such that
  \begin{align*}
    F_{\eps*}^n (\mu_{\cG_1}-\mu_{\cG_2})=c^{\lfloor \eps
    n/(\logstep_\Co \log \vei)\rfloor}(\mu_{\cH_1}-\mu_{\cH_2}),
  \end{align*}
  so in particular
  \begin{equation}\label{eq:memory_loss_iterated}
	\|F_{\eps*}^n (\mu_{\cG_1}-\mu_{\cG_2})\|_{\TV} \le 2c^{\lfloor \eps n/(\logstep_\Co \log \vei)\rfloor},
  \end{equation}
  where $\|\cdot\|_{\TV}$ denotes the total variation norm.  Let $\cG$
  be an arbitrary proper patch family;
  Corollary~\ref{cor:proper_invariant} implies that, for any $m$
  sufficiently large, $F_{\eps*}^m \mu_{\cG}$ admits a disintegration
  as a proper patch family.  Applying~\eqref{eq:memory_loss_iterated}
  to $\cG$ and one of such image patch families implies that
  \begin{align*}
    \|F_{\eps*}^n \mu_{\cG}-F_{\eps*}^{n+m}
    \mu_{\cG}\|_{\TV}=\|F_{\eps*}^n (\mu_{\cG}-F_{\eps*}^m
    \mu_{\cG})\|_{\TV}\le 2 e^{-c_{\ve}n},
  \end{align*}
  where $c_{\ve} = -\logstep_{\Co}\inv\log c \cdot \eps /\log \vei$.
  The above bound shows that the sequence $F_{\eps*}^n \mu_{\cG}$ is
  Cauchy and thus converges to a limit probability measure; moreover,
  by~\eqref{eq:memory_loss_iterated}, the limit of
  $F_{\eps*}^n \mu_{\cG}$ is independent of the initial proper patch
  family $\cG$.  We denote this common limit probability measure by
  $\mu_\eps$: observe that $\mu_\eps$ is $F_\eps$-invariant by design
  and, by definition of the total variation norm:
  \begin{equation}\label{eq:equidistribution}
	|\mu_{\cG}(B\circ F_\eps^n)-\mu_\eps(B)|\le 2 e^{-c_{\ve}n}
  \end{equation}
  for any measurable function $B:\T^2 \to \R$ such that
  $\sup|B|\le 1$.  Since the measures $F_{\eps*}^n \mu_{\cG}$ are
  absolutely continuous with respect to Lebesgue for all $n$, so is
  $\mu_\eps$. It follows that the bound~\eqref{eq:equidistribution}
  applies whenever $\|B\|_{L^\infty(\Leb)}\le 1$.

  Next we show decay of correlations for observables
  $A\in \cC^2(\bT^2,\R)$ and $B\in L^\infty(\Leb)$. First consider the
  case where $\|d\log A\|_{\cC^1}\le \kappa$, where $\kappa$ is as in
  Lemma~\ref{lem:smooth_density_proper}. Then
  $\Leb(A\cdot B\circ F_\eps^n)=\Leb(A)\nu(B\circ F_\eps^n)$, where
  $\nu$ is the absolutely continuous probability measure defined by
  $d\nu/d\Leb = A/\Leb(A)$. By Lemma~\ref{lem:smooth_density_proper},
  $\nu$ admits a disintegration as a proper patch family, so
  by~\eqref{eq:equidistribution}, we have
  \begin{align*}
    |\Leb(A\cdot B\circ F_\eps^n)-\Leb(A)\mu_\eps(B)|\le
    2|\Leb(A)|\|B\|_{L^\infty(\Leb)} e^{-c_{\ve}n}.
  \end{align*}
  Next we consider the case of an arbitrary $\cC^2$ observable
  $A$. Without loss of generality, we can restrict to the case where
  $\|A\|_{\cC^2}\le 1$. Let $\alpha > 0$ be sufficiently big such that
  $\|d\log(A+\alpha)\|_{\cC^1}\le \kappa$ whenever
  $\|A\|_{\cC^2}\le 1$. Then $A= (A+\alpha)-\alpha$ can be written as
  the difference of two observables $A_i$ that satisfy
  $\|d\log A_i\|_{\cC^1}\le \kappa$.  Thus for any $A\in \cC^2$ and
  any $B\in L^\infty(\Leb)$, we have
  \begin{equation*}
	|\Leb(A\cdot B\circ F_\eps^n)-\Leb(A)\mu_\eps(B)|\le \Const
    \|A\|_{\cC^2}\|B\|_{L^\infty(\Leb)}\, e^{-c_{\ve} n}.
  \end{equation*}

  It remains to show that $\mu_\eps$ is the unique physical measure
  for $F_\eps$.
  % if discuss basin It remains to show that $\mu_\eps$ is the unique
  % physical measure for $F_\eps$ and that its basin has full Lebesgue
  % measure.
  Since $\cC^2$ is dense in $L^1(\Leb)$, by a standard approximation
  argument we have
  $\Leb(v\cdot w\circ F_\eps^n)\to \Leb(v)\mu_\eps(w)$ for all
  $v\in L^1(\Leb), w \in L^\infty(\Leb)$.  In particular, since
  $\mu_{\ve}\ll\Leb$, the above shows that $\mu_{\ve}$ is mixing, hence
  ergodic; it is then immediate to conclude, once again by absolute
  continuity, that $\mu_{\ve}$ is a physical measure.

  Let $\nu_{\ve}$ be a (possibly different) physical measure for
  $F_{\ve}$ and let $B(\nu_{\ve})$ denote its basin.\ignore{\footnote{add
    definition of basin}} On the one hand,
  by setting $v =\mathbf{1}_{B(\nu_{\ve})}/\Leb(B(\nu_{\ve}))$ –where
  $\mathbf{1}_{B(\nu_{\ve})}$ denotes the indicator function of
  $B(\nu_{\ve})$– we have
  \begin{align*}
    \frac1n\sum_{j = 0}^{n-1}\Leb(v\cdot w\circ F_{\ve}^{j}) = \frac{1}{\Leb(B(\nu_{\ve}))}\int_{B(\nu_{\ve})}\frac{1}{n}\sum_{j=0}^{n-1}w\circ
    F_\eps^j \, d\Leb \to \mu_{\ve}(w).
  \end{align*}
  On the other hand, by the Dominated Convergence Theorem and the
  definition of the basin of $\nu_{\ve}$ we have for all
  $w \in \cC^0$:
  \begin{align*}
    \frac{1}{\Leb(B(\nu_{\ve}))}\int_{B(\nu_{\ve})}\frac{1}{n}\sum_{j=0}^{n-1}w\circ
    F_\eps^j \, d\Leb\to \nu_{\ve}(w).
  \end{align*}
  It follows that $\nu_{\ve} = \mu_{\eps}$.
  % Finally, since $B(\mu_\eps)$ is $F_\eps$-invariant and $\mu_\eps$
  % is ergodic, we have
  % \[ \Leb(B(\mu_\eps))=\lim_{n\to \infty}\Leb(\mathbf{1}_{B(\mu_\eps)}\circ F_\eps^n)=\mu_\eps(\mathbf{1}_{B(\mu_\eps)})=1. \]
\end{proof}
\begin{rmk}
  It is worthwhile to observe that, if a measure $\mu_{\ve}$
  satisfies~\eqref{eq:decay-of-correlations}, it is necessarily a
  physical measure, regardless of absolute continuity  with respect to
  Lebesgue.  In fact~\cite[Theorem 3.3]{castorrini-liverani}
  guarantees that $F_\eps$ admits (possibly several) ergodic physical
  measures; call one of them $\nu_{\ve}$; then the closing argument of
  the proof shows that $\mu_{\ve}$ and $\nu_{\ve}$ must coincide,
  hence $\mu_{\ve}$ is physical.
\end{rmk}
\subsection{Proof of the Coupling Theorem – Bootstrap}\label{sec:reduct-dynam-near}
We now begin with the proof of Theorem~\ref{thm:bootstrap}; the first
important observation is that it suffices to prove the result under
more favourable assumptions.
\begin{thm}\label{thm:coupling_near_sink}
  Let $C > 0$ be as in Corollary~\ref{cor:convenient}. For any $Q > 0$
  sufficiently large, there exist constants $c'\in(0,1)$ and
  $\logstep_2>0$ such that the following holds. For $i=1,2$, let
  $\cH_i$ be a patch family such that any standard patch $\plaque$ in
  $\cH_i$ is supported on $\bT\times B(\theta_{-},2C\sqrt{\eps})$ and
  satisfies $\El(\plaque)\le Q$.  Then, for any
  $n\ge \logstep_{2} \vei \log\vei$, there exist proper patch families
  $\widetilde{\cH}_1$ and $\widetilde{\cH}_2$ such that
  \begin{align*}
    F_{\eps*}^{n}( \mu_{\cH_1} - \mu_{\cH_2}) =
    c'(\mu_{\widetilde{\cH}_1} - \mu_{\widetilde{\cH}_2}).
  \end{align*}
\end{thm}
\ignore{the assumption that the patch families $\cG_{i}$ involved in
  the statement belong to a neighbourhood of the sink.  In fact by
  Corollary~\ref{cor:convenient}, there exists $C, \logstep > 0$ so
  that the pushforward of any patch family by $V(\vei\log\vei)$
  iterates of the dynamics assigns at least half of its mass to
  $\bT\times B(\theta_{-},C\sqrt{\eps})$. Thus our main challenge\
  reduces to prove that we can couple standard patches supported on
  $\bT\times B(\theta_-,C\sqrt{\eps})$; this is the content of the
  following theorem.  }

We will prove Theorem~\ref{thm:coupling_near_sink} in
Subsection~\ref{sec:coupling_near_sink}.  We now show how it implies
Theorem~\ref{thm:bootstrap}.  Recall that in
Section~\ref{sec:std_pairs} we fixed a timescale $\clock$ (and set
$\clockN = \lfloor \clock\vei \rfloor$).
\begin{proof}[Proof of Theorem~\ref{thm:bootstrap}]
  Let $i\in \{1,2\}$. By
  Remark~\ref{rmk:standard-patches-are-families}, $\mu_{\cG_i}$ admits
  a representation as a family of regular standard pairs.  Thus by
  Corollary~\ref{cor:convenient}, there exists $\logstep$ so that
  whenever $m'\ge \lfloor\logstep\log\vei\rfloor\clockN$ we have:
  \begin{align*}
    \mu_{\cG_i}(\theta_{m'} \in B(\theta_-,C\sqrt{\eps}))\ge \tfrac12.
  \end{align*}
  Moreover, by Corollary~\ref{cor:proper_invariant}, for
  $m'\ge \logstep_{1}\vei\log\vei$, we can choose a proper patch family
\  $\cH_i \in [F_{\eps*}^{m'}{\cG_i}]$. For $\eps$ small enough, any
  standard rectangle that intersects
  $\T\times B(\theta_-,C\sqrt{\eps})$ is contained in
  $\T\times B(\theta_-,2C\sqrt{\eps})$; we thus obtain that
  \begin{align*}
    \lambda_{\cH_i}(\supp\plaque\subset\T\times B(\theta_-,2C\sqrt{\eps}))\ge \tfrac{1}{2}.
  \end{align*}
  Let
  $S = \{\plaque \in \plaqueSet:\supp\plaque \subset
  B(\theta_-,2C\sqrt{\eps}),\El(\plaque)\le 8\lyapFnDrift\}$, where
  $\lyapFnDrift$ is given by Proposition~\ref{prop:geometric_drift},
  and define $p_i = \lambda_{\cH_i}(\plaque \in S)$.  Since $\cH_i$ is
  proper, $\lambda_{\cH_i}(\El(\plaque))\le 2\lyapFnDrift$, so
  Markov's inequality implies that
  \begin{align*}
    p_i \ge \tfrac{1}{2} - \lambda_{\cH_i}(\El(\plaque)> 8\lyapFnDrift)\ge \tfrac14.
  \end{align*}
  Consider the patch family
  ${\cH}^*_i = \tfrac{1}{4}\cH_i|\{\plaque\in
  S\}+\tfrac{3}{4}\cE_i$, where
  \begin{align*}
    \cE_i = \tfrac{4}{3}(p_i-\tfrac{1}{4})\cH_i|\{\plaque\in
    S\}+\tfrac43 (1-p_i)\cH_i|\{\plaque \notin S\}.
  \end{align*}
  Observe that ${\cH}^{*}_i \in [\cH_i]$ and
  $\lambda_{{\cH}^{*}_i}(\El(\plaque))=\lambda_{\cH_i}(\El(\plaque))$,
  so in particular
  \begin{align*}
    \lambda_{\cE_i}(\El(\plaque))\le \tfrac43
    \lambda_{{\cH}^{*}_i}(\El(\plaque))\le \tfrac83
    \lyapFnDrift.
  \end{align*}
%%%%%%%%%%%%%%%%%%%%%%%%%%%%%%%%%%%%%%%%%%%%%%%%%%%%%%%%%%%%
  By Proposition~\ref{prop:geometric_drift}, it follows that, for
  $\eps$ small enough and for any $m''\ge \logstep_{1}\vei\log\vei$,
  we can choose a proper patch family
  $\widehat{\cE}_i\in [F_{\eps*}^{m''} {\cE_i}]$.

  We can now apply Theorem~\ref{thm:coupling_near_sink} to
  ${\cH}^{*}_{i}$, with $Q = \frac83B$ and obtain, for
  $m''\ge \logstep_2 \vei\log\vei$, proper patch families
  $\widetilde{\cH}^{*}_i$  such that
  \begin{align*}
    F_{\eps*}^{m''}( \mu_{\cH_1|\{\plaque \in S\}} -
    \mu_{\cH_2|\{\plaque \in S\}}) = c'(\mu_{\widetilde{\cH}^{*}_1} - \mu_{\widetilde{\cH}^{*}_2}).
  \end{align*}
  Thus we obtain that
  \begin{align*}
    F_{\eps*}^{m'+m''}(\mu_{\cG_1}-\mu_{\cG_2}) &=F_{\eps*}^{m''}(\mu_{{\cH}^{*}_1}-\mu_{{\cH}^{*}_2})=\tfrac34(\mu_{\widetilde{\cE}_1}-\mu_{\widetilde{\cE}_2})+\tfrac14 c'(\mu_{\widetilde{\cH}^{*}_1} - \mu_{\widetilde{\cH}^{*}_2})\\
                                               &= c(\mu_{\widetilde{\cG}_1}-\mu_{\widetilde{\cG}_2}),
	\end{align*}
	where $c = (\tfrac 34 + \tfrac14 c')$ and
    \begin{align*}
      \widetilde{\cG}_i = \frac{3}{3 + c'}\,\widetilde{\cE}_i + \frac{ c'}{3 + c'}\,\widetilde{\cH}^{*}_i.
    \end{align*}
	Note that $\widetilde{\cG}_i$ is proper since it is a convex
    combination of proper patch families.  The theorem follows
    choosing $\logstep_{\Co} =
    2\max\{\logstep\clock,\logstep_1,\logstep_2\}$ so that if $n >
    \logstep_{\Co}\vei\log\ve\inv$ we can find $m', m''$ as above so that $n = m'+m''$.
\end{proof}
\subsection{Proof of the Coupling Theorem –
  Conclusion}\label{sec:coupling_near_sink}
In this section we prove Theorem~\ref{thm:coupling_near_sink},
concluding the proof of the Coupling argument.  We prove this theorem
by pushing forward $\cH_1$ and $\cH_2$ to obtain standard patches
whose supports overlap.  Let $I\subset B(\theta_-,C\sqrt{\eps})$ be an
arbitrary interval such that $|I|\ge\eps$; by
Lemma~\ref{lem:lclt_corol} there exists $p > 0$ so that, for $\clock$
sufficiently large:
$\mu_{\cH_i}(\theta_{\clockN}\in I)\ge p\eps^{1/2}$ for $i=1,2$.
However, if two standard patches intersect only near their boundaries,
then only a very small portion of their mass can be
coupled. Consequently, given a standard patch $\plaque$, we shall
consider intervals $I$ for which $\supp \plaque$ has a substantial
overlap with $\bT \times I$.

Let us now be more precise. By
Lemma~\ref{lem:patch_family_a_priori_bd}(a), for $i=1,2$ we can choose
patch families $\widehat{\cH}_i\in [F_{\eps*}^{\clockN}{\cH_i}]$
such that $\cM(\plaque)\le \widehat{Q}:=e^{\alpha_1 \clock}(1+Q)$ for
all $\plaque$ in $\widehat{\cH}_i$. In particular, by item (b) in the
definition of standard rectangle,
$|\pi_2 \supp \plaque|\ge \widehat{Q}\inv \patchheight\eps$ for all
$\plaque$ in $\widehat{\cH}_i$. Moreover, by the above paragraph we
have
\begin{equation}\label{eq:llt_patch_family}
	\lambda_{\widehat{\cH}_i}(\pi_2\supp\plaque\text{ intersects }I)\ge p\eps^{1/2}\text{ for }i=1,2,
\end{equation}
for any interval $I\subset B(\theta_{-},C\sqrt{\eps})$ such that $|I|\ge \eps$.

To ensure `substantial overlap' with a given standard patch $\plaque$, we consider intervals $I$ of length $L_\eps=\widehat{Q}\inv\patchheight\eps/10$ such that $I\subset \pi_2 \supp\plaque$. Next we show that we can find many such intervals that can be used for coupling:
\begin{sublem}\label{sublem:finding_good_bins}
	Let $\patchheight\ge 10\widehat Q$ . Then there exists $s>0$ and a collection of $\lfloor s\eps^{-1/2}\rfloor$ intervals $\{I_j\}$ of length $L_\eps$ such that\footnote{For sets $S_1, S_2 \subset \T$ we denote $d(S_1,S_2)=\inf_{p_1\in S_1, p_2\in S_2}d(p_1,p_2)$.} $d(I_j,I_k)\ge 2\patchheight \eps$ for $j\ne k$ and $\lambda_{\widehat{\cH}_i}(I_j \subset \pi_2 \supp \plaque)\ge \frac{1}{2}p\eps^{1/2}$ for $i=1,2$ and all $j$.
\end{sublem}
\begin{proof}
  Let $\{S_k\}_{k=1}^5$ be adjacent intervals of length $L_\eps$ in
  $B(\theta_,C\sqrt{\eps})$. We claim that
  \begin{equation*}
    \#\{k\in\{1,\cdots,5\}:\lambda_{\widehat{\cH}_i}(S_k \subset \pi_2
    \supp \plaque)\ge \tfrac12 p\eps^{1/2}\}\ge 3
  \end{equation*}
  for $i=1,2$. It follows by the pigeon-hole principle that we can
  choose $k\in\{1,\cdots,5\}$ such that
  \[\lambda_{\widehat{\cH}_i}(S_k \subset \pi_2 \supp \plaque)\ge
    \tfrac12 p\eps^{1/2}\quad\text{ for $i=1,2$.}\]

  The proof of the sub-lemma follows since we can choose
  $\lfloor 2C\eps^{-1/2}/(3\patchheight)\rfloor$ intervals $\{J_j\}$
  of length $5L_\eps$ in $B(\theta_-,C\sqrt{\eps})$ such that
  $d(J_j,J_k)\ge 2\patchheight\eps$ for $j\ne k$, then choose
  intervals $I_j\subset J_j$ of length $L_\eps$ such that
  $\lambda_{\widehat\cH_i}(I_j \subset \pi_2\supp\plaque)\ge \tfrac12
  p\eps^{1/2}$ for $i=1,2$.

  It remains to prove the claim. Fix $i\in \{1,2\}$ and let $\plaque$
  be a standard patch in $\widehat\cH_i$.  As observed before,
  $\pi_2 \supp\plaque$ is an interval of length at least
  $\widehat Q\inv\patchheight\eps=10L_\eps$; thus if
  $\pi_2 \supp\plaque$ intersects $S_3$ then either
  $S_1\cup S_2 \subset \pi_2\supp\plaque$ or
  $S_4\cup S_5 \subset \pi_2\supp\plaque$. Since
  $\patchheight\ge 10\widehat Q$, we have $L_\eps\ge \eps$. Hence
  by~\eqref{eq:llt_patch_family}, it follows that either
  \[\lambda_{\widehat\cH_i}(S_1\cup S_2 \subset \pi_2\supp\plaque)\ge
    \tfrac12 p\eps^{1/2}\text{ or }\lambda_{\widehat\cH_i}(S_4\cup S_5
    \subset \pi_2\supp\plaque)\ge \tfrac12 p\eps^{1/2}.\] Without
  loss, we may suppose that
  $\lambda_{\widehat\cH_i}(S_1\cup S_2 \subset \pi_2\supp\plaque)\ge
  \tfrac12 p\eps^{1/2}$. By applying the same reasoning as above with
  $S_4$ in place of $S_3$, either
  \[\lambda_{\widehat\cH_i}(S_3 \subset \pi_2\supp\plaque)\ge \tfrac12
    p\eps^{1/2}\text{ or }\lambda_{\widehat\cH_i}(S_5 \subset
    \pi_2\supp\plaque)\ge \tfrac12 p\eps^{1/ 2},\] which completes the
  proof of the claim.
\end{proof}
\begin{prp}\label{prop:coupling_on_good_bins}
	For $\patchheight$ large enough, there exists $p'>0$, $\logstep_{4}>0$ such that the following holds for all $\eps$ sufficiently small. Let $\plaque^{(i)}$, $i=1,2$ be \regular{} standard patches such that $\cM(\plaque_i)\le \widehat{Q}$. Suppose that there exists an interval $I$ of length $L_\eps$ such that $I \subset \pi_2 \supp\plaque_1 \cap \pi_2 \supp\plaque_2$. Then for all $n\ge \logstep_{4}\vei\log\vei$ there exists a Borel probability measure $m$ on $\T^2$ and proper patch families $\cG_1$, $\cG_2$ such that
	\begin{equation}\label{eq:coupling_on_good_bins_via_ref_measure}
		F_{\eps*}^{n}\mu_{\plaque_i} = (1-p')m + p' \mu_{\cG_{i}} \quad\text{ for $i=1,2$.}
	\end{equation}
\end{prp}
\begin{proof}
	The proof of this proposition splits into several steps.
	\begin{enumerate}[label=\textbf{Step \arabic*}:,leftmargin=0pt, labelwidth=*, labelsep=0.5em, align=left]
		\item\label{step:strip_inside} Write $\stdr^{(i)} = \supp \plaque^{(i)}$. There exists $n_1=\cO(1)$ such that $F_\eps^{n_1}\stdr^{(1)}\cap F_\eps^{n_1}\stdr^{(2)}$ contains a strip for $\patchheight$ large enough.
	\end{enumerate}

	Write $I = B(\theta_0,L_\eps/2)$. More precisely, we claim that for $\patchheight$ sufficiently large, there exists $n_1$ such that for all $\eps>0$ we have $\bT \times B(\theta_0,L_\eps/4) \subset F_\eps^{n_1} \stdr^{(i)}$ for $i=1,2$. Indeed, fix $i\in \{1,2\}$ and let $\bG_0^{(i)}$ and $\bG_1^{(i)}$ denote the bottom and top boundary standard curves of $\stdr^{(i)}$, respectively. Choose $n_1$ such that $3^{n_1} \delta/2>1$. Then since any standard curve is an unstable curve, by~\eqref{eq:unstable_cone_expansion}, $\pi_1 F_\eps^{n_1} \bG_0 = \pi_1 F_\eps^{n_1} \bG^{(i)}_1=\bT$. Observe that $\bG^{(i)}_0$ and $\bG^{(i)}_1$ are of height at most $\eps\delta\scc1$ and $|\pi_2 F^{n_1}_\eps p - \pi_2 p|\le \Const {n_1}\eps$ for any $p\in \bT^2$. Thus for $\patchheight$ sufficiently large, $F_\eps^{n_1} \bG^{(i)}_0$ and $F_\eps^{n_1} \bG^{(i)}_1$ lie below and above $\bT\times B(\theta_0,L_\eps/4)$, respectively.
	%	Consequently, any point $p\in \bT\times B(\theta_0,L_\eps/4)$ lies on a centre manifold $\Wc$ that intersects $F_\eps^n \bG^{(i)}_0$ and $F_\eps^n \bG^{(i)}_1$ once. By integrability and invariance of the centre slope field, $\Wc$ is contained in the image of a centre manifold in $\stdr^{(i)}$, so in particular $p \in F_\eps^n \stdr^{(i)}$.
	Let $p\in \bT\times B(\theta_0,L_\eps/4)$. Choose a local centre
    manifold $\Wc$ from $p$ to a point $q\in F_\eps^{n_1} \bG^{(i)}_0$ and
    let $r\in \bG^{(i)}_0$ be such that $F_\eps^{n_1} r=q$. Then $\Wc$ and
    $F_\eps^{n_1}\Wc_{\stdr^{(i)}}(r)$ are both local centre manifolds
    with endpoint $q$, so by unique integrability
    (Theorem~\ref{thm:unique-integrability-centre-mfolds}) either
    $\Wc\subseteq F_\eps^{n_1}\Wc_{\stdr^{(i)}}(r)$ or
    $F_\eps^{n_1}\Wc_{\stdr^{(i)}}(r)\subseteq \Wc$. Since
    $F_\eps^{n_1}\Wc_{\stdr^{(i)}}(r)$ intersects $F_\eps^{n_1} \bG^{(i)}_1$
    and $\Wc$ does not, it follows that
    $\Wc \subset F_\eps^{n_1}\Wc_{\stdr^{(i)}}(r)$. Hence
    $p \in F_\eps^{n_1}\Wc_{\stdr^{(i)}}(r)\subset F_\eps^{n_1} \stdr^{(i)}$;
    since $p$ was arbitrary, this concludes the proof of step 1.
	\begin{enumerate}[label=\textbf{Step \arabic*}:, resume,leftmargin=0pt, labelwidth=*, labelsep=0.5em, align=left]
		\item Let $i\in\{1,2\}$. There exist regular standard patches $\bar{\plaque}^{(i)}_j$ such that $\mathcal\Rough(\bar{\plaque}^{(i)}_j)\le \Const$ and $\mathcal\shor(\bar{\plaque}^{(i)}_j)\le 2\widehat{Q}$ along with weights $c_j^{(i)}>0$ so that
		\begin{equation}\label{eq:pushforward_explicit_decomp}
			F_{\eps*}^{n_1} \mu_{\plaque^{(i)}} = \sum_j c_j^{(i)}\mu_{\bplaque_j^{(i)}}.
		\end{equation}
		Moreover, there exists a constant $\mathfrak{b}>0$ uniform in $\eps$ such that $c_j^{(i)}> \mathfrak{b}$.
	\end{enumerate}

	%	Fix $i\in\{1,2\}$. By Lemma~\ref{lem:invariance_patches}(a), there
	%	exists a collection of regular standard patches
	%	$\{\bplaque_j^{(i)}\}$ and associated diffeomorphisms
	%	$\varphi_j^{(i)}:\supp \bplaque_j^{(i)}\to \stdr^{(i)}$ such that
	%	\begin{equation}\label{eq:pushforward_explicit_decomp}
		%		F_{\eps*}^n \mu_{\plaque^{(i)}} = \sum_j c_j^{(i)}\mu_{\bplaque_j^{(i)}}, \text{ where }c_j^{(i)}=\mu_{\plaque}(\varphi_j^{(i)}\supp \bplaque_j^{(i)}).
		%	\end{equation}
	%	Moreover, $\mathcal{\Rough}(\bplaque_j^{(i)})\le \Const$,
	%	$\mathcal{\shor}(\bplaque_j^{(i)})\le e^{\centreMaxExp
		%		n\ve}\widehat{Q}\le 2\widehat{Q}$ for $\eps$ small enough and
	%	$c_{j}^{(i)} > \Const$ (observe that $\Rough$ is bounded by
	%	$\widehat Q$).

	Since $\cM(\plaque^{(i)})\le \widehat{Q}$ and $n_1=O(1)$ this step follows immediately from Proposition~\ref{lem:invariance_patches} and Lemma~\ref{lem:lower_bd_weights}.

	\begin{enumerate}[label=\textbf{Step \arabic*}:, resume,leftmargin=0pt, labelwidth=*, labelsep=0.5em, align=left]
		\item There exist indices $j_1,j_2$ such that $\supp \plaque^{(1)}_{j_1}\cap \supp \plaque^{(1)}_{j_1}$ contains $A''\times B''$ where $A''$ and $B''$ are intervals such that $|A''|\ge \delta/9$ and $|B''|\ge L_{\eps}/20$.
	\end{enumerate}

	Since
	$\bT\times B(\theta_0,L_\eps/4)\subset F_\eps^{n_1} \stdr_1$ we can
	choose $j_1$ such that $\supp \bplaque_{j_1}^{(1)}$ intersects
	$\{\theta=\theta_0\}$. Now by
	Lemma~\ref{lem:straight_rect_in_std_rect}, we can choose intervals
	$A$, $B$ such that $A\times B\subset \supp \bplaque_{j_1}^{(1)}$,
	$|A|\ge \delta/3$, $|B|\ge \widehat Q\inv \Delta\eps/4$ and
	$d(\theta_0,B)\le \scc1 \eps\delta$. Observe that
	$B' = B \cap B(\theta_0,L_\eps/4)$ satisfies $|B'|\ge L_\eps/5$
	for $\patchheight$ large enough. Now choose $j_2$ such that $\supp \plaque_{j_2}^{(2)}$ intersects the midpoint of $A\times B'$. Then by the same argument there exist intervals $A''$, $B''$ such that \[A'' \times B'' \subset \supp \bplaque_{j_1}^{(1)}\cap \supp \bplaque_{j_2}^{(2)},\]
	$|A''|= \delta/9$, $|B''|= L_\eps/20$.
	To ease notation, let us write $\widetilde{\plaque}^{(i)}=\bar{\plaque}_{j_i}^{(i)}$ and $\widetilde{\stdr}^{(i)}=\supp \widetilde{\plaque}^{(i)}$ for $i=1,2$.

	%	\todo[inline]{Step 4: use $A'' \times B''$ to construct a standard rectangle contained in $\supp \bplaque_{j_1}^{(1)}\cap \supp \bplaque_{j_2}^{(2)}$}
	\begin{enumerate}[label=\textbf{Step \arabic*}:, resume,leftmargin=0pt, labelwidth=*, labelsep=0.5em, align=left]
		\item Use $A'' \times B''$ to construct a (mod 0) partition of $\widetilde{\stdr}^{(i)}$ into $(30,\Const)$-standard rectangles for $i\in \{1,2\}$, including an element $S$ that is common to both partitions.
	\end{enumerate}

	Let $A'''$ and $B'''$ denote the middle third of $A''$ and $B''$, respectively. Let $b_1$, $b_2$ be the endpoints of $B'''$ and $q_1$, $q_2$ be the endpoints of $A''' \times \{b_1\}$. Then $\Wc_{\widetilde{\stdr}^{(i)}}(p_1)$, $\Wc_{\widetilde{\stdr}^{(i)}}(p_2)$, $\T\times \{b_1\}$ and $\T\times \{b_2\}$ partition $\widetilde{\stdr}^{(i)}$ into nine closed regions $P_j^{(i)}$ (see Figure~\ref{fig:splitting_into_standard}).
	\begin{figure}[ht]
		\centering
		\begin{tikzpicture}
			\coordinate (p1_curve_top) at (2.6,0);
			\coordinate (p1_curve_bottom) at (3,-3.8);

			\coordinate (p2_curve_top) at (4.2,0.7);
			\coordinate (p2_curve_bottom) at  (4.3,-3.9);

			\draw [thick, black] (-0.2,0) to [out=35,in=180] (p1_curve_top)
			to [out=15,in=165] (p2_curve_top) to [out=0,in=180] (6.5,1.2); % upper boundary

			\draw [thick, black] (-0.2,-4) to [out=35,in=180] (p1_curve_bottom) to [out=-10,in=160] (p2_curve_bottom) to [out=0,in=160] (6.9,-3.6) ; % lower boundary

			\draw [thick, black] (0,0.4) to [out=-35,in=85] (0,-1) to [out=-105,in=100] (0,-2) to [out=-80,in=100] (0.5,-3) to [out=-90,in=75] (0,-4.2) ; % left boundary centre manifold

			\draw [thick, black] (6.2,1.4) to [out=-35,in=85] (6.2,-1) to [out=-105,in=100] (6.7,-2) to [out=-80,in=100] (6.5,-3) to [out=-80,in=75] (6.5,-3.7) ; % right boundary centre manifold

			%			\draw [thick, black] (3,0) to [out=-35,in=85] (3,-1) to [out=-105,in=100] (3,-2) to [out=-80,in=100] (3.5,-3) to [out=-90,in=75] (3,-3.8) ; % middle right centre manifold

			%			\node [left] at (-0.1,0.3) {$p$};
			%			\node [right] at (6,-3.7) {$p'$};
			\draw[dashed, gray, thick] (0,-1) -- (6.2,-1);
			\draw[dashed, gray, thick] (0,-2) -- (6.7, -2);
			%			\node [below] at (3,-3.8) {$\bG_0$%$\cW^\ve_{u}(p,\bar \delta)$
				%			};
			\node [left] at (2.9,-3.3) {$\Wc_{\widetilde{\stdr}^{(i)}}\!(p_1)$};
			\node [above right] at (4.5,-3.6) {$\Wc_{\widetilde{\stdr}^{(i)}}\!(p_2)$};
			%			\node [below] at (3,-1.5) {$K$};

			% middle-left centre manifold
			\draw [thick, black]
			(p1_curve_top)
			to [out=-35,in=85] (2.5,-0.8)
			to [out=-105,in=100] (3,-1.8)
			to [out=-80,in=100] (2.8,-3)
			to [out=-90,in=75] (p1_curve_bottom);

			% middle-right centre manifold
			\draw [thick, black]
			(p2_curve_top)
			to [out=-35,in=85] (4.2,-0.7)
			to [out=-105,in=100] (4.7,-1.7)
			to [out=-80,in=100] (4.5,-3)
			to [out=-90,in=75] (p2_curve_bottom);

			\node[below left] at (2.9,-2) {$p_1$};
			\node[below right] at (4.6,-2) {$p_2$};
			\node at (3.8,-1.5) {$S$};
		\end{tikzpicture}
		\caption{Partition of $\widetilde{\stdr}^{(i)}$ into nine closed regions, including the central region $S$}
		\label{fig:splitting_into_standard}
	\end{figure}
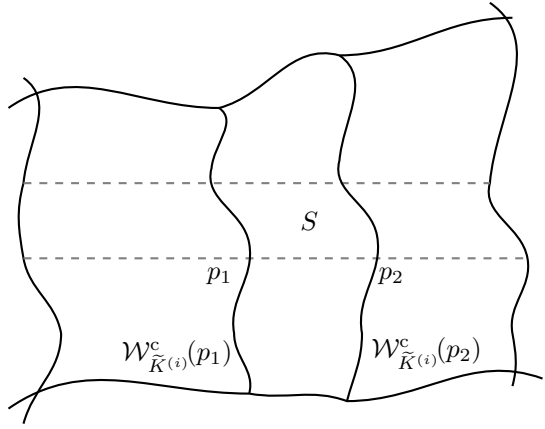
	%	Let $S$ denote the central region, which is is common to the partitions of $\widetilde{\stdr}^{(1)}$ and $\widetilde{\stdr}^{(2)}$. % alt phrasing

	Observe that the central region, which we denote by $S$, is common to the partitions of $\widetilde{\stdr}^{(1)}$ and $\widetilde{\stdr}^{(2)}$.

	It remains to show that the regions $P_j^{(i)}$ are standard rectangles. First note that one of the top/bottom boundary curves of $P_j^{(i)}$ is a horizontal line segment of width at least $|A'''|=\frac{1}{27}\delta$ and the other is either a horizontal line segment or a subcurve of a standard curve. Since any centre manifold $\Wc$ in $\widetilde{\stdr}^{(i)}$ satisfies $|\pi_2\Wc|\le \centreConeConst\patchheight\eps$ it follows that the other one of these boundary curves is of width at least $|A'''|-2\centreConeConst\patchheight\eps\ge \frac{1}{30}\delta$ for $\eps$ small enough, so both curves are $30$-standard curves. Next we note that the left and right boundary curves of $P_j^{(i)}$ are centre manifolds, one of which is of at height at least $|B'''|$. Hence by~\eqref{eq:difference-in-height-for-Wc}, for any $q\in P_j^{(i)}$ we have $\height \Wc_{P_j^{(i)}}(q)\ge |B'''|-2\pscc1\delta\ve\ge \Const \eps$ for $\patchheight$ large enough, so $P_j^{(i)}$ is a $(30,\Const)$-standard rectangle.
	%	First note that the top and bottom boundary curves of $P_j^{(i)}$ are either horizontal line segments or subcurves of a standard curve. % alt phrasing
	%	\vspace{3in}
	\begin{enumerate}[label=\textbf{Step \arabic*}:, resume,leftmargin=0pt, labelwidth=*, labelsep=0.5em, align=left]
		\item Let $m = \Leb(S)\inv\Leb|_S$. There exists $\tau\in(0,1)$ uniform in $\eps$ such that for $i\in \{1,2\}$ we can write $\mu_{\wplaque^{(i)}}(S)\inv \mu_{\wplaque^{(i)}}|_S$ as a convex combination
		\begin{equation}\label{eq:coupling_on_S}
			\mu_{\wplaque^{(i)}}(S)\inv \mu_{\wplaque^{(i)}}|_S = \tau m_S + \sum_j d_j^{(i)}\mu_{\bS^{(i)}_j},
		\end{equation}
		where $\bS^{(i)}_j$ are $((90,\Const),(\Const,\Const))$-prestandard patches.
	\end{enumerate}
		Let $\rho^{(i)}$ denote the density of $\mu_{\wplaque^{(i)}}(S)\inv \mu_{\wplaque^{(i)}}|_S$. Then $\rho^{(i)}\in \Stdd_{\Rough}(S)$ where $\Rough = \mathcal{\Rough}(\wplaque^{(i)})\le \Const$, so by Lemma~\ref{lem:split_density_leb}, there exists $\breve{\rho}^{(i)}\in \Stdd_{\Const}(S)$ and $\tau\in (0,1)$ uniform in $\eps$ such that
	\begin{equation}\label{eq:S_density_decomp_split_density}
		\rho^{(i)}=\tau \frac{1}{\Leb(S)}+(1-\tau)\breve \rho^{(i)}.
	\end{equation}
	Now since $S$ is a $(30,\Const)$ standard rectangle, Lemma~\ref{lem:create-foliation} implies that $S$ can be partitioned into $(90,\Const)$-prestandard rectangles $\mathcal{S}_j=(S_j,\eta_j)$ such that
	\[ \|d\log \partial_\theta \foli_j\|_\infty \le \Const ,\quad \|H\log \partial_\theta \foli_j\|_\infty \le \Const .\]
	Let $\hat{\rho}_j^{(i)}\in \Stdd_{\Const}(S_j)$ be the restriction of $\breve{\rho}^{(i)}$ to $S_j$, normalized to be a probability density.
	Then by Lemma~\ref{lem:disintegration-regularity}, it follows that $\mathbb{S}^{(i)}_j=(\mathcal{S}_j,\hat{\rho}_j^{(i)})$ is a $((90,\Const),(\Const,\Const))$-standard patch. Hence by~\eqref{eq:S_density_decomp_split_density} we gather that~\eqref{eq:coupling_on_S} holds with $d_j^{(i)}={(1-\tau)}\int_S \breve\rho^{(i)} d\Leb$.
	%	Hence by again applying Lemmata~\ref{lem:create-foliation} and~\ref{lem:disintegration-regularity} we obtain that there exist $(10,\Const,\Const,\Const)$-prestandard patches $\bP_j^{(i)}$ and constants $e_j^{(i)}>0$ such that
	\begin{enumerate}[label=\textbf{Step \arabic*}:, resume,leftmargin=0pt, labelwidth=*, labelsep=0.5em, align=left]
		\item Conclusion of the proof.
	\end{enumerate}
	Next consider the prestandard rectangles $P_j^{(i)}$ that partition $\widetilde{\stdr}^{(i)}_j$ other than $S$.
	Applying Lemmata~\ref{lem:create-foliation} and~\ref{lem:disintegration-regularity} to the standard rectangles $P_j^{(i)}$ with associated densities $d\mu_{\wplaque^{(i)}}/d\Leb$ normalized to be probability densities yields that there exist $((90,\Const),(\Const,\Const))$-prestandard patches $\bP_l^{(i)}$ and constants $e_l^{(i)}>0$ such that
	\begin{align*}
		\mu_{\wplaque^{(i)}}&=
		\mu_{\wplaque^{(i)}}|_S + \sum_{j\colon P_j^{(i)}\ne S} e^{(i)}_j \mu_{\bP_j^{(i)}}\\ &=
		\mu_{\wplaque^{(i)}}|_S + \sum_l e^{(i)}_l \mu_{\bP_l^{(i)}}\\
		&=\mu_{\wplaque^{(i)}}(S)\tau m_S + \mu_{\wplaque^{(i)}}(S)\sum_j d_j^{(i)}\mu_{\bS^{(i)}_j}+\sum_l e^{(i)}_l \mu_{\bP_l^{(i)}}\\
		&=: \mu_{\wplaque^{(i)}}(S)\tau m_S + (1-\mu_{\wplaque^{(i)}}(S)\tau)\nu^{(i)}.
	\end{align*}
	It follows that for any
    $0 < p' \le \tau\min_i c_{j_i}^{(i)}\mu_{\wplaque^{(i)}}(S)$
    we can write
    \begin{align*}
      c^{(i)}_{j_i}\mu_{\wplaque^{(i)}}=p' m_S +(c^{(i)}_{j_i}-p')\tilde \nu^{(i)}
    \end{align*}
    where $\tilde \nu^{(i)}$ is a convex combination of $\mu_{\wplaque^{(i)}}$ and $\nu^{(i)}$. Now by~\eqref{eq:patch_density_bds},
	\[ \mu_{\wplaque^{(i)}}(S)\ge \inf
      \frac{d\mu_{\wplaque^{(i)}}}{d\Leb}\Leb(S)\ge
      e^{-\Const}\frac{\Leb(S)}{\Leb(K)} \] and
    by~\eqref{eq:bound-on-Lebesgue-patch} it follows that
    $\mu_{\wplaque^{(i)}}(S)$ is bounded away from $0$ uniformly in
    $\eps$, so $p'>0$ can be chosen uniformly in $\eps$.

	Recall that $\wplaque^{(i)}=\bar{\plaque}_{j_i}^{(i)}$. Thus by~\eqref{eq:pushforward_explicit_decomp} we obtain that
	\[ F_{\eps*}^{n_1} \mu_{\plaque^{(i)}}= p'm_S+(1-p')\sigma^{(i)}, \]
    where $\sigma^{(i)}$ is a convex combination of the probability
    measures $\{\mu_{\bS_j^{(i)}}\}_j$, $\{\mu_{\bP_j^{(i)}}\}_j$ and
    $\{\mu_{\bar{\plaque}_j^{(i)}}\}_j$. Note that $\{\bS_j^{(i)}\}$,
    $\{\bP_j^{(i)}\}$, $\{\bar{\plaque}_j^{(i)}\}$ are all
    $((90,\Const),(\Const,\Const))$-prestandard
    patches.\footnote{Indeed, $\sigma^{(j)}$ would admit a
      representation as a standard family if not for the fact that
      these patches are not regular and only prestandard.} Hence by
    Lemma~\ref{lem:invariance_patches}, for $\eps$ sufficiently small,
    there exists a patch family
    $\cF_i \in [F_{\eps*}^{\lfloor\vei\rfloor} \sigma^{(i)}]$ such
    that $\cM(\plaque)\le \Const$ for any standard patch $\plaque$ in
    $\cF_i$. Finally, by Proposition~\ref{prop:geometric_drift}, for
    $\eps$ sufficiently small and $n_2\ge \logstep_{1}\vei\log\vei$
    there exists a proper patch family
    $\cG_i \in [F_{\eps*}^{n_2}\mu_{\cF_i}]$, so we gather that
	\[ F_{\eps*}^{n+[\vei]+n_2}\mu_{\plaque^{(i)}} = p' F_{\eps*}^{[\vei]+n_2}m_S+(1-p')\mu_{\cG_i}. \]
	Set $\logstep_{4}=2(n_1 + 1 + \logstep_1)$. Then for $\eps$ sufficiently small, any $n\ge \logstep_4\vei \log\vei$ can be written in the form $n=n_1+[\vei]+n_2$ with $n_2\ge \logstep_1 \vei \log\vei$ so the proposition follows with $m = F_{\eps*}^{[\vei]+n_2}m_S$.
\end{proof}
We are now in a position to complete the proof of
Theorem~\ref{thm:coupling_near_sink}. Let $s>0$ and
$\{I_j\}_{1\le j \le \lfloor s\eps^{-1/2}\rfloor}$ be as defined in
Sub-lemma~\ref{sublem:finding_good_bins}. Set
$\beta_\eps = \lfloor s\eps^{-1/2}\rfloor$ and for
$1\le j \le \beta_\eps$ define
$C_j = \{\plaque \in \plaqueSet: I_j \subset \pi_2
\supp\plaque\}$. Now by~\eqref{eq:height_of_rect_upper_bd}, for any
$\plaque \in \plaqueSet$ we have
$ |\pi_2 \supp\plaque|\le \frac32\patchheight\ve $ for $\patchheight$
sufficiently large. Since $d(I_j,I_{j'})\ge 2\patchheight\eps$ for
$j\ne j'$, it follows that the sets $C_j$ are disjoint. Thus for
$i\in \{1,2\}$ we have
\begin{equation}\label{eq:hat_H_i_conditioned_decomp}
  \sum_{j=1}^{\beta_\eps}p_{i,j}\widehat{\cH}_i|\{\plaque \in C_j\}+\left(1-\sum_{j=1}^{\beta_\eps}p_{i,j}\right)\widehat{\cH}_i\bigg|\bigg\{\plaque \nin \bigcup_{j=1}^{\beta_\eps} C_j\bigg\}\in [\widehat{\cH}_i],
\end{equation}
where $p_{i,j} = \lambda_{\widehat{\cH}_i}(\plaque \in C_j)\ge \frac12 p\eps^{1/2}$.

Next, observe that Proposition~\ref{prop:coupling_on_good_bins}
implies that for any $n'\ge \logstep_{4} \vei \log\vei$ there exist
proper patch families $\cD_{1,j}$, $\cD_{2,j}$ such that
\begin{equation}\label{eq:coupling_on_good_bins_for_families}
	F_{\eps*}^{n'}(\mu_{\widehat{\cH}_i|\{\plaque \in C_j\}}-\mu_{\widehat{\cH}_i|\{\plaque \in C_j\}}) =p'(\mu_{\cD_{1,j}} - \mu_{\cD_{2,j}}).
\end{equation}
Indeed,~\eqref{eq:coupling_on_good_bins_via_ref_measure} implies that
$F_{\eps*}^n(\mu_{\plaque^{(1)}}-\mu_{\plaque^{(2)}})=p'(\mu_{\cG_1}-\mu_{\cG_2})$. Let
$\mathcal{A}_{i,j}$ denote the index set of
$\widehat{\cH}_i|\{\plaque \in C_j\}$ and
$\plaque_{i,j}:\mathcal{A}_{i,j}\to \plaqueSet$ denote the map onto
$\plaqueSet$ associated with $\widehat{\cH}_i|\{\plaque \in
C_j\}$. Then writing
$\mu_{\widehat{\cH}_1|\{\plaque \in
  C_j\}}-\mu_{\widehat{\cH}_2|\{\plaque \in C_j\}}$ as a convex
combination of
$\mu_{\plaque_{1,j}(\alpha_1)}- \mu_{\plaque_{2,j}(\alpha_2)}$ over
$(\alpha_1,\alpha_2)\in \mathcal{A}_{1,j}\times \mathcal{A}_{2,j}$ and
applying Proposition~\ref{prop:coupling_on_good_bins} with
$\plaque^{(i)}=\plaque_{i,j}(\alpha_i)$ and $n=n'$
proves~\eqref{eq:coupling_on_good_bins_for_families}.

Choose $w \in (0,\frac12 p\eps^{1/2})$ such that $w\beta_\eps = \frac14 ps$. By combining~\eqref{eq:hat_H_i_conditioned_decomp} and~\eqref{eq:coupling_on_good_bins_for_families} we obtain that
\[
  F_{\eps*}^{n'}(\mu_{\widehat{\cH}_1}-\mu_{\widehat{\cH}_2})=\sum_{j=1}^{\beta_\eps}wp'(\mu_{\cD_{1,j}}
  -
  \mu_{\cD_{2,j}})+(1-w\beta_\eps)F_{\eps*}^{n'}(\mu_{\cE_1}-\mu_{\cE_2}), \]
where $\cE_i$ is a convex combination of
$\widehat{\cH}_i|\{\plaque \in C_j\}$ and
$\widehat{\cH}_i|\{\plaque \notin \cup_j C_j\}$.
%More precisely,
%\[ \widehat{\cH}_i = \sum_{j=1}^{m_\eps}\frac{p_{i,j} -
%  w}{1-w\beta_\eps}+\frac{1}{1-w\beta_\eps}\left(1-\sum_{j=1}^{\beta_\eps}p_{i,j}\right)\widehat{\cH}_i\bigg|\bigg\{\plaque
%  \nin \bigcup_{j=1}^{\beta_\eps} C_j\bigg\}. \]
Now any standard patch $\plaque$ in $\widehat{\cH}_i$ satisfies
$\cM(\plaque)\le \widehat{Q}$ so in particular this holds for any
standard patch in $\cE_i$. Thus by
Proposition~\ref{prop:geometric_drift}, for $\eps$ small enough and
$n'\ge \max\{\logstep_{1},\logstep_{4}\}\vei\log\vei$ we can choose a
proper patch family $\widetilde{\cE}_i\in [F^{n'}_{\eps*}\cE_i]$. It
follows that
\[ F^{n'}_{\eps*}(\mu_{\widehat{\cH}_1}-\mu_{\widehat{\cH}_2})=c'(\mu_{\widetilde{\cH}_1}-\mu_{\widetilde{\cH}_2}) \]
where $c' = 1 - p'(1- w\beta_\eps)=1-p'(1-\frac14 ps)$ and
\[ \widetilde{\cH}_i =
  \sum_{j=1}^{\beta_\eps}\frac{wp'}{c'}\cD_{i,j}+\frac{1-w\beta_\eps}{c'}\,
  \widetilde{\cE}_i. \] Note that $\widetilde{\cH}_i$ is proper since
it is a convex combination of proper patch families. Finally, recall
that $\widehat{\cH}_i \in [F_{\eps*}^{\clockN}{\cH_i}]$ so
$F^{\clockN+n'}_{\eps*}(\mu_{\cH_1}-\mu_{\cH_2})=c'(\mu_{\widetilde{\cH}_1}-\mu_{\widetilde{\cH}_2})$. The
theorem follows by choosing
$\logstep_{2} = 2(\clock+\max\{\logstep_{1},\logstep_{4}\})$ so that for
any $n\ge \logstep_{2}\vei\log\vei$ we can find $n'$ as above such
that $n = \clockN + n'$.

The above discussion at last allows to determine the value of
$\patchheight$ which had been fixed in
Section~\ref{sec:standard-patches}; we need to choose it large enough
so that Sub-lemma~\ref{sublem:finding_good_bins} (with $Q = 8B/3$, and
$B$ as in Proposition~\ref{prop:geometric_drift}) and
Proposition~\ref{prop:coupling_on_good_bins} hold.
%%% Local Variables:
%%% TeX-parse-self: t
%%% TeX-auto-save: t
%%% tex-main-file: "phe2"
%%% TeX-master: "phe2"
%%% mode: latex
%%% End:

\section{Conclusions}\label{sec:conclusions}
As mentioned in the introduction, the main purpose of this paper is to
illustrate a technique that can be used to obtain relatively sharp
bounds on the decay of correlations for systems of the
form~\eqref{eq:map} in the mostly expanding case.  Of course it is
tempting to ask what else can be done using this technique; for
instance, in~\cite{MR3556527} it is studied the case of multiple sinks
under the assumption that every sink is mostly contracting.  Arguments
parallel to those given in the paper would likely allow to obtain with
moderate effort similar results in the case of multiple sinks under
the assumption that every sink is mostly expanding.  However, we
believe that a much more interesting situation to study is the generic
case of multiple sinks, in which there may be some sinks that are
mostly expanding, while the others are mostly contracting.  We plan to
address this case in a follow-up paper.

Another natural question concerns the sharpness of our bounds on the
rate of decay of correlations.  We believe that the factor $\log\vei$
present in the bound~\eqref{e_lowerBoundRate-weaker} for $c_{\ve}$ in
our Main Theorem is artificial, and that a more efficient coupling
argument could provide a bound $c_{\ve}\ge C_{2}\ve$, provided that we
take $C_{1} = \ve^{-c_{\#}}$.  Potentially, this improved coupling
argument might also allow to improve the estimate on the rate of decay
in the multiple-sink scenario.  We also plan to work towards this in a
follow-up paper.

Finally, we would like to add a remark about the case in which the
averaged system has no zeros.  This is of course an extremely
interesting situation, and also related to the case in which
$\bar\omega$ is identically equal to $0$, which is relevant from the
point of view of statistical mechanics (as it would correspond to the
system at equilibrium, with no currents).  It would appear that in
this case there should be a unique physical measure enjoying
exponential decay of correlations with rate $c_{\ve} = \Const\ve^{2}$,
but substantially more work is needed before we can improve our
techniques to the extent of obtaining sharp results in this situation.
The missing ingredient –if we wanted to pursue the same strategy as in
this paper– would be a local limit theorem at timescales $\ve^{-2}$,
which appears unfeasible at the moment.
\appendix
\appendix
\section{Second derivative bounds}\label{sec:second-deriv-bounds}
Recall that $\chvar:\T\times\bR/(\vei\bZ) \to \T^2$ is a change of
variable given by $\chvar(x,y) = (x,\eps y)$.  Given
$q = (x,y)\in\T^{2}$ and a neighbourhood $V\ni q$ (sufficiently
small), for any $n > 0$, we say that a diffeomorphism
$\varphi_{n}:V\to\T^{2}$ is a \emph{local inverse of $F_{\ve}^{n}$ at
  $q$} if $F^{n}_{\ve}\circ\varphi_{n} = \textrm{Id}$.  This appendix
is dedicated to proving the following results about local inverses.
Recall the notation $\Upsilon_{n}^{\eps}$ introduced
in~\eqref{epsnStepCenter}.
% I think this paragraph is a useful reminder in case the reader gets
% confused.
Recall our conventions for the differential and Hessian operators
outlined in Section~\ref{sec:conventions}.
\begin{lem}\label{lem:dFeps_inv_conj_bd} There exists a constant $D>0$
  such that for all $n\ge 0$, all $q\in\T^{2}$, all local inverses
  $\varphi_{n}$ at $q$, and all $\eps>0$ sufficiently small:
  \begin{align*}
    \| d_{\chvar^{-1}(q)}(\chvar^{-1} \circ \varphi_n\circ \chvar)\|&\le D\,\Upsilon_n^\eps(\varphi_{n}(q)).
  \end{align*}
\end{lem}
\begin{lem}\label{lem:C2_bds} For any $T > 0$ let $n\le T\vei$,
  $q\in \T^{2}$, $\varphi_{n}$ be a local inverse at $q$ and $\ve$ be
  sufficiently small; we have
  \begin{enumerate}
    \item\label{item:log_jacob_C2_bd}
    $ \|d_{\chvar^{-1}(q)}(\log \det [d
    \varphi_{n}]\circ{\chvar})\|\le C_{T} $ and
    $\| H_{\chvar^{-1}(q)}(\log \det [d
    \varphi_{n}]\circ{\chvar})\|\le C_{T}$,
    \item\label{item:chvar_hessian_bd}
    $ \| H_{\chvar^{-1}(q)}(\chvar^{-1} \circ \varphi_n\circ
    \chvar)\|\le C_{T}.$
	\end{enumerate}
\end{lem}
We now consider $T > 0$ to be fixed throughout this appendix; we also
consider $q$ and $V\ni q$ to be fixed arbitrarily and a local inverse
$\varphi_{n}$ at $q$ also to be fixed arbitrarily.  To simplify
notation, throughout this appendix we write
$G_n^\eps = \varphi_{n}\circ \chvar$; observe that this function is
defined on $\chvar^{-1}V$.

We will also need to establish some auxiliary results (recall the definition of $\Gamma_n^\eps$ from~\eqref{epsnStepUnstable}):
\begin{lem}\label{lem:log_jacob_auxiliary} Let $n$ be as above and
  $\eps>0$ be sufficiently small.
  Then
	\begin{enumerate}
      \item\label{item:log_jacob_formula}
      $\log \det [d\varphi_{n}]\circ \chvar = \log \Upsilon_n^\eps
      \circ G_n^\eps - \log \Gamma_n^\eps \circ G_n^\eps.$
      \item\label{item:dlog_Upsilon_Gamma}
      $\|d(\log \Upsilon_n^\eps \circ G_n^\eps) \|_{\infty}\le C_{T}$
      and
      $\|d(\log \Gamma_n^\eps \circ G_n^\eps) \|_{\infty}\le C_{T}$,
      \item\label{item:Hlog_Upsilon_Gamma}
      $\|H(\log \Upsilon_n^\eps \circ G_n^\eps) \|_{\infty}\le C_{T}$
      and
      $\|H(\log \Gamma_n^\eps \circ G_n^\eps) \|_{\infty}\le C_{T}$.
	\end{enumerate}
\end{lem}
We also need the following lemma:
\begin{lem}\label{lem:Fn_inv_chvar_bds} We have that:
  \begin{enumerate}
    \item\label{item:Fn_inv_chvar_grad}
    $\| dG_n^\eps\|_{\infty} \le C_{T} (\lambda^{-n}+\eps)$,
    \item\label{item:Fn_inv_chvar_hessian}
    $\| HG_n^\eps\|_{\infty} \le C_{T} (\lambda^{-n}+\eps)$.
	\end{enumerate}
\end{lem}
Before proceeding any further, we find convenient, for the remainder
of this section, to lift the locally invertible map $F_{\ve}$ on
$\bT^{2}$ to a (globally) invertible map $\hat F_{\ve}$ on the
universal cover $\bR^{2}$.

Let $U = \varphi_{n}V$ and let $\hat U$ denote an arbitrary lift of
$U$ to $\bR^{2}$; for $0\le k\le n$, let
$\hat U_{k} = \hat F_{\ve}^{k}\hat U \subset \bR^{2}$ and denote with
$\hat V = \hat U_{n}$.  Similarly, we consider a lift of $\chvar$ to
$\hat\chvar:\bR^{2}\to\bR^{2}$ and we define for $0\le k\le n$ the
maps $\hat G_{k}^{\ve} = F^{-k}_{\ve}\circ\hat\chvar$.  In particular,
if $\bar q\in \hat\chvar\inv U_{n}$, then
$\hat G_{k}^{\ve}\bar q = F^{-k}_{\ve}\circ\chvar \bar q\in \hat
U_{n-k}$.  All functions defined on $\bT^{2}$ (\eg $\Upsilon_{k}$,
$\Gamma_{k}$, $\cdots$) will also be considered to be lifted to
$\bR^{2}$ (\eg $\hat\Upsilon_{k}$, $\hat\Gamma_{k}$, $\cdots$).

Working with the lifted system allows for a more compact formulation of
the computations; since we will only deal with the lifted system in
the remainder of this section, we will abuse notation and drop all hats
from our notation for lifts.

Let us first obtain an expression for $dF_\eps^{-n}$ on $V$.  Let
$q\in V$ and $p = F^{-n}_{\ve}q\in U$; of course
$d_{q}F_{\ve}^{-n} = [d_pF_{\ve}^{n}]\inv$. Recall now that
by~\eqref{epsnStepCenter}, we have:
\[ [d_pF_{\ve}^{n}]\inv (0,1) = \Upsilon^\ve_{n}(p)(s^\ve_{n}(p),1)\]
and by~\eqref{epsnStepUnstable}, we have
\[ [d_pF_{\ve}^{n}]\inv(1,\ve w^\ve_n(p)) = \Gamma_n^\ve(p)\inv (1,0). \]
By writing $(1,0) = (1,\ve w^\ve_n(p)) - \ve (0,w^\ve_n(p))$, it
follows that
\[ [d_pF_{\ve}^{n}]\inv(1,0)= (\Gamma_n^\ve(p)\inv - \ve w^\ve_n(p)\Upsilon^\ve_{n}(p)s^\ve_{n}(p), -\ve w^\ve_n(p)\Upsilon^\ve_{n}(p)). \]
Thus
\begin{equation}\label{Fn_inv}
  d_q F_\eps^{-n} = \begin{pmatrix}
    \Gamma_n^\ve(p)\inv - \ve w^\ve_n(p)\Upsilon^\ve_{n}(p)s^\ve_{n}(p) & \Upsilon^\ve_{n}(p)s^\ve_{n}(p) \\
    \qquad -\ve w^\ve_n(p)\Upsilon^\ve_{n}(p) & \Upsilon^\ve_{n}(p)
	\end{pmatrix}.
\end{equation}
We are now in a position to prove Lemma~\ref{lem:dFeps_inv_conj_bd}
and Lemma~\ref{lem:Fn_inv_chvar_bds}\ref{item:Fn_inv_chvar_grad}.
\begin{proof}[Proof of Lemma~\ref{lem:dFeps_inv_conj_bd} and Lemma~\ref{lem:Fn_inv_chvar_bds}\ref{item:Fn_inv_chvar_grad}]
  Given $q\in V$, let $\bar q = \chvar\inv q$ and
  $p = F^{-n}_{\ve} q = G_n^\eps \bar q$. Then by~\eqref{Fn_inv},
\begin{align}
  d_{\bar q}G_n^\eps &=
                       d_qF_{\ve}^{-n}\begin{pmatrix} 1 & 0\\ 0 & \eps
			\end{pmatrix} \nonumber \\ &= \begin{pmatrix}
\Gamma_n^\ve(p)\inv - \ve w^\ve_n(p)\Upsilon^\ve_{n}(p)s^\ve_{n}(p) &
\ve\Upsilon^\ve_{n}(p)s^\ve_{n}(p) \\ \qquad -\eps
w^\ve_n(p)\Upsilon^\ve_{n}(p) & \ve\Upsilon^\ve_{n}(p)
			\end{pmatrix} \nonumber\\ &= \begin{pmatrix}
\Gamma_n^\ve(p)^{-1} & 0 \\ 0 & 0
			\end{pmatrix} +
\eps\Upsilon^\ve_{n}(p)M_n(p), \label{eq:dG_n_matrix}
\intertext{where} M_n(p)&=\begin{pmatrix} - w^\ve_n(p)s^\ve_{n}(p) &
s^\ve_{n}(p) \\ \qquad -w^\ve_n(p) & 1
\end{pmatrix}. \nonumber \intertext{It follows that} d_{\bar q}(\chvar\inv \circ G_n^\eps) &= \begin{pmatrix}
  \Gamma_n^\ve(p)^{-1} & 0 \\ 0 & 0
						\end{pmatrix} +
\Upsilon^\ve_{n}(p)\begin{pmatrix} \eps M^{11}_n(p) & \eps M^{12}_n(p)
\\ M^{21}_n(p) & M^{22}_n(p)
						\end{pmatrix}. \label{eq:dFn_inv_conj_matrix}
\end{align}
Since $|s_n^\eps|\le \centreConeConst$ and
$|w_n^\eps|\le \unstableConeConst$, the entries of $M_n$ are uniformly
bounded. By~\eqref{NstepExp} and~\eqref{UnstableExp}, for $ \ve $
sufficiently small we have
$\Gamma_n^\ve(p)\inv\le \Upsilon^\ve_{n}(p)$ for all $n\ge 0$. This
completes the proof of Lemma~\ref{lem:dFeps_inv_conj_bd} since
$d(\chvar\inv\circ F^{-n}_{\ve}\circ\chvar) = d(\chvar\inv\circ
\varphi_{n}\circ\chvar)$.

For $n\le T \eps\inv$, note that we also have
$\Gamma_n^\ve(p)\inv\le C_{T}\lambda^{-n}$ and
$\Upsilon_n^\eps(p)\le C_{T}$; hence~\eqref{eq:dG_n_matrix} concludes
the proof of
Lemma~\ref{lem:Fn_inv_chvar_bds}\ref{item:Fn_inv_chvar_grad}.
\end{proof}
\begin{proof}[Proof of Lemma~\ref{lem:log_jacob_auxiliary}(a) and (b)]
  We start by obtaining an expression for $\log \det dF_\eps^{-n}$.
  As before, given $q\in V$, we let $p = F^{-n}_{\ve} q\in U$; by
  construction (or inspecting~\eqref{Fn_inv}):
  \begin{equation*}
    \det d_q F_\eps^{-n} = \Upsilon^\ve_{n}(p)/\Gamma_n^\ve(p).
  \end{equation*}
  In particular
  \begin{align*}
    \log \det [dF_\eps^{-n}]\circ \chvar &= \log\Upsilon_n^\eps \circ
                                           G_n^\eps -
                                           \log\Gamma_n^\eps \circ
                                           G_n^\eps,
  \end{align*}
  which proves
  Lemma~\ref{lem:log_jacob_auxiliary}\ref{item:log_jacob_formula}.
  Now, let $p_{k} = F^{k}_{\ve} p\in U_{k}$; by~\eqref{NstepExp}, we
  have that
  \[ \Upsilon_n^\ve(p) = \prod_{k=0}^{n-1} \bigg(1 + \ve
    \Big(\partial_\theta\omega(p_{k})+\partial_x\omega(p_{k})s^\ve_{n-k}(p_{k})\Big)\bigg)^{-1}. \]
  Moreover, for any $0\le l \le n$, if $q\in U_{l}$, let us define
  $\tilde w^{\ve}_{l}(q) = w^{\ve}_{l}(F^{-l}_{\ve}q)$;
  by~\eqref{UnstableExp}, we have
  \[ \Gamma_n^\ve(p)=\Lambda_n (p) \prod_{k=0}^{n-1}\left(1+\ve
      \frac{\partial_\theta f(p_{k})}{\partial_x f(p_{k})} \tilde
      w^\ve_{k}(p_{k})\right),\] where
  $\Lambda_n(p) = \prod_{k=0}^{n-1} \partial_x f(p_{k})$.  Taking
  logarithms of the above expressions, we can rewrite them as:
  \begin{align}
    \log\Upsilon_n^\eps \circ G_n^\eps &= \sum_{k =1}^{n}A_k\circ G_k^\eps, &\log\Gamma_n^\eps \circ G_n^\eps &= \sum_{k =1}^{n}B_k\circ G_k^\eps \label{eq:log-Gamma-Upsilon}
  \end{align}
  where:
  \begin{align*}
	A_k &= -\log \left(1+\ve({\partial_{\theta}\omega}+{\partial_{x}\omega}\,s^{\ve}_{k})\right),\\
    B_k &= \log \partial_{x}f%
          + \log \left(1+\ve\frac{\partial_{\theta}f}{\partial_{x}f}\tilde w^{\ve}_{n-k}\right).
  \end{align*}
  In order to complete the proof of
  part~\ref{item:dlog_Upsilon_Gamma}, we need the following sub-lemma:
  \begin{sublem}\label{sublem:det_grad_bds}The following estimates
    hold for $0\le l \le n \le T \vei$:
	\begin{enumerate}
    \item $ \|d\tilde w_l^\eps \|_{\infty}\le C_{T} $ and $
      \|ds_l^\eps \|_{\infty}\le C_{T} l$,
    \item $ \|dA_l \|_{\infty} \le C_{T} $ and
      $ \|dB_l \|_{\infty} \le C_{T}$.
	\end{enumerate}
  \end{sublem}
  \begin{proof}
	We begin by bounding $\| d\tilde{w}_l^\eps
    \|_{\infty}$. By~\eqref{RecurUnstable}, for any
    $0 < l \le n$, $q\in U_{l}$ and $-l\le j \le n-l$, let
    $q_{j} = F^{j}_{\ve}q\in U_{l+j}$.  We can write
	\begin{equation*}
      \tilde{w}_l^\eps(q) = \Xi^{+}\big(g_{l}(q)\big),\quad \text{ where }g_{l}(q) = \big(F_\eps\inv(q),\, \tilde{w}^\eps_{l-1}(F_\eps\inv(q))\big).
	\end{equation*}
	Differentiating the above expression we gather:

    \begin{align}
      d_{q}\tilde w_{l}^{\ve}%
      &= \partial_{1}\Xi^{+}\big(g_{l}(q)\big)d_q F_\eps\inv +\nonumber %\\
      \partial_{2}\Xi^{+}\big(g_{l}(q)\big)\left[ d_{q_{-1}}\tilde
      w_{l-1}^{\ve}\right] d_{q}F\inv_{\ve}
      \intertext{and iterating:}
      d_{q}\tilde w_{l}^{\ve}%
      &= \sum_{j = 0}^{l-1}
		\Bigg[\prod_{i=0}^{j-1} \partial_2\Xi^{+}\big(g_{l-i}(q_{-i})\big)\Bigg]
		\partial_{1}\Xi^{+}(g_{l-j}\big(q_{-j})\big)d_{q}F_{\ve}^{-(j+1)}.\label{eq:dw_expansion} %\\
	\end{align}
    Using~\eqref{eq:def_Xi}, since $\tilde w_{l}^{\ve}$ is uniformly
	bounded, we obtain that
	\begin{equation}
      \begin{aligned}\label{eq:Xi_plus_derivative_bds}
		\|\partial_{1}\Xi^{+}\circ g_{l-i}\|_{\infty} &\le \Const,&%\\
		\|\partial_{2}\Xi^{+}\circ g_{l-i}\|_{\infty} &\le (1+\Const\ve)\lambda^{-1},\\
		\|H\Xi^{+}\circ g_{l-i}\|_{\infty} &\le \Const.
      \end{aligned}
	\end{equation}
	(We shall use the bound on $H\Xi^{+}\circ g_{l-i}$ later.)  By
    Lemma~\ref{lem:lyapunov}, we also know that
    $\|dF_{\ve}^{-k}\|_{\infty}\le\Const (1+\Const\ve)^{k}$;
    collecting the above estimates, by the arbitrariness of $q$ we
    conclude:
	\begin{align*}
      \|d\tilde w_{l}^{\ve}\|_{\infty}%
      &\le \Const\sum_{j =0}^{l-1}\lambda^{-j}(1+\Const\ve)^{2j+1} = C_{T}.
	\end{align*}
    The computations for $s_{l}^{\ve}$ are similar: recall
    from~\eqref{EpsCNstepRecur} that we can write
	\begin{align*}
      s_{l}^{\ve}(q) &= \Xi^{-}\big(h_l(q)\big),\quad \text{ where }h_l(q)=\big(q, s_{l-1}^{\ve}(F_{\ve}(q))\big).
	\end{align*}
	Differentiating the above expression, and iterating, we obtain
	\begin{align}
      d_{q} s_{l}^{\ve}
      &= \partial_{1}\Xi^{-}\big(h_l(q)\big) + \partial_{2}\Xi^{-}\big(h_l(q)\big)\left[d_{q_{1}}s^{\ve}_{l-1}\right]d_{q}F_{\ve}\nonumber\\
      &= \sum_{j = 0}^{l-1}
		\Bigg[\prod_{i=0}^{j-1} \partial_2\Xi^{-}\big(h_{l-i}(q_i)\big)\Bigg]
		\partial_{1}\Xi^{-}(h_{l-j}\big(q_j)\big)d_{q}F_{\ve}^{j}.\label{eq:ds_expansion}
	\end{align}

	Using the definition of $\Xi^{-}$ and the fact that $s_l^\eps$ is uniformly bounded we gather:
	\begin{equation}
      \begin{aligned}\label{eq:Xi_minus_derivative_bds}
		\|\partial_{1}\Xi^{-}\circ h_{l-i}\|_{\infty} &\le \Const,&%\\
		\|\partial_{2}\Xi^{-}(h_{l-i}(q_i))\| &\le (1+\Const\ve)\partial_{x}f(q_i)\inv,\\
		\|H\Xi^{-}\circ h_{l-i}\|_{\infty} &\le \Const.
      \end{aligned}
	\end{equation}
	(We shall use the bound on $H\Xi^{-}\circ h_{l-i}$ later.) 	By Lemma~\ref{lem:lyapunov}, we have
    $\|\Lambda_{j}^{-1}\cdot d F^{j}_{\ve}\|\le \Const
    (1+\Const\ve)^{j}$; we conclude by the arbitrarity of $q$ that
 	\begin{align*}
      \|d s_{l}^{\ve}\|_{\infty} &\le \Const\sum_{j = 0}^{l-1}(1+\Const
                                   \ve)^{j}\le \Const l(1+\Const\ve)^{l}\le C_{T}l.
	\end{align*}
	This completes the proof of (a). By the definition of $A_l$ and
    $B_l$ we have
	\begin{align*}
      \|dA_l\|_{\infty}&\le \Const\eps(1+\|ds^\eps_l \|_{\infty})\le C_{T}
                    \intertext{and}
      \|dB_l\|_{\infty}&\le \Const + \Const\eps(1+\|d_q \tilde{w}^\eps_{n-l})\|_{\infty}\le C_{T},
	\end{align*}
	which proves part (b).
  \end{proof}
  We now complete the proof of
  Lemma~\ref{lem:log_jacob_auxiliary}\ref{item:dlog_Upsilon_Gamma}. Combining~\eqref{eq:log-Gamma-Upsilon}
  with Sub-lemma~\ref{sublem:det_grad_bds}(b) and
  Lemma~\ref{lem:Fn_inv_chvar_bds}\ref{item:Fn_inv_chvar_grad} yields
  that
  \[ \|d(\log \Upsilon_n^\eps \circ G_n^\eps) \|_{\infty}\le
    \sum_{k=1}^n \|dA_k \|_{\infty}\|dG_k^\eps \|_{\infty} \le
    \sum_{k=1}^n \Const(\lambda^{-n} + \eps)\le C_{T}, \]

  where we have used the assumption $n\le T\eps\inv$. By exactly the
  same argument,
  $\|d(\log \Gamma_n^\eps \circ G_n^\eps) \|_{\infty}\le C_{T}$.
\end{proof}

\begin{proof}[Proof of
  Lemma~\ref{lem:Fn_inv_chvar_bds}\ref{item:Fn_inv_chvar_hessian} and
  Lemma~\ref{lem:C2_bds}\ref{item:chvar_hessian_bd}]
  We proceed by using~\eqref{eq:dG_n_matrix}
  and~\eqref{eq:dFn_inv_conj_matrix} to bound the derivative of the
  entries of $dG_n^\eps$ and $d(\chvar\inv\circ G_n^\eps)$. Since
  $\Gamma_n^\eps(p)^{-1}\le \Const \lambda^{-n}$ and
  $\Upsilon_n^\eps(p)\le C_{T}$,
  Lemma~\ref{lem:log_jacob_auxiliary}\ref{item:dlog_Upsilon_Gamma}
  implies that
\begin{align*}
  \|d([\Gamma_n^\eps]^{-1}\circ G_n^\eps) \|_{\infty} = \|[\Gamma_n^\eps]^{-1}\circ G_n^\eps \ d(\log\Gamma_n^\eps\circ G_n^\eps) \|_{\infty} \le C_{T}\lambda^{-n}
\end{align*}
and $\|\Upsilon^\eps_n \circ G_n^\eps \|_{\cC^1} \le \Const$. Hence it suffices to show that
\begin{align*}
	\|M_n^{1j}\circ G_n^\eps\|_{\cC^1}&\le C_{T} \eps\inv(\lambda^{-n}+\eps), & \|M_n^{2j}\circ G_n^\eps \|_{\cC^1}\le C_{T}
\end{align*}
for $j=1,2$.

By Sub-lemma~\ref{sublem:det_grad_bds}(a),
\begin{align*}
	\|d(w_n^\eps \circ G_n^\eps) \|_{\infty} = \|d(\tilde{w}_n^\eps \circ \chvar) \|_{\infty}\le \|d\tilde{w}_n^\eps\|_{\infty} \|d\chvar\|_{\infty}\le C_{T}
\end{align*}
so
$\|M^{21}_n\circ G_n^\eps \|_{\cC^1}=\|w_n^\eps \circ G_n^\eps
\|_{\cC^1}\le C_{T}$. Moreover, by
Lemma~\ref{lem:Fn_inv_chvar_bds}\ref{item:Fn_inv_chvar_grad} and
Sub-lemma~\ref{sublem:det_grad_bds}(a),
\begin{align*}
  \|d(M^{12}_n \circ G_n^\eps) \|_{\infty}=\|d(s_n^\eps \circ G_n^\eps) \|_{\infty}\le C_{T} \eps\inv(\lambda^{-n}+\eps).
\end{align*}
Since $M^{11}_n = s^\eps_n w^\eps_n$, the desired bound on
$ \|M^{11}_n \circ G_n^\eps\|_{\cC^1}$ follows from the product rule.
\end{proof}
In order to conclude the proofs in this appendix, we will make
repeated use of the following bound on the Hessian of composite
functions:
\begin{lem}\label{lem:hessian_chain_rule} Let $\mathcal U\subseteq \R^m$ and
  $\mathcal V\subseteq \R^k$ be open sets and let
  $f:\mathcal U \to \mathcal V$ and $g:\mathcal V\to \R$ be $\cC^2$
  functions. Then for all $p\in \mathcal U$,
  \begin{align*}
    \|H_p(g \circ f) \|\le \|H_{f(p)}g \|\|d_p f\|^2 +k\|d_{f(p)}g
    \|\|H_p f \|.
  \end{align*}
\end{lem}
\begin{proof}
	Observe that for $i,j = 1,\cdots,m$:
	\begin{align*}
		\partial_{ij}(g\circ f)(p) &= \sum_{s,t=1}^k \partial_j f_t(p) \partial_{st}g(f(p)) \partial_i f_s(p) + \sum_{s=1}^k \partial_s g(f(p)) \partial_{ij} f_s(p) \\
		&= \big[(d_{p}f)^T H_{f(p)}g \, d_{p}f\big]_{ij} + \sum_{s=1}^k \partial_s g(f(p))[H_p f_s]_{ij}.
	\end{align*}
	Hence
    \begin{align*}
      \|H_p(g\circ f)\| \le \|d_{p}f\| \|H_{f(p)}g \| \|d_{p}f
      \|+\|d_{f(p)}g\|\sum_{s=1}^k \| H_p f_s\|,
    \end{align*}
    and the result
    follows.
\end{proof}
\begin{proof}[Proof of
  Lemma~\ref{lem:log_jacob_auxiliary}\ref{item:Hlog_Upsilon_Gamma}] We
  proceed by using~\eqref{eq:log-Gamma-Upsilon}. Note that we can
  write
	\begin{align*}
      A_k(q)&= \alpha(q, \eps s^\eps_k(q)), & B_k(q) = \beta(q, \ve\tilde{w}^\eps_{n-k}(q)),
	\end{align*}
	where $\alpha$ and $\beta$ are $\cC^2$ functions with norms that
    are uniformly bounded in $\ve$. By
    Lemma~\ref{lem:hessian_chain_rule}, it follows that
	\begin{align*}
		\|H(A_k \circ G_k^\eps)\|_{\infty} \le &(1+\eps\|ds_k^\eps\|_{\infty})^2 \|dG_k^\eps \|_{\infty}^2 \|H\alpha\|_{\infty}
		\\ & \quad +3\|d\alpha\|_{\infty}(\|HG_k^\eps \|_{\infty} + \eps \|H({s}_k^\eps \circ G_k^\eps)\|_{\infty})
	\end{align*}
	Hence by applying Lemma~\ref{lem:Fn_inv_chvar_bds} in combination
    with Sub-lemma~\ref{sublem:det_grad_bds}(a), we obtain that
	\begin{align*}
		\| H(A_k \circ G_k^\eps)\|_\infty &\le C_{T}(\lambda^{-k}+\eps)^2 + C_{T} (\lambda^{-k}+\eps)+\Const \eps\| H(s_k^\eps \circ G_k^\eps)\|_\infty
		\intertext{Similarly, we have}
		\| H(B_k \circ G_k^\eps)\|_\infty &\le C_{T}(\lambda^{-k}+\eps)^2 + C_{T} (\lambda^{-k}+\eps)+\Const \eps\| H(\tilde{w}_{n-k}^\eps \circ G_k^\eps)\|_\infty.
	\end{align*}
	Since $\sum_{k=1}^n (\lambda^{-k}+\eps)\le C_{T}$, it follows that
    it suffices to prove that
	\begin{align*}
		\| H(s_k^\eps \circ G_k^\eps)\|_\infty&\le C_{T}, & \| H(\tilde{w}_{n-k}^\eps \circ G_k^\eps)\|_\infty\le C_{T}
	\end{align*}
	for $1\le k\le n$.

    Note that $F_\eps^i \circ G_k^\eps = G_{k-i}^\eps$ for $i\le k$ so
    by using~\eqref{eq:ds_expansion} with $l=k$, we obtain that
\begin{align*}
	d (s_k^\eps \circ G_k^\eps) =\sum_{j=0}^{k-1}u_jV_jdG_{k-j}^\eps,
\end{align*}
where
\begin{align*}
	u_j &= \prod_{i=0}^{j-1} \partial_2\Xi^{-}\circ h_{k-i}\circ G_{k-i}^\eps, & V_j = \partial_1\Xi^{-}\circ h_{k-j}\circ G_{k-j}^\eps.
\end{align*}
Recall that $h_m(q) = \big(q,s_{m-1}^\eps(F_\eps(q))\big)$ for $1\le m\le n$. By Sub-lemma~\ref{sublem:det_grad_bds}(a) and Lemma~\ref{lem:Fn_inv_chvar_bds} it follows that
\begin{align*}
	\|d(h_m \circ G_m^\eps) \|_{\infty} &\le (1+\Const \|ds_{m-1}^\eps\|_{\infty})\|dG_m^\eps\|_{\infty}\le C_{T} m(\lambda^{-m}+\eps)\le C_{T}.
\end{align*}
Hence using~\eqref{eq:Xi_plus_derivative_bds} yields that
\begin{equation}\label{eq:dv_j-bound}
	\begin{aligned}
		\| d v_j \|_{\infty} &\le j\max_m\|\partial_2\Xi^{-}\circ h_m\|_{\infty}^{j-1}\max_m\|H\Xi^{-}\circ h_m\|_{\infty} \|d(h_m\circ G^\eps_m)\|_{\infty}\\
		&\le C_{T} j(1+C_{T}\ve)^{j-1}\lambda^{-(j-1)}\le C_{T}
	\end{aligned}
\end{equation}
and $\| V_j \|_{\cC^1}\le C_{T}$. Thus by Lemma~\ref{lem:Fn_inv_chvar_bds}, it follows that
\begin{equation}
\begin{aligned}\label{eq:Hs_k-eps_final_calc}
	\|H(s_k^\eps \circ G_k^\eps) \|_{\infty}&=\|d(d(s_k^\eps \circ G_k^\eps)^T) \|_{\infty} \le C_{T} \sum_{j=0}^{k-1}\|u_j\|_{\cC^1}\|V_j\|_{\cC^1}\|dG_{k-j}^\eps\|_{\cC^1}\\
	&\le C_{T}\sum_{j=0}^{k-1}(\lambda^{-(k-j)}+\eps)\le C_{T}.
\end{aligned}
\end{equation}

	It remains to show that $\| H(\tilde{w}_{n-k}^\eps \circ G_k^\eps)\|_\infty\le C_{T}$. Since $F_\eps^{-i}\circ G_k^\eps = G_{i+k}^\eps$, applying~\eqref{eq:dw_expansion} with $l=n-k$ yields that
		\begin{align*}
		d(\tilde w_{n-k}^{\ve}\circ G_k^\eps)%
		&= \sum_{j = 0}^{n-k-1}
		\Bigg[\prod_{i=0}^{j-1} \partial_2\Xi^{+}\circ g_{n-k-i}\circ G_{i+k}^\eps\Bigg]
		\partial_1\Xi^{+}\circ g_{n-k-j}\circ G_{j+k}^\eps \ dG_{j+k+1}^\eps.
	\end{align*}
	Recall that
    $g_m(q)=\big(F_\eps\inv(q),\,
    \tilde{w}^\eps_{m-1}(F_\eps\inv(q))\big)$ so
    $\|dg_m \|_\infty\le C_{T}$ for $1\le m\le n$ by
    Lemma~\ref{sublem:det_grad_bds}(a). Hence
    by~\eqref{eq:Xi_plus_derivative_bds} and calculations similar
    to~\eqref{eq:dv_j-bound} and~\eqref{eq:Hs_k-eps_final_calc}, it
    follows that
    $\| H(\tilde{w}_{n-k}^\eps \circ G_k^\eps)\|_\infty\le C_{T}$, as
    required.
\end{proof}

\printnomenclature
\bibliographystyle{abbrv}
\bibliography{rw}
%%
%\newpage
%\listoftodos\relax

\end{document}